\documentclass[10pt,a4paper]{article}

\usepackage[english]{babel}
\usepackage[utf8]{inputenc}
\usepackage{authblk}
\usepackage{appendix}

\usepackage[T1]{fontenc}

\usepackage[absolute]{textpos}

\usepackage[vlined,ruled]{algorithm2e}
\usepackage[left=3cm,right=3cm,top=2cm,bottom=2cm]{geometry}
\SetKwInOut{Input}{input}
\SetKwInOut{Output}{output}
\SetKwComment{Comment}{}{}

\SetCommentSty{mycmtsty}

\usepackage[vlined,ruled]{algorithm2e}
\usepackage{graphicx}
\usepackage{subfigure}
\usepackage{amsmath, amsthm}
\usepackage{hyperref}
\usepackage{amsfonts}
\usepackage{float}
\usepackage{pstricks,pst-plot,pst-text,pst-tree,pst-eps,pst-fill,pst-node,pst-math}
\usepackage[justification = centering]{caption}
\usepackage{tikz}
\usepackage{setspace}
\usepackage{enumitem}
\usepackage{esint}
\usepackage{empheq}
\usepackage{mathtools, cases}
\usepackage{nccmath}

\usepackage{titlesec}
\titleformat{\chapter}[display]   
{\normalfont\huge\bfseries}{\chaptertitlename\ \thechapter}{20pt}{\Huge}   
\titlespacing*{\chapter}{0pt}{-30pt}{30pt}

\usepackage{fancyvrb}
\VerbatimFootnotes 

\usepackage{amsthm}
\usepackage{subfigure}
\usepackage{appendix}
\usepackage{bigints}

\usepackage{mathtools}
\mathtoolsset{showonlyrefs,showmanualtags}

\newtheorem{theoreme}{Theorem}

\newtheorem{lemme}{Lemma}
\newtheorem{remark}{Remark}
\newtheorem{prop}{Proposition}
\newtheorem{corol}{Corollary}

\begin{document}

\title{\Huge Linear elliptic homogenization for a class of highly oscillating non-periodic potentials}

\date{}



\author[,1]{R\'emi Goudey\thanks{Corresponding author: \texttt{remi.goudey@enpc.fr}}}
\author[1]{Claude Le Bris}

\affil[1]{CERMICS, \'Ecole des Ponts and MATHERIALS project-team, INRIA,\qquad \qquad \quad 6 \& 8, avenue Blaise Pascal, 77455 Marne-La-Vall\'ee Cedex 2, FRANCE.}

\maketitle


\begin{abstract}
    We consider an homogenization problem for the second order elliptic equation  $- \Delta  u^{\varepsilon} + \dfrac{1}{\varepsilon} V(./\varepsilon) u^{\varepsilon} + \nu u^{\varepsilon} =f$ when the highly oscillatory potential $V$ belongs to a particular class of non-periodic potentials. We show the existence of an adapted corrector and prove the convergence of $u^{\varepsilon}$ to its homogenized limit.  
\end{abstract}

\section{Introduction}

Our purpose is to homogenize the stationary Schr\"odinger equation : 
\begin{equation}
\label{homog_eps_schro}
    \left\{
    \begin{array}{cc}
      - \Delta u^{\varepsilon} + \dfrac{1}{\varepsilon}V(./\varepsilon) u^{\varepsilon} + \nu u^{\varepsilon} = f   & \text{on }  \Omega, \\
      u^{\varepsilon} = 0   & \text{on } \partial \Omega, 
    \end{array}
    \right.
\end{equation}
where $\Omega$ is a bounded domain of $\mathbb{R}^d$ ($d\geq 1)$, $f$ is a function in $L^2(\Omega)$, $\varepsilon>0$ is a small scale parameter, $V\in L^{\infty}(\mathbb{R}^d)$ is an highly oscillatory non-periodic potential that models a perturbed periodic geometry and $\nu \in \mathbb{R}^d$ is a fixed scalar. The potential V is assumed to have a vanishing "average" in the following sense~: 
\begin{equation}
\label{average_V_schro}
    \lim_{\varepsilon \to 0} V(./\varepsilon) = 0 \quad \text{ in } L^{\infty}(\mathbb{R}^d)-\star,
\end{equation} 
which is a necessary assumption to expect the convergence of $u^{\varepsilon}$ to a non trivial function $u^*$ when $\varepsilon$ converges to 0 due to the exploding term $\dfrac{1}{\varepsilon}V(./\varepsilon)$ in~\eqref{homog_eps_schro}. In order to avoid some technical details, we also assume that $\Omega$ is sufficiently regular, say $\mathcal{C}^1$.

When the potential $V=V_{per}$ is periodic, the homogenization problem \eqref{homog_eps_schro} is well known, see for instance \cite[Chapter 1, Section 12]{bensoussan2011asymptotic}.
The behavior of $u^{\varepsilon}$ is then described using a corrector $w_{per}$, that is the periodic solution (unique up to the addition of an irrelevant constant) to
\begin{equation}
\label{corrector_per_schro}
    \Delta w_{per} = V_{per} \quad \text{on } \mathbb{R}^d.
\end{equation}
More precisely, $u^{\varepsilon}$ converges, strongly in $L^2(\Omega)$ and weakly in $H^1(\Omega)$, to the unique solution $u^*$ in $H^1(\Omega)$ to 
\begin{equation}
    \left\{
    \begin{array}{cc}
       -\Delta u^* + (V_{per})^* u^* + \nu u^* = f  &  \text{on } \Omega,\\
        u^* = 0 & \text{on } \partial \Omega,
    \end{array}
   \right. 
\end{equation}
where the homogenized potential $(V_{per})^*$ is the scalar constant defined by
\begin{equation}
\label{eq_homog_per_schro}
    (V_{per})^* = \langle V_{per} w_{per} \rangle = - \langle |\nabla w_{per}|^2 \rangle,
\end{equation}
the rightmost equality being a consequence of the virial theorem applied to the setting of \eqref{corrector_per_schro}, in that case a simple application of the Green formula. 

We note that the well-posedness of \eqref{homog_eps_schro} requires an additional assumption. In \cite{bensoussan2011asymptotic}, the homogenization of~\eqref{homog_eps_schro} is performed under the (sufficient) assumption
\begin{equation}
\label{Schro_condition_suff_per}
     \mu_1+\langle V_{per} w_{per} \rangle + \nu > 0,
\end{equation}
where $\mu_1$ is the first eigenvalue of $-\Delta$ on $H^1_0(\Omega)$. The existence and uniqueness of $u^{\varepsilon}$, for $\varepsilon$ sufficiently small, is then a consequence  of the convergence of $\lambda_1^{\varepsilon}$, the first eigenvalue of the operator $-\Delta + \dfrac{1}{\varepsilon} V(./\varepsilon) + \nu$ with homogeneous Dirichlet boundary conditions, to $\mu_1 + \langle w_{per} V_{per} \rangle + \nu$. One can remark that assumption \eqref{Schro_condition_suff_per} only depends on the properties of the potential $V_{per}$ since $w_{per}$ is solution to \eqref{corrector_per_schro}.

The behavior of $u^{\varepsilon}$ can also be made explicit using the periodic corrector. The sequence of remainders
\begin{equation}
    R^{\varepsilon} = u^{\varepsilon} -u^* - \varepsilon u^* w_{per}(./\varepsilon)
\end{equation}
is shown to strongly converge to 0 in $H^1(\Omega)$. The convergence of the eigenvalues of $-\Delta + \dfrac{1}{\varepsilon}V(./\varepsilon) +\nu$ is also studied in \cite{zhang2021estimates, cances2021second}. Respectively denoting by $\lambda_l^{\varepsilon}$ and $\lambda^*_l$ the lowest $l^{th}$ eigenvalue (counting multiplicities) of $-\Delta + \dfrac{1}{\varepsilon}V(./\varepsilon) + \nu $ and $-\Delta + \langle w_{per} V_{per} \rangle + \nu$, the results of \cite[Theorem 1.4]{zhang2021estimates} show the existence of a constant $C>0$ such that $\displaystyle |\lambda_l^{\varepsilon} - \lambda^*_l| \leq C |\lambda_l^*|^{3/2} \varepsilon$. All these results are essentially established using two specific properties of the corrector $w_{per}$ : the strict sublinearity at infinity, that is the uniform convergence to 0 of $\varepsilon w_{per}(./\varepsilon)$, and the weak convergence of the periodic function $|\nabla w_{per}(./\varepsilon)|^2$ to its average.
For completeness, let us also mention that both elliptic and parabolic variants of \eqref{homog_eps_schro} have also been studied in \cite{allaire2005homogenization},\cite[Chapter 1, Section5]{lions1981some} and were alternatively considered in the case where the potential $V$ is random and stationary, see for instance  \cite{bal2010homogenization, bal2015limiting, iftimie2008homogenization}.

The aim of the present contribution is to extend the results of the elliptic periodic case to an elliptic deterministic non-periodic setting that models a periodic geometry perturbed by a certain class of so-called defects which we make precise in the next section.

\subsection{The non-periodic case : mathematical setting and assumptions}

\label{Section_intro_setting_schro}

Throughout this work, we denote by $B_R$ the ball of radius $R>0$ centered at the origin and by $B_R(x)$ the ball of radius $R$ and center $x \in \mathbb{R}^d$, by $|A|$ the volume of any measurable subset $A\subset \mathbb{R}^d$ and by $A/\varepsilon = \left\{\dfrac{x}{\varepsilon} \ \middle| \ x \in A\right\}$ the scaling of $A$ by $\varepsilon>0$. In addition, for a normed vector space $(X,\|.\|_X)$ and a $d$-dimensional vector $f \in X^d$, we use the notation $\|f\|_{X} \equiv \displaystyle \sum_{i=1}^d\|f_i \|_{X}$. 

The class of potentials $V$ we consider in this paper consists of those that read as
\begin{equation}
\label{potential_form_schro}
    V(x) = g_{per} + \sum_{k \in \mathbb{Z}^d} \varphi(x-k-Z_k),
\end{equation}
where $g_{per}$ is a $Q$-periodic function (where $Q = ]0,1[^d$ denotes the unit cell of $\mathbb{R}^d$), $\varphi$ belongs to $\mathcal{D}(\mathbb{R}^d)$, the space of smooth compactly supported functions, and $Z = (Z_k)_{k\in \mathbb{Z}^d}$ is a vector-valued sequence that satisfies 
\begin{equation}
\label{hyp_borne_Z_schro}
   Z \in \left(l^{\infty}(\mathbb{Z}^d)\right)^d.
\end{equation}
In a certain sense, the potential $V$ is a perturbation of the periodic potential 
\begin{equation}
\label{periodic_potential_schro}
    \displaystyle V_{per} = g_{per} + \sum_{k \in \mathbb{Z}^d} \varphi(x-k)
\end{equation}
 by the "defect"
 \begin{equation}
 \label{perturbation_potential_schro}
     \Tilde{V} = \displaystyle \sum_{k \in \mathbb{Z}^d}\left(\varphi(x-k-Z_k) - \varphi(x-k)\right).
 \end{equation}
A simple calculation shows that assumption \eqref{average_V_schro} satisfied by $V$ is actually equivalent to
\begin{equation}
\label{Schro_condition_moyenne2}
    \displaystyle \int_{\mathbb{R}^d} \varphi = - \langle g_{per} \rangle,
\end{equation}
where $\langle g_{per} \rangle$ is the average of the $Q$-periodic function $g_{per}$. In this work, we additionally assume that 
\begin{equation}
\label{hyp_holder_continuite_schro}
   g_{per} \in \mathcal{C}^{0,\alpha}(\mathbb{R}^d), \text{ for some } \alpha \in ]0,1[, 
\end{equation}
where $\mathcal{C}^{0,\alpha}(\mathbb{R}^d)$ denotes the space of $\alpha$-H\"older continuous functions. This assumption ensures that the periodic corrector $w_{per}$, solution to $\Delta w_{per} = V_{per}$, belongs to $L^{\infty}(\mathbb{R}^d)$, as a consequence of the Schauder regularity theory. 

The specific form \eqref{potential_form_schro} of the potential $V$ is inspired by the work \cite{MR1974463}, related to the minimization of the energy of an infinite non-periodic system of particles. In this work were introduced several algebras generated by functions of the form $\displaystyle \sum_{k \in \mathbb{Z}^d} \varphi(x-X_k)$ together with general assumptions related to the distribution of the points $X_k$. For some particular sets $\{X_k\}_{k \in \mathbb{Z}^d}$, this setting has been employed as a \emph{motivation} to introduce several linear elliptic homogenization problems for \emph{local} perturbations of a periodic geometry in \cite{blanc2018precised, blanc2018correctors, blanc2015local, blanc2012possible}. In these cases, $X_k = k + Z_k$, where $Z_k$ is a compactly supported sequence. Stochastic cases with stationary coefficients were also considered in \cite{blanc2007stochastic} using sets of random points $X_k(\omega)$. Although the homogenization of \eqref{homog_eps_schro} in the whole generality of the sequence $X_k$ introduced in \cite{MR1974463} is to date an open mathematical question, the aim of the present contribution is to introduce a somewhat general case of sequences of the form $X_k = k + Z_k$ that model \emph{non-local} perturbations of a periodic setting when $Z_k$ does not vanish or only slowly vanishes at infinity, and for which the homogenization problem can be addressed. This work is also motivated by the hope to establish a theory of homogenization for the diffusion operator $-\operatorname{div}(a(./\varepsilon) \nabla)$, for a coefficient of the form $\displaystyle \sum_{k \in \mathbb{Z}^d} \varphi(x-X_k)$, or of a related general form. To some extent, equation \eqref{homog_eps_schro} is a \emph{bilinear} version ($V$ \emph{multiplies} the unknown function $u^{\varepsilon}$) of the diffusion equation $-\operatorname{div}(a(./\varepsilon) \nabla u^{\varepsilon}) = f$ where $a$ and $u^{\varepsilon}$ are, on the other hand, \emph{fully} entangled.

One of the main difficulties in the non-periodic setting \eqref{potential_form_schro} is the study of the corrector equation 
\begin{equation}
\label{corrector_general_schro}
   \Delta w = V \quad \text{on } \mathbb{R}^d,
\end{equation}
which, in sharp contrast with the periodic setting, cannot be reduced to an equation posed on a bounded domain. In particular, this prevents us from using classical techniques (the Lax-Milgram lemma for instance) to solve this equation. A natural approach to show the existence of a solution to \eqref{corrector_general_schro} consists in finding a solution of the form $w=G\ast V$, where $G$ denotes the Green function associated with $\Delta$ on $\mathbb{R}^d$. For $d\geq 3$, it is well-known that $G$ is of the form $G(x) = C(d) \dfrac{1}{|x|^{d-2}}$, where the constant $C(d)$ only depends on the ambient dimension, and the difficulty for solving our problem is therefore threefold. First, the existence of such a solution $w$ has to be established, the definition of $G \ast V$ being not obvious since the potential $V$ is non-periodic and is only known to belong to $L^{\infty}(\mathbb{R}^d)$. Second, we have to establish its strict sub-linearity at infinity in the sense that $\varepsilon w(./\varepsilon)$ converges to 0 on $\Omega$. We shall recall that the weak convergence of $\nabla w(./\varepsilon)$ to~0 is actually sufficient to show the latter property (see Lemma \ref{lemme_sous_linearité}). Third, we have to rigorously establish that the homogenized potential appearing in \eqref{eq_homog_per_schro} is the weak limit of $|\nabla w|^2(./\varepsilon)$ when $\varepsilon$ converges to 0. In particular, for our purpose, we have to prove some bounds, uniform with respect to $\varepsilon$ and satisfied by $\nabla w$ on $\Omega/\varepsilon$, at least in $L^2$.

To address the question related to the existence of a solution $w$ to \eqref{corrector_general_schro}, our approach first consists in using the specific structure of $V$, that is a perturbation of the periodic potential \eqref{periodic_potential_schro} by \eqref{perturbation_potential_schro} in order to find a corrector of the form
\begin{equation}
\label{corrector_form_schro}
    w = w_{per} + \Tilde{w},
\end{equation}
where $\Delta w_{per} = V_{per}$ and $\nabla \Tilde{w}$ formally reads as
\begin{equation}
\label{form_solution_schro}
   \nabla \Tilde{w} = \sum_{k \in \mathbb{Z}^d} \nabla G \ast \left(\varphi \ (.-k-Z_k) - \varphi(.-k)\right).
\end{equation} 

Since $\nabla G(x-k)$ behaves as $\dfrac{1}{|x-k|^{d-1}}$ at infinity, obtaining in the right hand side of \eqref{form_solution_schro} a series that normally converges requires to increase by more than one the exponent in the rate of decay with respect to k for large k. At the very least, a decay in $\dfrac{1}{|k|^d}$, that is a critical decay in the ambient dimension $d$, is sufficient. But then $\nabla \tilde{w}$ will be expected to only be a $BMO$ function (where $BMO$ denotes the space of functions with bounded mean oscillations, see \cite[Chapter IV]{stein1993harmonic} for instance) and not an $L^\infty$ function, which will in turn generate some additional technical difficulties in how $\nabla w$ is employed in the homogenization proof. We will return to this below, in Sections \ref{Section_approach_taylor} and \ref{Schro_linear_problem}.

For the time being, let us observe that the gain in the rate of decay must come from a suitable combination of, on the one hand, the properties of the function $\varphi$, or functions constructed from $\varphi$, appearing in the right hand side of \eqref{form_solution_schro}, and, on the other hand, of the properties of the sequence~$Z$.
One possible approach to realize the difficulty is to perform a formal Taylor expansion in the right-hand side of \eqref{form_solution_schro}. Then we have
\begin{equation}
\label{Taylor_expansion_intro_schro}
    \nabla \Tilde{w} = \sum_{k \in \mathbb{Z}^d} \nabla G \ast \left(-Z_k.\nabla \varphi(.-k) + \int_{0}^1 (1-t) Z_k^T D^2\varphi(.-k- tZ_k) Z_k  dt \right).
\end{equation}
In \eqref{Taylor_expansion_intro_schro}, it is immediately realized that the remainder term of order two yields an absolutely converging contribution to the construction of $\nabla \tilde{w}$. This term indeed only contains second derivatives of $\varphi$, which, in \eqref{Taylor_expansion_intro_schro}, give by integration by parts third-order derivatives $D^3 G$ which all decay like~$\dfrac{1}{|x|^{d+1}}$. The only possibly delicate term is the leftmost, linear term on the right-hand side of~\eqref{Taylor_expansion_intro_schro}, for which only one level, and not two levels, of derivation are gained. Put differently, the key point is the consideration of the equation
\begin{equation}
\label{schro_eq_linear_intro}
    \Delta w_1 = \sum_{k \in \mathbb{Z}^d} Z_k.\nabla \varphi(.-k),
\end{equation}
which is related to the convergence of the sum
\begin{equation}
\label{sum_linear_intro_schro}
   \sum_{k \in \mathbb{Z}^d} \nabla G \ast \left(Z_k.\nabla \varphi(.-k)\right).
\end{equation}
For this purpose, a possible additional assumption is that the integral of $\varphi$ is required to vanish. We note in passing that this algebraic condition is not a matter of a simple renormalization, since in our setting and except under trivial circumstances, there does not exist any partition of unity, that is, any function $\chi \in \mathcal{D}(\mathbb{R}^d)$ such that $\displaystyle \int_{\mathbb{R}^d} \chi = 1$ and $1 \equiv \displaystyle \sum_{k \in \mathbb{Z}^d} \chi(.-k-Z_k)$ (we will return to this later in Section \ref{section_linear_particular_case}). Then the first term of \eqref{Taylor_expansion_intro_schro} also gives an absolutely converging series (see Lemma \ref{Prop_elementaire_schro} for details), and thus, by linearity and combination of all terms, $\nabla \tilde{w}$ is shown to exist. In addition, it is indeed a $L^\infty$ function. 

Another, alternative, possible option is to assume that the Ces\`aro means of $Z$ rapidly vanish, that is, for every $R>0$, $x_0\in \mathbb{R}^d$,  $\displaystyle  \dfrac{\varepsilon^d}{|B_R|} \sum_{k \in B_R(x_0)/\varepsilon} Z_k$ rapidly converges to 0 as $\varepsilon$ vanishes, in which case the convergence of \eqref{sum_linear_intro_schro} in $L^{\infty}$ can be established (see Proposition \ref{Prop_schro_convergence_average}). For example, it is the case when $Z_k$ itself rapidly vanishes at infinity. We will elaborate upon this particular assumption below. 

Yet another option is to assume that $Z$ is the discrete gradient of a sequence $(T_k)_{k\in \mathbb{Z}^d}$, that is $(Z_k)_i = T_{k+e_i} - T_k$ for every $i \in \{1,...d\}$, such that $|T_k| =O\left( |k|^{\alpha}\right)$ for $\alpha \in [0,1[$. In that case again, an integration by parts, this time a discrete one, yields the absolute convergence of \eqref{sum_linear_intro_schro}. 

Our only take-away message to the reader is, here, that various combinations of assumptions may be considered and although we briefly consider some of the above specific cases in Section \ref{section_linear_particular_case}, the main purpose of the present contribution is to establish the existence of a corrector adapted to our setting \eqref{homog_eps_schro}-\eqref{potential_form_schro} under a set of "general" assumptions satisfied by the pair $(\varphi,Z)$ that are as weak as possible. In particular, our aim is to homogenize equation \eqref{homog_eps_schro} even if $\displaystyle \int_{\mathbb{R}^d} \varphi \neq 0$ and $Z_k$ does not vanish by any mean at infinity. As specified above, for such a general setting, we will only be able to show the convergence of \eqref{sum_linear_intro_schro} in a somewhat weak sense using the properties of the operator $T : f \mapsto \nabla^2 G \ast f$  which, as a particular element of the class of Calder\'on-Zygmund operators, is known to be only continuous from $L^{\infty}(\mathbb{R}^d)$ to $BMO(\mathbb{R}^d)$.

In any event, once the gradient \eqref{form_solution_schro} is shown to exist in a suitable functional space (namely~$L^\infty$ or $BMO$), we have to use the corrector function $w$ defined by \eqref{corrector_form_schro} for the homogenization process. This second step brings additional constraints on our input parameters $\varphi$ and $Z$. In the first place, $\displaystyle \nabla w(./\varepsilon)$ must weakly converges to 0 as $\varepsilon$ vanishes in order for $w$ to be strictly sublinear at infinity. Furthermore, simply considering the periodic setting and the expression \eqref{eq_homog_per_schro} of the homogenized coefficient in that setting, we anticipate that we will have to make sure that the weak limit of $ |\nabla w(./\varepsilon)|^2 $ exists. These two conditions (strict sublinearity of $w$ at infinity and weak convergence of $|\nabla w(./\varepsilon) |^2$) bring additional constraints on $\nabla w$ itself besides the only convergence of the sum~\eqref{form_solution_schro}. These constraints are again related to, and can be expressed with, suitable properties of the function $\varphi$ and the sequence $Z$. Actually, since we intend to be as general as possible regarding the function $\varphi$, we will only impose constraints on Z. The first of these constraints, meant to ensure the strict-sublinearity of $w$, has already been mentioned above for a different purpose. It is related to the average of Z. As for the weak convergence of $|\nabla w(./\varepsilon)|^2$, it is intuitive to realize that \emph{correlations} of the sequence Z will matter. 

Given the above general and somewhat vague considerations, we now make precise the detailed properties of the sequence $Z$ that we will assume throughout our work. The necessity of such assumptions will be motivated in Section \ref{dimension_1_schro}  with an illustrative one-dimensional example for which the corrector can be explicitly determined. 

Our first assumption is related to the strict sublinearity at infinity of the corrector. We assume that $Z$ has an average, that is there exists a constant $\langle Z \rangle \in \mathbb{R}^d$ such that 
\begin{equation}
    \label{A0}
    \tag{A1}
    \forall R>0, x_0 \in \mathbb{R}^d, \quad \lim_{\varepsilon \to 0} \dfrac{\varepsilon^d}{|B_R|} \sum_{k \in B_R(x_0)/\varepsilon} Z_k = \langle Z \rangle.
\end{equation} 
We note that Assumption \eqref{A0} is stronger that the existence of a simple average of $Z_k$ in the sense that $\displaystyle \lim_{R \to \infty} \dfrac{1}{|B_R|} \sum_{k\in B_R(x_0)} Z_k$ exists. Here, since $B_R(x_0)/\varepsilon = B_{\frac{R}{\varepsilon}}\left(\frac{x_0}{\varepsilon}\right)$, the center of the ball of radius $\dfrac{R}{\varepsilon}$ may depends on $\varepsilon$ and \eqref{A0} actually provides a certain uniformity of the convergence with respect to this center. 

The next three assumptions regard the auto-correlations of $Z$ and are related to the specific sum \eqref{sum_linear_intro_schro} that we need to manipulate. Denoting by $\overline{Z}_k : = Z_k - \langle Z \rangle$, we assume the existence of a family of constants, denoted by $\mathcal{C}_{l,i,j}$ for every $l \in \mathbb{Z}^d$ and $i,j \in \{1,...,d\}$, such that 
\begin{equation}
    \label{A00}
    \tag{A2.a}
     \forall R>0, x_0 \in \mathbb{R}^d, \quad \lim_{\varepsilon \to 0} \dfrac{\varepsilon^d}{|B_R|} \sum_{k \in B_R(x_0)/\varepsilon} (\overline{Z}_k)_i(\overline{Z}_k)_j = \mathcal{C}_{l,i,j}.
\end{equation}
\begin{equation}
\label{A2}
\tag{A2.b}
\begin{array}{cc}
\exists \delta : \mathbb{R}^+ \rightarrow \mathbb{R}^+, \ \forall i,j \in \{1,...,d\},  \ \forall R>0, \forall x_0 \in \mathbb{R}^d,  \exists \ \gamma : \mathbb{R}^+ \rightarrow  \mathbb{R}^+, \ &  \\[0.2cm]
   \displaystyle  \sup_{|l|\leq \frac{1}{\delta(\varepsilon)}} \left| \dfrac{\varepsilon^d}{|B_R|}\sum_{k \in B_R(x_0)/\varepsilon} (\overline{Z}_k)_i (\overline{Z}_{k+l})_j - \mathcal{C}_{l,i,j}\right| \leq \gamma(\varepsilon) \text{ and } \left\{\begin{array}{cc}
      \displaystyle \lim_{\varepsilon \to 0} \gamma(\varepsilon) |\ln(\varepsilon)| = 0,  &  \\[0.2cm]
      \displaystyle \lim_{\varepsilon\to 0 } \varepsilon^{-1} \delta(\varepsilon) = 0.& 
   \end{array}
   \right.
\end{array}
\end{equation}
\begin{equation}
\label{A3}
\tag{A2.c}
   \forall i,j \in \{1,...d\}, \ x \mapsto  \sum_{|l|\leq L} \mathcal{C}_{l,i,j} \left(\partial_i \partial_j G * \varphi\right)(x-l) \quad  \text{converges in $L^1_{loc}(\mathbb{R}^d)$ when } L \to +\infty.
\end{equation}
We shall see in Section \ref{Schro_linear_problem} that Assumptions \eqref{A00} \eqref{A2} and \eqref{A3} will allow us to establish the weak-convergence of $|\nabla w_1(./\varepsilon)|^2$, where $w_1$ is solution to \eqref{schro_eq_linear_intro} with a gradient of the form \eqref{sum_linear_intro_schro}. As we have just sketched above, equation \eqref{schro_eq_linear_intro} will indeed be  key in our present study. We also note that in Assumption \eqref{A2}, the rate $|\ln(\varepsilon)|$ is related to the decreasing rate of the second derivatives of $G$.

As specified, \eqref{A00} \eqref{A2} and \eqref{A3} will only be sufficient to study the solution to equation~\eqref{schro_eq_linear_intro} for which the right hand side is linear with respect to $Z$. However, in our original problem, the potential \eqref{potential_form_schro} is actually nonlinear with respect to $Z$ and we shall see that this nonlinearity implies that the convergence of $|\nabla w(./\varepsilon)|^2$, where $w$ will be given by \eqref{corrector_form_schro}, requires an additional strong assumption related to the second-order correlations of $Z$. We therefore assume that
\begin{equation}
\label{A4}
\tag{A3}
\begin{array}{cc}
   \forall F \in \mathcal{C}^0(\mathbb{R}^d \times \mathbb{R}^d), \ \forall l \in \mathbb{Z}^d, \ \exists C_{F,l}\in \mathbb{R}, \ \forall R>0, \forall x_0 \in \mathbb{R}^d, &  \\\vspace{3pt}
   \displaystyle  \lim_{\varepsilon \to 0} \dfrac{\varepsilon^d}{|B_R|}\sum_{k \in B_R(x_0)/\varepsilon} F(Z_k,Z_{k+l}) = C_{F,l}.
\end{array}
\end{equation}
Considering respectively $F(x,y) = x_i$ and $F(x,y) = (x_i-\langle Z\rangle_i)(y_j-\langle Z \rangle_j)$ for every $i,j \in \{1,...,d\}$, Assumption \eqref{A4} of course implies \eqref{A0} and \eqref{A00} but we chose here to consider these assumptions separately for a pedagogic purpose. 

We give several examples of sequences $Z$ satisfying assumptions \eqref{A0} to \eqref{A4} in Section~\ref{Section_schro_examples}. We refer to Figure \ref{fig2_schro} for an illustration in dimension $d=2$ of two sequences that satisfy our assumptions respectively in a case of local perturbation $\left(\displaystyle\lim_{|k|\to \infty} Z_k = 0\right)$ and of non-local perturbation $\left(\displaystyle\lim_{|k|\to \infty} Z_k \neq 0 \right)$.

We also note that, if $V$ is a potential of the form \eqref{potential_form_schro} and $Z$ satisfies \eqref{A0}, then $V = g_{per} + \displaystyle \sum_{k \in \mathbb{Z}^d} \overline{\varphi}(.-k-\overline{Z}_k)$ which is also of the form \eqref{potential_form_schro} where $\overline{\varphi} = \varphi(.-\langle Z \rangle)$ and $\overline{Z_k} = Z_k - \langle Z \rangle$ has a vanishing average. Consequently, for simplicity, we will sometimes assume $\langle Z \rangle=0$ in \eqref{A0} without loss of generality.

\begin{figure}[h!]
    \centering
    \includegraphics[scale =0.315]{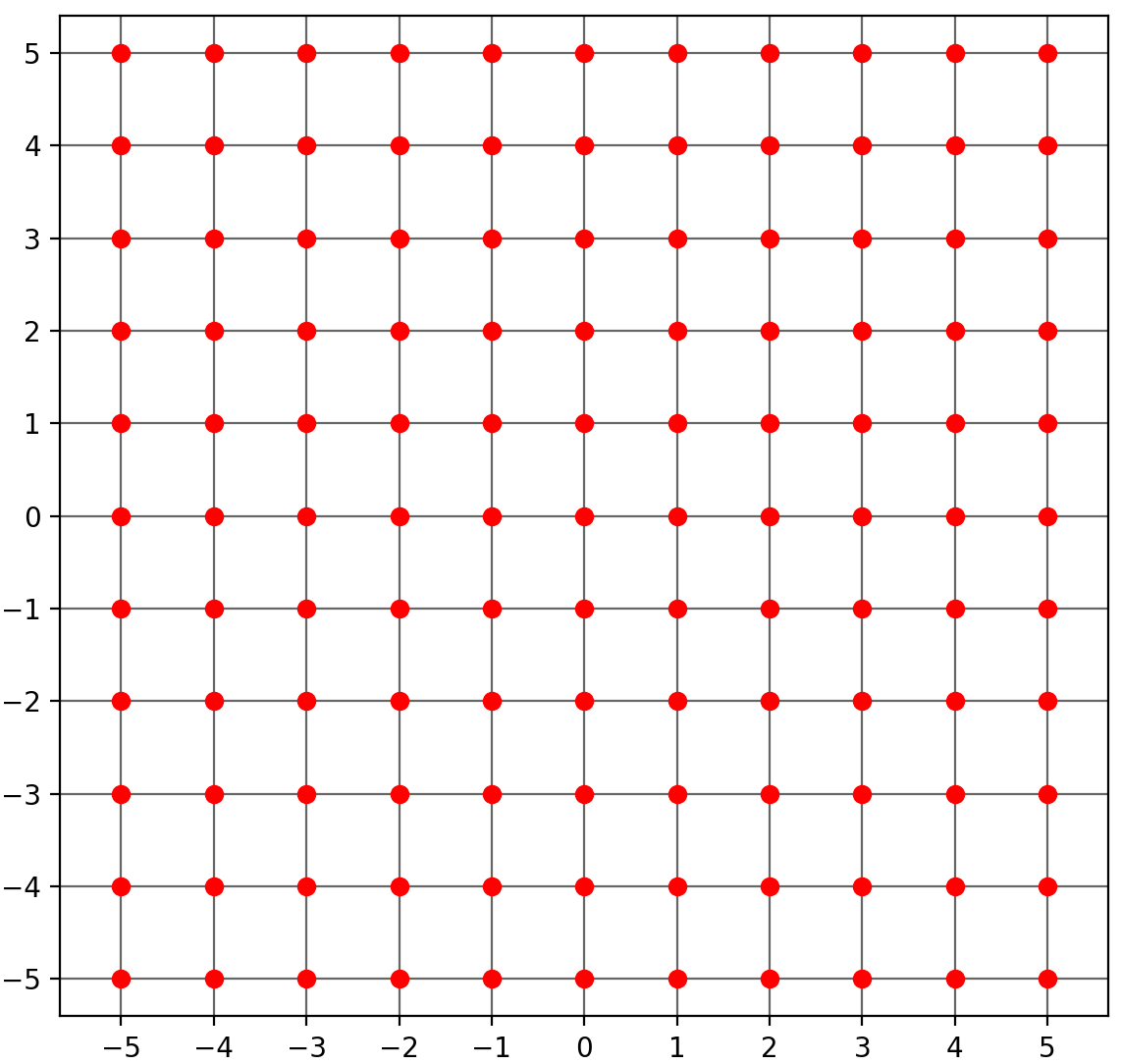} \
    \includegraphics[scale =0.325]{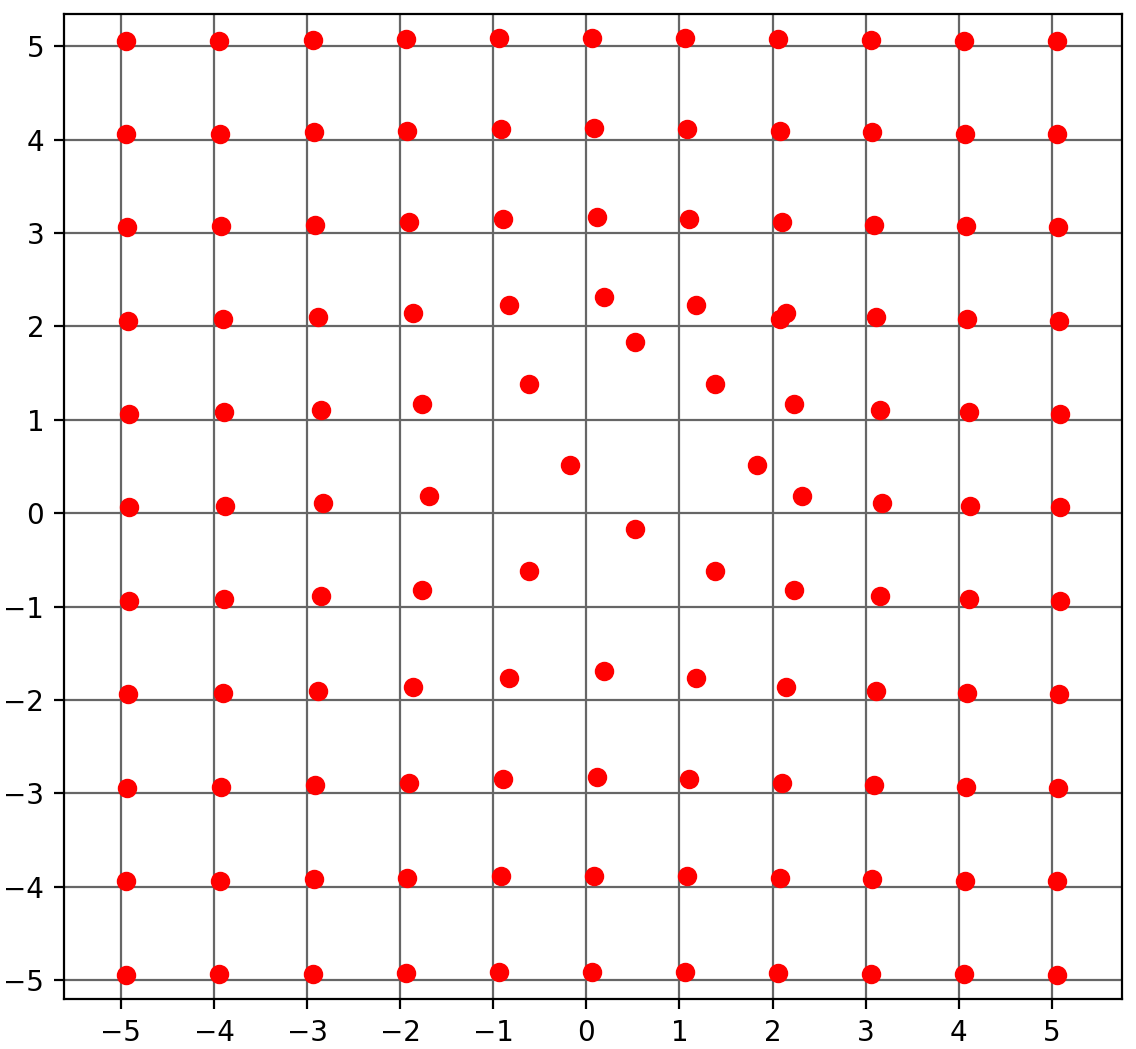} \
    \includegraphics[scale =0.295]{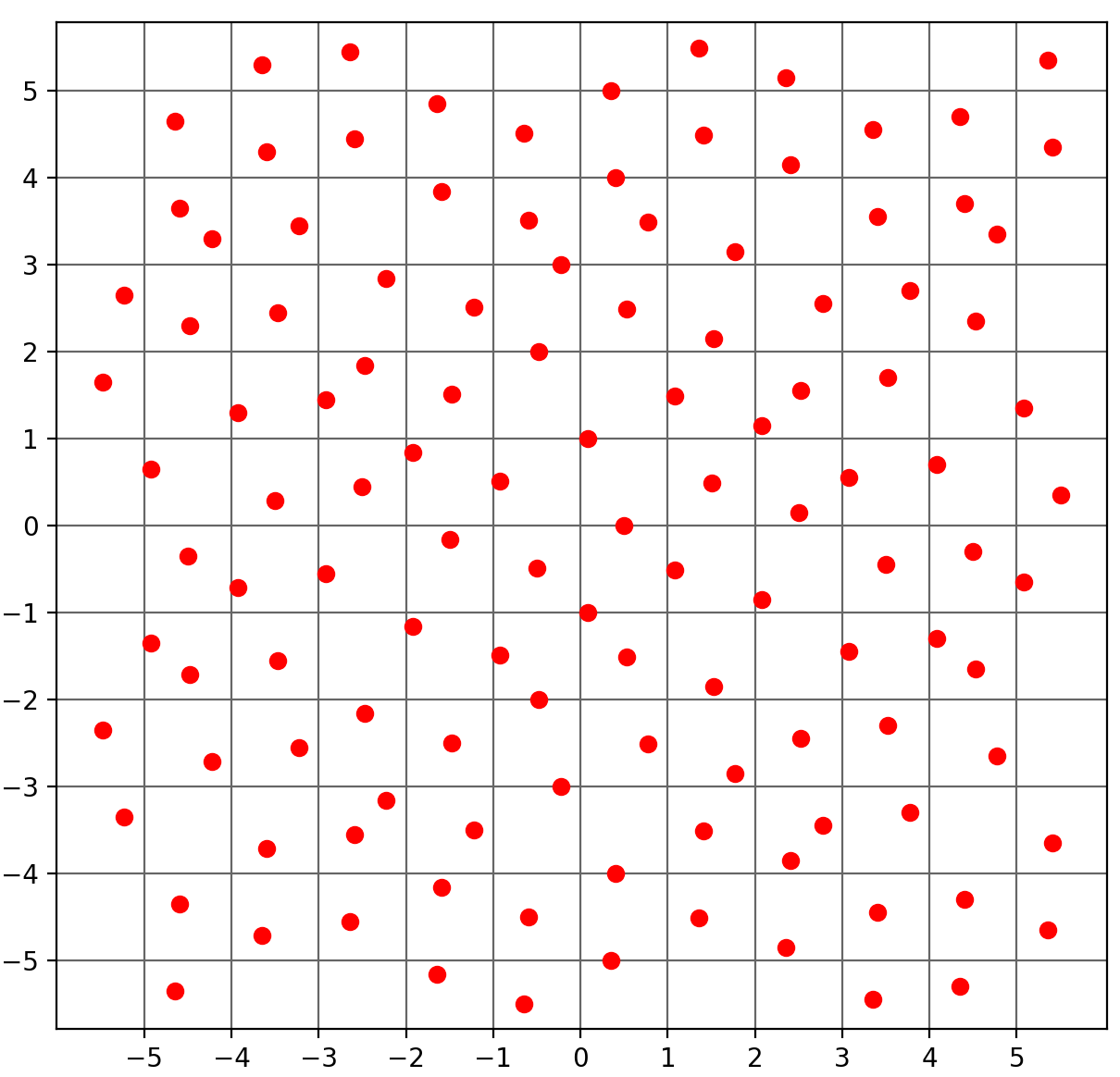}
    \caption{Illustration of $X_k = k +Z_k$ for admissible sequences $Z_k$ ($d=2$).\\
    \textbf{Left}. Reference periodic case : $Z_k =0$ \ \textbf{Center}. Local perturbation : $\displaystyle \lim_{|k| \to \infty} Z_k = 0$ \ \\\textbf{Right}. Non-local perturbations. } 
    \label{fig2_schro}
\end{figure}

\subsection{Main results}

Our main result regarding the existence of a corrector adapted to our setting is stated in the next theorem when $d\geq 2$. The case of dimension $d=1$ may be addressed by analytical arguments that are briefly presented in Section \ref{dimension_1_schro}. 
\begin{theoreme}
\label{theorem1_schro}
Assume $d\geq 2$. Assume that $V$ is a potential of the form \eqref{potential_form_schro}, where $\varphi \in \mathcal{D}(\mathbb{R}^d)$ and $g_{per}$ is a periodic function satisfying \eqref{hyp_holder_continuite_schro}, which also satisfies \eqref{average_V_schro}. Assume that $Z$ satisfies \eqref{hyp_borne_Z_schro}, \eqref{A0}, \eqref{A00}, \eqref{A2},  \eqref{A3}, and \eqref{A4}. 

Then, for every $R>0$ and every $\varepsilon>0$, there exists $W_{\varepsilon,R} \in L^1_{loc}(\mathbb{R}^d)$ solution to 
\begin{equation}
\label{equ_correcteur_theorem1}
    \Delta W_{\varepsilon,R} = V \quad \text{on } B_{R/\varepsilon},
\end{equation}
such that $\left(  \nabla W_{\varepsilon, R}(./\varepsilon)\right)_{\varepsilon>0}$ is bounded in $\left(L^p(B_R)\right)^d$ for every $p \in [1, +\infty[$ and
\begin{empheq}[left=\empheqlbrace]{alignat=2}
\quad & \nabla W_{\varepsilon,R}(./\varepsilon) \stackrel{\varepsilon \to 0}{\longrightarrow} 0 \quad \text{weakly in } L^p(B_R), \ \forall p \in [1, +\infty[, \label{b1}\\
\quad &  \varepsilon W_{\varepsilon,R}(./\varepsilon) \stackrel{\varepsilon \to 0}{\longrightarrow} 0 \quad \text{strongly in } L^{\infty}(B_R),  \label{c1} \\
\quad & \exists \mathcal{M} \in \mathbb{R}, \quad   |\nabla W_{\varepsilon,R}|^2(./\varepsilon) \stackrel{\varepsilon \to 0}{\longrightarrow} \mathcal{M} \quad \text{weakly in } L^p(B_R), \ \forall p \in [1, +\infty[. \label{d1}
\end{empheq}
\end{theoreme}

We point out one \emph{major} difference from the periodic case : the corrector is "$\varepsilon$-dependent". It however depends on $\varepsilon$ in a very specific manner which we may make entirely explicit. We shall indeed see that it is of the form $W_{\varepsilon,R} = w - x.C_{\varepsilon,R}$ where $w$ is a particular solution of \eqref{equ_correcteur_theorem1} with a gradient in $\left(BMO(\mathbb{R}^d)\right)^d$ and $C_{\varepsilon,R}$ is a constant related to the average of $\nabla w$ on $B_{R/\varepsilon}$. In fact, $W_{\varepsilon,R}$ is not even a solution to $\Delta w = V$ on the whole space $\mathbb{R}^d$ but only on the ball $B_{R/\varepsilon}$. This corrector is however sufficient to address the homogenization problem \eqref{homog_eps_schro} since one eventually only needs to evaluate the corrector on the domain $\Omega/\varepsilon$. As specified above, the proof of Theorem \ref{theorem1_schro} uses the specific structure of $V$, that is the sum of the periodic potential $V_{per}$ and of a perturbation $\Tilde{V}$ respectively defined by \eqref{periodic_potential_schro} and \eqref{perturbation_potential_schro}.

Our methodology to study the perturbation term $\Tilde{V}$ is detailed in Section \ref{Schro_preliminary_section} and consists in performing a first order Taylor expansion of the potential $\Tilde{V}$ with respect to $Z_k$ with the aim to solve both the equation induced by the first order term (linear with respect to $Z$) and the equation corresponding to the remainder of the expansion exactly suggested by our decomposition \eqref{Taylor_expansion_intro_schro}. We shall see that the difficulties to solve these two equations are very different in nature and, as we have sketched in Section \ref{Section_intro_setting_schro}, are related to the properties of the derivatives of the Green function~$G$. 

We now turn to the study of the convergence of $u^{\varepsilon}$ which is next addressed in two steps. Similarly to the periodic case, the existence of a corrector stated in Theorem \ref{theorem1_schro} first allows to establish the well-posedness of  \eqref{homog_eps_schro} when
\begin{equation}
\label{H}
   \mu_1 -\mathcal{M} + \nu > 0.
\end{equation}
This result is a consequence of our Proposition \ref{prop_convergence_vp} in Section \ref{schro_scetion_wellposedness} (the results of which are only based upon those of Theorem \ref{theorem1_schro}) which shows that the first eigenvalue $\lambda_1^{\varepsilon}$ of $-\Delta + \dfrac{1}{\varepsilon}V(./\varepsilon) + \nu $ with homogeneous Dirichlet boundary conditions converges to $\mu_1 - \mathcal{M} + \nu $ when $\varepsilon$ vanishes as is the case in the periodic setting. Under assumption \eqref{H}, we next homogenize \eqref{homog_eps_schro} and show the strong convergence of $u^{\varepsilon}$ to the function $u^*$ in $L^2(\Omega)$, solution to 
\begin{equation}
\label{equation_homog_schrodinger}
    \left\{ 
    \begin{array}{cc}
        -\Delta u^* - \mathcal{M} u^* + \nu u^* = f & \text{on } \Omega, \\
         u^* = 0 & \text{on } \partial \Omega. 
    \end{array}
\right.
\end{equation}
This result is established in Proposition \ref{prop_schro_homog} in Section \ref{Sction_homog_schro}. The behavior of $u^{\varepsilon}$ in $H^1(\Omega)$ is given in the following theorem :  
\begin{theoreme}[Homogenization results]
\label{schro_theorem_2}
Assume, as in Theorem \ref{theorem1_schro}, that $d\geq 2$, $V$ is a potential of the form \eqref{potential_form_schro} that satisfies \eqref{average_V_schro}-\eqref{hyp_holder_continuite_schro}, $Z$ satisfies \eqref{hyp_borne_Z_schro}, \eqref{A0}, \eqref{A00}, \eqref{A2},  \eqref{A3}, and \eqref{A4} and denote $W_{\varepsilon, \Omega} = W_{\varepsilon, R}$ where $R = Diam(\Omega)$. Assume also \eqref{H} is satisfied. 

Then, for $\varepsilon$ sufficiently small, there exists an unique solution $u^{\varepsilon} \in H^1_0(\Omega)$ to $\eqref{homog_eps_schro}$. In addition, the sequence
\begin{equation}
\label{schro_def_R}
    R^{\varepsilon} = u^{\varepsilon} - u^* - \varepsilon u^* W_{\varepsilon,\Omega}(./\varepsilon)
\end{equation}
strongly converges to 0 in $H^1(\Omega)$ as $\varepsilon$ vanishes. 
\end{theoreme}

In a second step, following a method introduced in \cite{zhang2021estimates} in the periodic case, we use the convergence of the operator $-\Delta + \dfrac{1}{\varepsilon}V(./\varepsilon) + \nu$ that holds under the particular assumption \eqref{H} to study the convergence of \emph{all} its eigenvalues. This allows us to establish the following generalization of Theorem \ref{schro_theorem_2} :

\begin{theoreme}
\label{schro_theorem_3}
Assume, as in Theorems \ref{theorem1_schro} and \ref{schro_theorem_2}, that $d\geq 2$, $V$ is a potential of the form \eqref{potential_form_schro} that satisfies \eqref{average_V_schro}-\eqref{hyp_holder_continuite_schro}, $Z$ satisfies \eqref{hyp_borne_Z_schro}, \eqref{A0}, \eqref{A00}, \eqref{A2},  \eqref{A3}, and \eqref{A4} but, instead of \eqref{H}, assume $\mu_l-\mathcal{M}+\nu \neq 0$ for all $l \in \mathbb{N}^*$, where $\mu_l$ is the $l^{th}$ eigenvalue (counting multiplicities) of -$\Delta$ on $\Omega$ with homogeneous Dirichlet boundary conditions. Then the conclusions of Theorem \ref{schro_theorem_2} hold true. 
\end{theoreme}

Given the convergence of all the eigenvalues of $-\Delta + \dfrac{1}{\varepsilon}V(./\varepsilon)$ established in Proposition \ref{prop_convergence_autre_vp}, the proof of Theorem \ref{schro_theorem_3} is actually a simple adaptation of the proof of Theorem \ref{schro_theorem_2}. Although this result also holds in the periodic case, to the best of our knowledge it has never been explicitly stated in the literature and that is why we prove it in the sequel.

Our article is organized as follows. We begin by presenting our approach to solve equation~\eqref{equ_correcteur_theorem1} and by collecting some preliminary results in Section \ref{Schro_preliminary_section}. In that section, we also give several examples of admissible sequences $Z$ that satisfy Assumptions \eqref{A0} to \eqref{A4}. Sections \ref{Schro_linear_problem} and \ref{Schro_section4} are next devoted to the study of the corrector equation and the proof of Theorem \ref{theorem1_schro}. Finally, in Section~\ref{Schro_section5}, we use the corrector to show Theorems \ref{schro_theorem_2} and \ref{schro_theorem_3}. 

\section{Preliminaries}

\label{Schro_preliminary_section}

Our twofold purpose here is to motivate our assumptions (\ref{A0}) through (\ref{A4}) and to emphasize the difficulties related to the corrector equation \eqref{corrector_general_schro}. We first address a particular illustrative case in dimension $d=1$ and we next introduce our approach to solve \eqref{corrector_general_schro} for higher dimensions. We conclude this section with some examples of sequences $Z$ satisfying our assumptions. 

\subsection{A one-dimensional setting}
\label{dimension_1_schro}

We briefly study here an illustrative one-dimensional setting in order to motivate our Assumptions \eqref{A0}, \eqref{A00} and \eqref{A4}. When $d=1$, one can remark that \eqref{A2} and \eqref{A3} are not needed. These two assumptions are indeed related to the specific behavior of $\nabla^2 G$ in higher dimensions, in particular, to the lack of continuity of $f \mapsto \nabla^2 G \ast f$ from $\left(L^{\infty}(\mathbb{R}^d)\right)^d$ to $\left(L^{\infty}(\mathbb{R}^d)\right)^d$ when $d>1$. 

We recall here that we consider a potential $V$ of the form \eqref{potential_form_schro} when $g_{per}$ is periodic and $\varphi \in \mathcal{D}(\mathbb{R})$. For simplicity, we assume that $V_{per} = 0$ in \eqref{periodic_potential_schro}, that is, $\displaystyle g_{per} = -\sum_{k \in \mathbb{Z}} \varphi (.-k)$. For every $k \in \mathbb{Z}$, we evidently have 
$$\displaystyle \varphi(x-k-Z_k) - \varphi(x-~k) = -Z_k\int_0^{1} \varphi'(x-k-tZ_k)dt.$$
In this specific case, the derivative of the solution $w$ to \eqref{corrector_general_schro}, which reads here as $w'' = V$, is explicitly given by 
\begin{align*}
    w'(x) = - \sum_{k \in \mathbb{Z}} Z_k\int_0^{1} \varphi(x-k-tZ_k) dt +C,
\end{align*}
where $C$ is a constant.

We first investigate the strict sublinearity of $w$ at infinity. Up to the addition of a constant, we have $w(x) = \displaystyle \int_0^x w'(y) dy$. The strict sublinearity of $w$ at infinity  is therefore equivalent to the fact that $w'$ has a vanishing average in the following sense : 
\begin{equation}
   \forall x \in \mathbb{R}^*, \quad  \lim_{\varepsilon \to 0} \dfrac{\varepsilon}{|x|} \int_0^{x/\varepsilon} w'(y) dy = 0.
\end{equation} 
Since $Z_k$ is assumed to satisfy \eqref{hyp_borne_Z_schro} and $\varphi$ is compactly supported, we have for $\varepsilon$ sufficiently small : 
\begin{align}
\label{calcul_sublinearity_d1_schro}
\begin{split}
    \int_0^{x/\varepsilon}  \sum_{k \in \mathbb{Z}} Z_k \int_0^{1} \varphi(z-k-tZ_k) dt  dz & = \sum_{k = 0}^{[x/\varepsilon]} Z_k \int_{0}^{1} \int_{\mathbb{R}} \varphi(z) dz dt + O(1) \\
    & = \sum_{k = 0}^{[x/\varepsilon]} Z_k \int_{\mathbb{R}} \varphi(z) dz + O(1),
\end{split}
\end{align}
where we have denoted by $[x] \in \mathbb{Z}$ the integer part of $x$. We have then to distinguish two cases : 
\begin{itemize}
    \item[a)] If $\displaystyle \int_{\mathbb{R}}\varphi=0$, we have $\displaystyle \lim_{\varepsilon \to 0} \dfrac{\varepsilon}{|x|} \int_{0}^{x/\varepsilon} w'(y) dy = 0$. In this case, the choice $C=0$ is the unique value of $C$ that allows for the strict sublinearity of $w$ at infinity and no additional assumption is required for $Z$.
    \item[b)] If $\displaystyle \int_{\mathbb{R}}\varphi\neq 0$, the convergence of $\displaystyle \dfrac{\varepsilon}{|x|} \int_{0}^{x/\varepsilon} w'(y) dy$ is equivalent to the convergence of $\displaystyle \sum_{k = 0}^{[x/\varepsilon]} Z_k$ to a constant $\langle Z \rangle \in \mathbb{R}^d$ independent of $x\in \mathbb{R}$. We can actually show that this property is equivalent to \eqref{A0}. The choice $\displaystyle C =- \langle Z \rangle \int_{\mathbb{R}} \varphi$ therefore allows for the strict sublinearity of $w$ in this case.
\end{itemize}

We next study the weak convergence of $|w'(./\varepsilon)|^2$. To this aim, we consider
$\displaystyle \int_{\mathbb{R}} g(x) |w'(x/\varepsilon)|^2 dx$ for every characteristic function $g = 1_{B_M(x_0)}$, for $M>0$ and $x_0 \in \mathbb{R}$. We can show that the convergence of this quantity as $\varepsilon \to 0$ is equivalent to the convergence of
\begin{align}
\label{def_S_eps_1d_schro}
    S^{\varepsilon} = \sum_{l \in \mathbb{Z}}  \left( \varepsilon \sum_{k \in B_{M}(x_0)/\varepsilon} Z_kZ_{k+l} \int_{\mathbb{R}} \varphi(x) \int_0^{1}\int_0^{1}  \varphi(x-l+t Z_k-\Tilde{t}Z_{k+l}) d\Tilde{t} \ dt \ dx \right),
\end{align}
and if $S^{\varepsilon}$ converges, we have $ \displaystyle \lim_{\varepsilon \to 0} S^{\varepsilon} = \lim_{\varepsilon \to 0} \displaystyle \int_{\mathbb{R}} g(x) |w'(x/\varepsilon)|^2 dx$. 
This equivalence can be obtained with several changes of variables and using that $\varphi$ is compactly supported. We skip its proof for brevity. We therefore remark, in the right-hand side of \eqref{def_S_eps_1d_schro}, that a specific assumption regarding the correlations of the sequence $Z_k$ is required to obtain the convergence of $S^{\varepsilon}$. For instance, if we define
$$F_l(y,z) = y \hspace{0.7pt}z \int_{\mathbb{R}} \varphi(x) \int_0^{1}\int_0^{1}  \varphi(x-l+t y-\Tilde{t}z) d\Tilde{t} \ dt \ dx, \quad \text{for } l \in \mathbb{Z},$$
we have $S^{\varepsilon} = \displaystyle \sum_{l \in \mathbb{Z}}  \left( \varepsilon \sum_{k \in B_{M}(x_0)/\varepsilon} F_l(Z_{k},Z_{k+l})\right)$ and
Assumption \eqref{A4} gives the existence of a sequence $(C_{F_l,l})_{l\in \mathbb{Z}}$ of constants that only depend on $\varphi$ and such that 
$$\displaystyle \lim_{\varepsilon \to 0} \displaystyle S^{\varepsilon} = |B_M| \sum_{l \in \mathbb{Z}} C_{F_l,l} =  \left(\int_{\mathbb{R}} g\right) \sum_{l \in \mathbb{Z}} C_{F_l,l}.$$
We note that the sum is well defined since the number of terms such that $C_{F_l,l} \neq 0$ is finite, given the compact support of $\varphi$. 

This one-dimensional example therefore suffices to show that we need two specific properties regarding the distribution of the sequence $Z$, namely :\\
\textbf{(i)} the existence of an average to have the strict sublinearity at infinity of the corrector, which assumption is the point of \eqref{A0}, and \\
\textbf{(ii)} an assumption regarding the correlations of $Z$ to ensure the weak convergence of $|\nabla w(./\varepsilon)|^2$, which is the point of \eqref{A4} (and of $\eqref{A00}$ which is implied by \eqref{A4}). 

We however note that if $ \displaystyle \int_{\mathbb{R}} \varphi =0$, the existence of an average for $Z$ is not required to obtain the strict sublinearity of $w$ at infinity. This phenomenon will also be true for higher dimensions (see Section \ref{section_linear_particular_case}).

\subsection{Taylor expansion of $V$}

\label{Section_approach_taylor}

As specified in the introductory section, the study of the gradient of a solutions to \eqref{corrector_general_schro} is related to the convergence of sums of the form $\displaystyle \sum_{k \in \mathbb{Z}^d} \nabla \left( G \ast \psi_k \right)$, where $\psi_k$ is a compactly supported function for every $k\in \mathbb{Z}^d$ that depends on $\varphi$ (or its derivatives) and on $Z_k$. For $\psi \in \mathcal{D}(\mathbb{R}^d)$, it is therefore necessary to first recall the behavior of $u = G \ast \psi$, solution to $\Delta u = \psi$. The following elementary lemma recalls the answer to this question when the first or the first two moments of $\psi$ vanish. It will be essential throughout our study. 

\begin{lemme}
\label{Prop_elementaire_schro}
Assume $d\geq 2$. Let $\psi \in \mathcal{D}(\mathbb{R}^d)$ such that $\displaystyle \int_{\mathbb{R}^d} \psi(x) dx = 0.$
Then, if we denote $u= G\ast \psi$, there exists a constant $C_1>0$ such that for every $x\in \mathbb{R}^d$
\begin{equation}
\label{borne_green_schrodinger}
   \displaystyle |\nabla u(x)| \leq \dfrac{C_1}{1+|x|^{d}}. 
\end{equation}
If we additionally assume that $\displaystyle  \int_{\mathbb{R}^d} x \psi(x) dx = 0$, there exists a constant $C_2>0$ such that 
\begin{equation}
\label{borne_green_schrodinger2}
   \displaystyle |\nabla u(x)| \leq \dfrac{C_2}{1+|x|^{d+1}}. 
\end{equation}
\end{lemme}

\begin{proof}
For $d\geq 2$ and every $i\in \{1,...,d\}$, we know that $G \in L^1_{loc}(\mathbb{R}^d)$ and $\partial_i G(x) = C(d) \dfrac{x_i}{|x|^{d}} \in L^1_{loc}(\mathbb{R}^d)$,
where $C(d)$ only depends on the ambient dimension $d$. In particular, since $\psi$ is compactly supported, $u=G\ast \psi$ is well-defined and, for every $i \in \{1,...,d\}$, 
$$\partial_i u = \int_{\mathbb{R}^d} C(d)\dfrac{x_i-y_i}{|x-y|^{d}}  \psi(y) dy.$$
We denote by $A>0$ a radius such that the support of $\psi$ is included in $B_A$. An asymptotic expansion when $|x|>>A$ shows that, for every $y \in B_A$,
\begin{align*}
  \dfrac{1}{|x-y|^{d}} = \dfrac{1}{(|x|^2 - 2 \langle x,y \rangle + |y|^2)^{d/2}} = \dfrac{1}{|x|^{d}}\left( 1 + d \dfrac{ \langle x,y\rangle}{|x|^2}\right) + O\left(\dfrac{1}{|x|^{d+2}}\right),
\end{align*}
the remainder of this expansion being uniformly bounded with respect to $y\in B_A$.
Since $ \psi$ is supported in $B_A$, we can perform this expansion in the integral above when $|x|>>A$. Since $\displaystyle \int_{\mathbb{R}^d} \psi = 0$, we deduce 
\begin{align}
    \int_{B_A} \dfrac{x_i-y_i}{|x-y|^d} \psi(y) dy &=  \int_{B_A} \left( \dfrac{x_i}{|x|^d}+d  \dfrac{x_i \langle x,y \rangle}{|x|^{d+2}} - \dfrac{y_i}{|x|^d}\right) \psi(y)dy + O\left(\dfrac{1}{|x|^{d+1}}\right) \\
    & = \int_{B_A} \left( d  \dfrac{x_i \langle x,y \rangle}{|x|^{d+2}} - \dfrac{y_i}{|x|^d}\right) \psi(y)dy + O\left(\dfrac{1}{|x|^{d+1}}\right) \\
    & = O\left(\dfrac{1}{|x|^{d}}\right).
\end{align}
We therefore obtain \eqref{borne_green_schrodinger}.
In addition, when $\displaystyle \int_{\mathbb{R}^d} x\psi =0$, we also have  
$$ \dfrac{1}{|x|^d} \int_{B_A} y_i \psi(y)dy = 0,$$
and 
$$ \int_{B_A}   \dfrac{x_i \langle x,y \rangle}{|x|^{d+2}} \psi(y)dy =  \dfrac{x_i}{|x|^{d+2}} \left\langle x, \int_{B_A} y \psi(y) dy \right \rangle = 0.$$
We obtain $\displaystyle \int_{B_A} \dfrac{x_i-y_i}{|x-y|^d} \psi(y) dy = O\left(\dfrac{1}{|x|^{d+1}}\right)$, which shows \eqref{borne_green_schrodinger2}.

\end{proof}


We now present our approach to solve \eqref{corrector_general_schro}, or more precisely \eqref{equ_correcteur_theorem1}. The specific structure of $V$ first allows us to perform a Taylor expansion :
\begin{equation}
\begin{array}{cc}
     V(x)  & =  \displaystyle g_{per} + \sum_{k \in \mathbb{Z}^d} \varphi(x-k) - Z_k.\nabla \varphi(x-k)  \\
     & + \displaystyle \int_{0}^1 (1-t) Z_k^T D^2\varphi(x-k- tZ_k) Z_k  dt,
\end{array}
\end{equation}
as already briefly mentioned in \eqref{Taylor_expansion_intro_schro}. By linearity, the corrector equation can therefore be split into three different equations : 
\begin{enumerate}[label=(\alph*)]
    \item \label{a_schro} $\Delta w_{per} = V_{per}$, where $V_{per} = g_{per} + \displaystyle \sum_{k \in \mathbb{Z}^d} \varphi(.-k)$ is a  $Q$-periodic function such that $\langle V_{per} \rangle = 0$ as a consequence of \eqref{Schro_condition_moyenne2}. 
    \item \label{b_schro} $\Delta w_1 = V_1$, where $V_1 = \displaystyle \sum_{k \in \mathbb{Z}^d} Z_k. \nabla \varphi(.-k)$ is the first order term of the Taylor expansion and is linear with respect to $Z$.
    \item \label{c_schro} $\Delta w_2 = V_2$, where $V_2 = \displaystyle \sum_{k \in \mathbb{Z}^d} \int_{0}^1 (1-t) Z_k^T D^2\varphi(.-k - tZ_k) Z_k  dt$ is the remainder of the Taylor expansion and is nonlinear with respect to $Z$. 
\end{enumerate}
In the sequel, the proof of Theorem \ref{theorem1_schro} will consequently consist in solving each of these equations and in showing the expected properties of weak convergence satisfied by the gradient of their solution. Each of these equations is put under the form $\Delta u = \displaystyle \sum_{k \in \mathbb{Z}^d} \psi_k$ for a specific potential where $\psi_k \in \mathcal{D}(\mathbb{R}^d)$ depends on $\varphi$, $k$ and $Z_k$. In particular, the difficulties to study the gradients of their solution will be related to the convergence of the sum $\displaystyle \sum_{k \in \mathbb{Z}^d} \nabla \left( G \star \psi_k\right)$  and are various in nature.

Since $V_{per}$ is a periodic potential with a vanishing average, the existence of a periodic solution $w_{per}$, therefore strictly sublinear at infinity, is well known. So \ref{a_schro} is easily solved.

The third equation \ref{c_schro} is associated with the remainder of the Taylor expansion. As informally announced in the introduction, Lemma~\ref{Prop_elementaire_schro} shows that the presence of high order derivatives of $\varphi$ in $V_2$ ensures the normal convergence of the sum $\displaystyle \sum_{k \in \mathbb{Z}^d} \nabla \left( G \ast D^2 \varphi\right)(.-k)$. This allows to make explicit $\nabla w_2$ and to show the weak convergence of $\nabla w_2(./\varepsilon)$ to 0. However, the nonlinearity of $V_2$ with respect to $Z$ requires a strong assumption such as \eqref{A4} to obtain the expected convergence of $|\nabla w_2(./\varepsilon)|^2$. This problem is addressed in Section \ref{Schro_section4}. 

As for the second equation \ref{b_schro}, the gradient of a solution to $\Delta w_1 = V_1$, associated with the first order term of the Taylor expansion, is related to the potential convergence of the sum $\displaystyle \sum_{k \in \mathbb{Z}^d} \nabla \left( G \ast Z_k.\nabla \varphi \right)(.-k)$. Since the integral of $\nabla \varphi$ vanishes, Lemma \ref{Prop_elementaire_schro} also implies that the function $\nabla \left( G \ast Z_k.\nabla \varphi \right)(.-k)$ generically formally behaves as $\dfrac{Z_k}{|x-k|^d}$. Most of our arguments to come therefore consist in showing first that the specific properties of the Calder\'on-Zygmund operator $T : f \mapsto \nabla^2 G \ast f$, particularly its continuity from $L^{\infty}(\mathbb{R}^d)$ to $BMO(\mathbb{R}^d)$, allow to prove the convergence of this sum in $BMO(\mathbb{R}^d)$ and, secondly, that the properties of $BMO(\mathbb{R}^d)$ together with our assumptions \eqref{A0}, \eqref{A00}, \eqref{A2} and \eqref{A3} ensure its weak-convergence on $B_R/\varepsilon$, provided we subtract a certain $\varepsilon$-dependent constant. More precisely, in Proposition \ref{prop_correcteur_cas_lin_schro} which is established in Section \ref{Schro_linear_problem}, we shall obtain the expected weak-convergences for some $\displaystyle \nabla \Tilde{W}_{\varepsilon,R} = \nabla w_1 - \fint_{B_{R}/\varepsilon} \nabla w_1$. This is sufficient to finally construct the $\varepsilon$-dependent corrector of Theorem \ref{theorem1_schro} and perform the homogenization of problem~\eqref{homog_eps_schro}.  


\subsection{Examples of suitable sequences $Z$}

\label{Section_schro_examples}

We give here some examples of sequences $Z$ that satisfy Assumptions \eqref{A0} to \eqref{A4}. We begin by proving a technical property related to Assumption \eqref{A3}. 

\begin{prop}
\label{Prop_schro_convergence_average}
Assume $d\geq 2$. Let $\varphi \in \mathcal{D}(\mathbb{R}^d)$ and $Z$ be a sequence of $l^{\infty}(\mathbb{Z}^d)$. Assume there exists $\alpha>1$ and $C>0$ such that for every $R>0$ and $x_0 \in \mathbb{R}^d$, 
\begin{equation}
\label{assumption_convergence_average_schro}
    \left|\dfrac{1}{|B_R|} \sum_{k \in B_R(x_0)} Z_k\right| \leq \dfrac{C}{\left(\ln(1+R)\right)^{\alpha}}.
\end{equation}
Then, for every $i,j \in \{1,...,d\}$, the sequence of functions $\displaystyle x \mapsto \sum_{|k|\leq M}  Z_k (\partial_i \partial_j G \ast \varphi)(x-k)$ converges in $\displaystyle L^{\infty}_{loc}(\mathbb{R}^d)$ when $ M\to \infty$.
\end{prop}

\begin{proof}
In this proof we denote $\mathcal{K}_{i,j} = (\partial_i \partial_j G \ast \varphi)$ for every $i,j \in \{1,...,d\}$. We define $|x|_{\infty} = \displaystyle \sup_{i \in \{1,...,d\}} |x_i|$ and $Q_M= \left\{x \in \mathbb{R}^d \ \middle| \ \displaystyle |x|_{\infty} < M \right\}$ for $M>0$. We first show the result for the subsequence $(M_N)_{N \in \mathbb{N}} = (N^{2})_{N \in \mathbb{N}}$. We remark that if $\displaystyle \mathcal{Q}_{l} = \prod_{i=1}^d [|l_i|^{2}, |l_i + 1|^{2}]$ for $l\in \mathbb{Z}^d$, then $ \displaystyle Q_{N^2} = \bigcup_{n=0}^N \ \bigcup_{|l|_{\infty}=n} \mathcal{Q}_{l}$. Since $(l_i+1)^2 - l_i^2 = O(l_i)$ for every $i \in \{1,...,d\}$, we have that $|\mathcal{Q}_l| = O(n^d)$ when $\displaystyle |l|_{\infty}=n$. We next consider a compact subset $K \subset{\mathbb{R}^d}$ and we fix $x \in K$. For $l\in \mathbb{Z}^d$ and $k\in \mathcal{Q}_l$, we also define $\mathcal{M}_l(Z)= \displaystyle \dfrac{1}{|\mathcal{Q}_l|}\sum_{q\in \mathcal{Q}_l}Z_q$ and $\omega_{l,k} = \displaystyle \dfrac{1}{|\mathcal{Q}_l|}\sum_{q\in \mathcal{Q}_l}(Z_k-Z_q)$ and we have 
\begin{align}
    \sum_{|k|_{\infty}\leq N^2} Z_k \mathcal{K}_{i,j}(x-k) & = \sum_{n=0}^N  \sum_{|l|_{\infty} = n} \sum_{k \in \mathcal{Q}_{l}} Z_k \mathcal{K}_{i,j}(x-k)\\
    & =  \sum_{n=0}^N   \sum_{|l|_{\infty} = n} \sum_{k \in \mathcal{Q}_{l}} \omega_{l,k} \mathcal{K}_{i,j}(x-k) + \sum_{n=0}^N   \sum_{|l|_{\infty} = n}\sum_{k \in \mathcal{Q}_{l}} \mathcal{M}_l(Z)\mathcal{K}_{i,j}(x-k)\\
    &= S_N^1(x) + S_N^2(x).
\end{align}
We remark that $\displaystyle \sum_{k \in \mathcal{Q}_{l}} \omega_{l,k} =0$. For every $l \in \mathbb{Z}^d$, we consider a point $k_l$ of $\mathcal{Q}_{l} \cap \mathbb{Z}^d$ that can be chosen arbitrarily and we deduce 
\begin{align}
    S^1_N(x) = \sum_{n=0}^N \sum_{|l|_{\infty} = n} \sum_{k \in \mathcal{Q}_l \setminus{\{k_l\}}} \omega_{l,k} \left(\mathcal{K}_{i,j}(x-k) - \mathcal{K}_{i,j}(x-k_l)\right).  
\end{align}
When $|l|$ is sufficiently large, using that $\varphi$ is compactly supported, we have for $k \in \mathcal{Q}_l$ :
$$\mathcal{K}_{i,j}(x-k) - \mathcal{K}_{i,j}(x-k_l) = \displaystyle \int_{\mathbb{R}^d} \left(\partial_i \partial_j G (x-k-y) -  \partial_i \partial_jG(x-k_l-y)\right) \varphi(y) dy.$$
Moreover, for every $u,v \in \mathbb{R}^d$ such that $|u| > 2|v|$, the result of \cite[Lemma 7.18 p.151]{MR3099262} implies the existence of a constant $C_1>0$ such that $\left| \partial_i \partial_j G (u-v) -  \partial_i \partial_jG(u) \right| \leq C_1\dfrac{ |v|}{|u|^{d+1}}$.
We deduce the existence of a constant $C_2>0$ such that for $n \in \mathbb{N}$ sufficiently large, $l \in \mathbb{Z}^d$ such that $|l|_{\infty} = n$ and $k\in \mathcal{Q}_l$, 
\begin{equation}
    \left|\mathcal{K}_{i,j}(x-k) - \mathcal{K}_{i,j}(x-k_l)\right| \leq C_2 \dfrac{|k - k_l|}{n^{2(d+1)}} \leq C_2 \dfrac{ Diam(\mathcal{Q}_l)}{n^{2(d+1)}}. 
\end{equation}
Since $(l_i+1)^2 - l_i^2 = 2l_i+1$ for every $i\in \{1,...d\}$, we can show the existence of $C_3>0$ such that for every $n\in \mathbb{N}$ and $l \in \mathbb{Z}^d$ such that $ |l|_{\infty} = n$, the diameter of $\mathcal{Q}_l$ is bounded by $C_3n$ and we obtain 
\begin{equation}
\label{inegalite_schro_kij}
    \left|\mathcal{K}_{i,j}(x-k) - \mathcal{K}_{i,j}(x-k_l)\right|  \leq C_2C_3 \dfrac{1}{n^{2d+1}}.
\end{equation}
In addition, since $Z\in l^{\infty}(\mathbb{Z}^d)$, $\omega_{l,k}$ is uniformly bounded with respect to $l$ and $k$. We also know that the number of $l \in \mathbb{Z}^d$ such that $|l|_{\infty}=n$ is a $O(n^{d-1})$. Using also that $|\mathcal{Q}_l| = O(n^{d})$, \eqref{inegalite_schro_kij} therefore gives the existence of $C>0$ independent of $n$, $l$ and $k$ such that : 
\begin{align}
\label{ineg_schro_remark11}
    \left|\sum_{|l|_{\infty} = n} \sum_{k \in \mathcal{Q}_l \setminus{\{k_l\}}} \omega_{l,k} \left(\mathcal{K}_{i,j}(x-k) - \mathcal{K}_{i,j}(x-k_l)\right)\right| & \leq \sup_{l,k} |\omega_{l,k}|  \sum_{|l|_{\infty} = n} C_2C_3 \dfrac{|\mathcal{Q}_l|}{n^{2d+1}}  \leq C \dfrac{1}{n^{2}}.
\end{align}
It follows that the sum $S_N^1$ normally converges in $L^{\infty}_{loc}(\mathbb{R}^d)$. 

We next study the convergence of $S^2_N$. If we denote $x_l = \left( l_i^2 + l_i +\dfrac{1}{2}\right)_{i \in \{1,...,d\}}$, we have $\mathcal{Q}_l = \displaystyle \prod_{i=1}^d \left[-l_i - \dfrac{1}{2}, l_i + \dfrac{1}{2}\right] + x_l $. Since $\mathcal{K}_{i,j} = \partial_i \left(G \ast \partial_j \varphi\right)$ and $\displaystyle \int_{\mathbb{R}^d} \partial_j \varphi =0$, Lemma \ref{Prop_elementaire_schro} shows the existence of $C_4>0$ such that $|\mathcal{K}_{i,j}(x)| \leq \dfrac{C_4}{1+|x|^d}$.
A direct consequence of  \eqref{assumption_convergence_average_schro} also yields the existence of $C_5>0$ such that for every $l \in \mathbb{Z}^d$, we have 
\begin{equation}
    \left|\mathcal{M}_{l}(Z)\right| = \left| \displaystyle \dfrac{1}{|\mathcal{Q}_l|}\sum_{q\in \mathcal{Q}_l}Z_q\right| \leq \dfrac{C_5}{\ln(1+|l|)^{\alpha}}. 
\end{equation}
We obtain 
\begin{equation}
\label{ineg_schro_remark12}
    \left|\sum_{|l|_{\infty} = n} \sum_{k \in \mathcal{Q}_l} \mathcal{M}_l(Z)\mathcal{K}_{i,j}(x-k)\right| \leq C \sum_{|l|_{\infty} = n} \dfrac{|\mathcal{Q}_l|}{n^{2d}\ln(1+n)^{\alpha}} \leq \Tilde{C} \dfrac{1}{n\ln(1+n)^{\alpha}},
\end{equation}
where $C$ and $\Tilde{C}$ are independent of $n$. Since $\alpha >1$, we deduce that the sum $S^2_N$ normally converges in $L^{\infty}_{loc}(\mathbb{R}^d)$. We have finally established the convergence of the sum $\displaystyle \sum_{|k|\leq N^2} Z_k \mathcal{K}_{i,j}(x-k)$ in $L^{\infty}_{loc}(\mathbb{R}^d)$ when $N^2 \to \infty$. 

To conclude, for $M>0$, we denote by $[M]$ the integer part of $M$. We have
\begin{align}
    \sum_{|k|\leq M} Z_k \mathcal{K}_{i,j}(x-k) = \sum_{[\sqrt{M}]^2 < |k|\leq M} Z_k \mathcal{K}_{i,j}(x-k)+ \sum_{ |k|\leq [\sqrt{M}]^2} Z_k \mathcal{K}_{i,j}(x-k).
\end{align}
We have shown above the convergence of the second sum of the right-hand term. For the other sum, we first remark that $[\sqrt{M}]^2 \stackrel{M \to \infty}{\sim} M$ and $[\sqrt{M}]^2 \leq M < [\sqrt{M}]^2 + 1$, which imply that $\left|B_{M} \setminus{B_{[\sqrt{M}]^2}}\right| = O(M^{d-1})$. Using again that $|\mathcal{K}_{i,j}(x)| \leq \dfrac{C_4}{1+|x|^d}$,  we have the existence of $C>0$ such that for every $x$ in a compact subset $K$ : 
\begin{align}
    \left|\sum_{[\sqrt{M}]^2 < |k|\leq M} Z_k \mathcal{K}_{i,j}(x-k)\right| & \leq C_4 \sum_{[\sqrt{M}]^2 < |k|\leq M} \dfrac{\left|Z_k\right|}{|x-k|^d}\\
   &  \leq C\|Z\|_{l^{\infty}(\mathbb{Z}^d)} \sum_{[\sqrt{M}]^2 < |k|\leq M} \dfrac{1}{M^d} = O(M^{-1}) \stackrel{M \to \infty}{\longrightarrow} 0.
\end{align}
We can conclude that $ \displaystyle  \sum_{|k|\leq M} Z_k \mathcal{K}_{i,j}(x-k)$ converges in $L^{\infty}_{loc}(\mathbb{
R}^d)$ when $M \to \infty$. 
\end{proof}

\begin{remark}
Assumption \eqref{assumption_convergence_average_schro} has two important parts : a logarithmic rate of convergence for the sequence of the partial averages and the uniformity of this rate with respect to the center $x_0$ of the ball $B_R(x_0)$.
They ensure the convergence of the two sums that appear respectively in the rightmost inequality of \eqref{ineg_schro_remark11} and in the rightmost inequality of \eqref{ineg_schro_remark12}. Actually, we can remark that the result of Proposition \ref{Prop_schro_convergence_average} still holds under the weaker assumption  
\begin{equation}
    \left|\dfrac{1}{|B_R|} \sum_{k \in B_R(R^2 x_0)} Z_k\right| \leq \dfrac{C_K}{\left(\ln(1+R)\right)^{\alpha}},
\end{equation}
for every $x_0$ belonging to a compact subset of $\mathbb{R}^d$ denoted by $K$ in the above inequality. 
\end{remark}

We are now in position to introduce several examples of sequences $Z$ that satisfy our assumptions. 


To start with, we mention for completeness two \emph{elementary} examples for which the homogenization of \eqref{homog_eps_schro}-\eqref{potential_form_schro} can be easily addressed : the case of periodic sequences and the case of local defects such that $Z_k$ \emph{rapidly} decreases when $|k|\to \infty$. The first setting is, of course, periodic homogenization theory.
We define the second setting as that for which $Z \in \left(l^p(\mathbb{Z}^d)\right)^d$ for some $p\in [1,+\infty[$. The potentials $V_1$ and $V_2$ respectively defined in \ref{b_schro} and \ref{c_schro} are then both in $L^p(\mathbb{R}^d)$ (as a consequence of a discrete Young-type inequality together with the fact that $Z \in \left(l^{\infty}(\mathbb{Z}^d)\right)^d$) and the existence of a corrector is then provided by the classical properties of the Laplace operator. We now check in details our general assumptions cover these two settings.


1) \underline{Periodic sequences} : \\
 If $Z_k$ is periodic, $F(Z_k,Z_{k+l})$ is also periodic for every continuous function $F$ and it is well-known that, for every $R>0$ and $x_0\in\mathbb{R}^d$, the sequence $\displaystyle \left(\dfrac{\varepsilon^d}{|B_R|}\sum_{k\in B_{R}(x_0)/\varepsilon} F(Z_k,Z_{k+l})\right)_{\varepsilon>0}$ converges to the average of $F(Z_k,Z_{k+l})$ (in the sense of periodic sequence) and the convergence rate is at least $\varepsilon$ uniformly with respect to $R$, $x_0$ and $l$. $Z$ therefore satisfies Assumptions \eqref{A0}, \eqref{A00}, \eqref{A2} and \eqref{A4}. In addition, since the sequence $\left(Z_{k+l} - \langle Z \rangle\right)_{l \in \mathbb{Z}^d}$ is periodic with a vanishing average for every $k \in \mathbb{Z}^d$, the sequence $C_{l,i,j}$ given in \eqref{A00} is also periodic with a vanishing average. Using again that the convergence rate of $ \displaystyle \dfrac{\varepsilon^d}{|B_R|}\sum_{l\in B_{R}(x_0)/\varepsilon} C_{l,i,j}$ is at least $\varepsilon$ uniformly with respect to $x_0$, Assumption \eqref{A3} is therefore a consequence of Proposition \ref{Prop_schro_convergence_average}.  \\

 2) \underline{Sequences $Z \in \left(l^p(\mathbb{Z}^d)\right)^d$, for some $p\in[1,+\infty[$} : \\
In this case, since $\displaystyle \lim_{|k| \to \infty} Z_k = 0$, the sequence  $\displaystyle \dfrac{\varepsilon^d}{|B_R|}\sum_{k\in B_{R}(x_0)/\varepsilon} F(Z_k,Z_{k+l})$ converges to $F(0,0)$ for every continuous function $F$ and we have \eqref{A0}, \eqref{A00} and \eqref{A4}. Moreover, since  $C_{l,i,j} = 0$ (with the notation of \eqref{A00}), $Z$ clearly satisfies \eqref{A3} and, using that $Z \in \left(l^p(\mathbb{Z}^d)\right)^d$ and the H\"older inequality, we can easily show that \eqref{A2} holds with $\delta(\varepsilon) = 0$ and $\gamma(\varepsilon)=O(\varepsilon^{d/p})$. \\

We next introduce several examples of sequences $Z$ that model both local and non-local perturbations and for which the homogenization of \eqref{homog_eps_schro} with potentials of the form \eqref{potential_form_schro} will be addressed in the present work.


3) \underline{Sequences $Z$ that only \emph{slowly} converge to 0 when $|k| \to \infty$} : \\
These are sequences such that, as in the previous example, \eqref{A0}, \eqref{A00}, \eqref{A3} and \eqref{A4} are satisfied since $\displaystyle \lim_{|k|\to \infty} Z_k = 0$ and for which, moreover, Assumption \eqref{A2} holds, since we may ensure that $|Z_k| = O\left(\ln(|k|)^{-\alpha}\right)$ for $\alpha > \dfrac{1}{2}$. We however may take such sequences such that $Z \notin \left(l^p(\mathbb{Z}^d)\right)^d$ for any $p\in [1,+\infty[$. 



4) \underline{Deterministic approximations of random variables} : \\
Such sequences are deterministic sequences $Z$ that share the property of i.i.d sequences of random variables and are commonly used to simulate random processes. They are some low-discrepancy sequences. We refer to \cite{drmota2006sequences, franklin1963deterministic,kuipers2012uniform} for an overview of the theory of deterministic approximation of random sequences. \\
A particular example in dimension $d=1$ is given by $(Z_k)_{k \in \mathbb{N}} = \left\{k \hspace{0.05cm} \theta^p \right\}$ (where $\left\{x\right\}$ denotes the fractional part of $x \in \mathbb{R}$) for a fixed integer $p\geq 2$ and for almost all irrational number $\theta \in \mathbb{R}$, see \cite[Section 4]{franklin1963deterministic} for details. Such a sequence is not periodic, not even almost periodic. It is dense in $[0,1]$ and simulates a realization of uniform distribution on $[0,1]$. More precisely, the results of \cite{cigler1969theorem,lawton1959note} and \cite[Theorem 5.1 p.41]{kuipers2012uniform} ensure that 
\begin{align}
   \forall F \in \mathcal{C}^{0}(\mathbb{R}), \quad &\lim_{\varepsilon \to 0} \dfrac{\varepsilon}{R}\sum_{k  = [x_0/\varepsilon]}^{[R/\varepsilon + x_0/\varepsilon]} F(Z_k) = \int_0^1 F(t)dt,\\
    \forall G \in \mathcal{C}^{0}(\mathbb{R}\times \mathbb{R}), \forall l \neq 0, \ &\lim_{\varepsilon \to 0} \dfrac{\varepsilon}{R}\sum_{k  = [x_0/\varepsilon]}^{[R/\varepsilon + x_0/\varepsilon]} G(Z_k,Z_{k+l}) = \int_0^1 \int_0^1 G(t,u)dtdu.
\end{align}
This show that our Assumptions \eqref{A0},\eqref{A00} and \eqref{A4} are satisfied. In particular, for \eqref{A00}, we have 
\begin{equation}
    \lim_{\varepsilon \to 0} \dfrac{\varepsilon}{R}\sum_{k  = [x_0/\varepsilon]}^{[R/\varepsilon + x_0/\varepsilon]} Z_kZ_{k+l} = 0,
\end{equation}
which directly implies $\eqref{A3}$.
A deterministic equivalent of the law of the iterated logarithm (which can be established using the results and the methods introduced in \cite{jagerman1963autocorrelation}, \cite[Theorem 1.193, p.198]{drmota2006sequences} and \cite[Chapter 2]{kuipers2012uniform}) also ensures
\begin{equation}
    \sup_{l \neq 0} \left|\dfrac{\varepsilon}{R}\sum_{k  = [x_0/\varepsilon]}^{[R/\varepsilon + x_0/\varepsilon]} Z_k Z_{k+l}\right| \leq C \dfrac{\sqrt{\varepsilon}}{\ln(\ln(\varepsilon))}, 
\end{equation}
and implies that Assumption \eqref{A2} is satisfied with $\delta(\varepsilon) = 0$ and $\gamma(\varepsilon) = \dfrac{\sqrt{\varepsilon}}{\ln(\ln(\varepsilon))}$.
All these results can, of course, be generalized in higher dimensions considering the vectors 
$$\left(\left\{kd \hspace{0.05cm} \theta^p \right\}, \left\{(kd+1) \theta^p \right\},..., \left\{(kd+d-1)  \theta^p \right\}\right)_{k \in \mathbb{N}},$$
for $p>2d$, such as in \cite[Theorem 21]{franklin1963deterministic}.

5) \underline{Some other non periodic sequences that do not vanish at infinity} : \\
We give an example of such a sequence in dimension $d=2$ : $Z_{(k_1,k_2)} = \left(\cos(\sqrt{2} k_1), \sin(\sqrt{2} k_2)\right)$. 
In this case, for every $a,b \in \mathbb{Z}$ $N \in \mathbb{N}$, we have
\begin{equation}
    \dfrac{1}{4N^2}\sum_{k_1=-N+a}^{k_1=N+a}\sum_{k_2=-N+b}^{k_2=N+b} \cos(\sqrt{2}k_1) = \dfrac{\cos(\sqrt{2} \hspace{0.9pt}a) \sin\left(\dfrac{2N+1}{\sqrt{2}}\right)}{2N \sin\left(\dfrac{1}{\sqrt{2}}\right)},
\end{equation}
and 
\begin{equation}
    \dfrac{1}{4N^2}\sum_{k_1=-N+a}^{k_1=N+a}\sum_{k_2=-N+b}^{k_2=N+b} \sin(\sqrt{2}k_2) = \dfrac{\sin(\sqrt{2} \hspace{0.9pt}b) \sin\left(\dfrac{2N+1}{\sqrt{2}}\right)}{2N \sin\left(\dfrac{1}{\sqrt{2}}\right)}.
\end{equation}
This implies that \eqref{A0} is satisfied by $Z_k$ and there exists $M>0$ such that for every $R>0$ and $x_0 \in \mathbb{R}^d$, 
\begin{equation}
   \left| \dfrac{1}{|B_R|} \sum_{k \in B_R(x_0)} Z_k\right| \leq \dfrac{M}{R}. 
\end{equation}

Similarly, a direct calculation that we omit here shows the existence of $M_1>0$ and of a family of constants $C_{l,i,j}(n,m) \in \mathbb{R}$ for $l\in \mathbb{Z}^2$, $i,j \in \{1,2\}$, $n,m \in \mathbb{N}$, such that, for every $x_0 \in \mathbb{R}^2$ and $R>0$, we have 
    \begin{equation}
    \label{Coorelation_schro_example}
        \left| \dfrac{1}{|B_R|} \sum_{k \in B_R(x_0)} (Z_k)_i^n(Z_{k+l})_j^m - C_{l,i,j}(n,m)\right| \leq \dfrac{M_1}{R}. 
    \end{equation}
The constants $C_{l,i,j}(n,m)$ are linear combinations of $\left(\cos(h\sqrt{2}l_1)\right)_{h \in \mathbb{N}}$, $\left(\cos(h\sqrt{2}l_2)\right)_{h \in \mathbb{N}}$, $\left(\sin(h\sqrt{2}l_1)\right)_{h \in \mathbb{N}}$ and $\left(\sin(h\sqrt{2}l_2)\right)_{h \in \mathbb{N}}$. In particular, for every $l\in \mathbb{Z}^2$ : 
\begin{equation}
\label{def_C_particular_example_schro}
   C_l= C_{l}(1,1) = \left[ \begin{array}{cc}
       \dfrac{\cos(\sqrt{2}l_1)}{2}  & 0  \\
        0 & \dfrac{\cos(\sqrt{2}l_2)}{2} 
    \end{array}
    \right].
\end{equation}
This implies that $Z_k$ satisfies \eqref{A00} and $\eqref{A2}$ and, using the density of the polynomial functions in the set of continuous functions, \eqref{Coorelation_schro_example} also shows \eqref{A4}. In addition, a direct calculation again shows the existence of $M_2>0$ such that, for every $R>0$ and $x_0 \in \mathbb{R}^2$, we have
    \begin{equation}
    \label{taux_conv_example_schro}
       \left| \dfrac{1}{|B_R|} \sum_{k \in B_R(x_0)} C_l \right| \leq \dfrac{M_2}{R}.
    \end{equation}
Proposition \ref{Prop_schro_convergence_average} then ensures that Assumption \eqref{A3} is also satisfied.

\section{Corrector equation : the first-order equation \ref{b_schro}}

\label{Schro_linear_problem}

This section is devoted to the linear equation :
\begin{equation}
\label{correcteur_lienaire_schro}
    \Delta w_1 = \sum_{k \in \mathbb{Z}^d} Z_k.\nabla \varphi =  \operatorname{div}(\mathcal{V}),
\end{equation}
where we have denoted by
\begin{equation}
\label{primitiv_pot_lienar_schro}
    \displaystyle \mathcal{V} = \sum_{k \in \mathbb{Z}^d} Z_k \  \varphi(x-k).
\end{equation}
The existence of a solution to \eqref{correcteur_lienaire_schro}, which is the equation obtained for the subproblem \ref{b_schro} in our decomposition of Section \ref{Section_approach_taylor} of our original corrector equation \eqref{corrector_general_schro}, is related to the convergence of $\displaystyle \sum_{k \in \mathbb{Z}^d} \nabla^2 G \ast \left(Z_k \varphi(.-k)\right)$ and thus to the continuity of the Riesz operator $Tf = \nabla^2 G \ast f$ from $\left(L^{\infty}(\mathbb{R}^d)\right)^d$ to the space $BMO(\mathbb{R}^d)$ of functions with bounded mean oscillations. We will not solve \eqref{correcteur_lienaire_schro} itself but we will find a solution on every ball of radius $\dfrac{1}{\varepsilon}$. To study the properties of weak convergence satisfied by the gradient of this solution, we next use several properties of $BMO(\mathbb{R}^d)$ together with the specific properties of the sequence $Z$ ensured by assumptions \eqref{hyp_borne_Z_schro}, \eqref{A0}, \eqref{A00}, \eqref{A2} and \eqref{A3}. 
In addition, although we consider here a generic pair $(\varphi, Z)$ where $Z$ is only assumed to satisfy \eqref{hyp_borne_Z_schro}, \eqref{A0}, \eqref{A00}, \eqref{A2} and \eqref{A3}, there exist several specific choices of function $\varphi$ and sequence $Z_k$ for which the study of \eqref{correcteur_lienaire_schro} is actually simpler. We give some examples of such simpler settings in Section \ref{section_linear_particular_case}.

\subsection{Some preliminary results}
Here, we introduce several preliminary results related to the study of equation \eqref{correcteur_lienaire_schro}. These results are elementary and classical but we include them here for completeness. 
Our approach being based on the continuity of the Riesz operator from $\left(L^{\infty}(\mathbb{R}^d)\right)^d$ to $\left(BMO(\mathbb{R}^d)\right)^d$, we begin with two preliminary technical lemmas related to the functions of $BMO(\mathbb{R}^d)$ : Lemma $\ref{lemme_BMO_harmonique}$ regards the harmonic functions with a gradient in $\left(BMO(\mathbb{R}^d)\right)^d$ and Lemma~\ref{Lemme_convergence_BMO} shows a property of convergence satisfied by the functions that weakly converge to 0 in $BMO(\mathbb{R}^d)$. In Lemma \ref{lemme_sous_linearité} below, we also recall a classical result regarding the strict sublinearity at infinity of functions $u$ such that $\nabla u(./\varepsilon)$ weakly converges to $0$ as $\varepsilon$ vanishes. We finally conclude this section with Lemma \ref{prop_convergence_primlitiv_pot} which is related to the weak convergence of the functions defined as in \eqref{primitiv_pot_lienar_schro} and is the direct generalization of the calculation \eqref{calcul_sublinearity_d1_schro} performed for $d=1$. 

\begin{lemme}
\label{lemme_BMO_harmonique}
Let $v\in L^1_{loc}(\mathbb{R}^d)$ be a solution to $\Delta v = 0$ in $\mathcal{D}'(\mathbb{R}^d)$ such that $\nabla v \in \left(BMO(\mathbb{R}^d)\right)^d$. Then $\nabla v$ is constant. 
\end{lemme}

\begin{proof}
Differentiating the equation, for every $i\in \{1,...d\}$, we have $\Delta \partial_i v = 0$. Since $\partial_i v \in BMO(\mathbb{R}^d) \subset \mathcal{S}'(\mathbb{R}^d)$, $\partial_i v$ is an harmonic tempered distribution and it is therefore a polynomial function (see \cite[Example 4.11 p.142]{l1992theory} for instance). The only polynomials that belong to $BMO(\mathbb{R}^d)$ being the constants, we can conclude. 
\end{proof}

\begin{lemme}
\label{Lemme_convergence_BMO}
Let $(v_n)_{n\in \mathbb{N}}$ be a sequence of $BMO(\mathbb{R}^d)$ such that $v_n$ converges to 0 for the weak$-\star$ topology of $BMO(\mathbb{R}^d)$. Then, for every compactly supported function $g \in L^{\infty}(\mathbb{R}^d)$ such that $\displaystyle \int_{\mathbb{R}^d} g = 1$, we have
$$v_n - \int_{\mathbb{R}^d} g v_n \overset{n \to \infty}{\longrightarrow} 0 \quad  \text{in } \mathcal{D}'(\mathbb{R}^d).$$
\end{lemme}

\begin{proof}
Let $\phi \in \mathcal{D}(\mathbb{R}^d)$, we have : 
\begin{align*}
    \left\langle v_n - \int_{\mathbb{R}^d} g v_n , \phi \right\rangle_{\mathcal{D}',\mathcal{D}} &=  \int_{\mathbb{R}^d} v_n \phi  - \left(\int_{\mathbb{R}^d} \phi\right) \int_{\mathbb{R}^d} v_n g  \\ 
    & = \int_{\mathbb{R}^d} v_n \left(\phi - g\int \phi\right).
\end{align*}
We note that $\psi = \displaystyle \phi - g\int_{\mathbb{R}^d} \phi$ belongs to $L^{\infty}(\mathbb{R}^d)$, is compactly supported and its integral vanishes. Therefore (see \cite[Section 6.4.1]{MR3099262}), $\psi$ belongs to the Hardy space $\mathcal{H}^1(\mathbb{R}^d)$, that is to the predual of $BMO(\mathbb{R}^d)$ (see \cite[Chapter IV]{stein1993harmonic} for details). Since, by assumption, $v_n$ converges to 0 for the weak-$\star$ topology of $BMO$, it follows that the right-hand side in the latter equality converges to 0 when $n \to \infty$ and we conclude.
\end{proof}


\begin{lemme}
\label{lemme_sous_linearité}
Let $u\in L^1_{loc}(\mathbb{R}^d)$ such that u(0) = 0 and denote by $v_{\varepsilon} = \varepsilon u(./\varepsilon)$. Assume that $\nabla v_{\varepsilon}$ converges to 0 in  $\mathcal{D}'(\mathbb{R}^d)$ when $\varepsilon \to 0$ and that there exists $R>0$ and $p>d$ such that $\nabla v_{\varepsilon}$ is bounded in $L^p(B_{4R})$, uniformly with respect to $\varepsilon$. Then
\begin{equation}
   \lim_{\varepsilon \to 0} \|v_{\varepsilon}\|_{L^{\infty}(B_R)} = 0.
\end{equation}
\end{lemme}
\begin{proof}
The Morrey inequality (see for instance \cite[p.268]{evans10}) gives the existence of a constant $C>0$ independent of $\varepsilon$ such that for every $y,x$ in  $B_{2R}$
$$ |v_{\varepsilon}(x) - v_{\varepsilon}(y)| \leq C |x-y|^{1-d/p} \|\nabla v_{\varepsilon}\|_{L^p(B_{4R})}.$$
Therefore, since $\nabla v_{\varepsilon}$ is bounded in $L^p(B_{4R})$ uniformly with respect to $\varepsilon$ for $p>d$, the sequence $v_{\varepsilon}$ is also bounded in $L^{\infty}(B_{2R})$ and equicontinuous on $B_{2R}$, both uniformly with respect to $\varepsilon$. Thus, the Arzela-Ascoli theorem shows that the sequence $v_{\varepsilon}$, up to an extraction, converges uniformly on every compact of $B_{2R}$ to some function $v$. Since $\nabla v_{\varepsilon}$ converges to 0 in $\mathcal{D}'(\mathbb{R}^d)$, $v$ is constant and, since $v_{\varepsilon}(0) = 0$ for every $\varepsilon>0$, we necessarily have $v=0$. Finally $v = 0$ is the only adherent value of $v_{\varepsilon}$ in $L^{\infty}(B_R)$ and we can conclude that the whole sequence $v_{\varepsilon}$ converges to 0 in $L^{\infty}(B_{R})$. 
\end{proof}

\begin{lemme}
\label{prop_convergence_primlitiv_pot}
Let $\varphi \in \mathcal{D}(\mathbb{R}^d)$ and $\mathcal{V}$ be a function of the form \eqref{primitiv_pot_lienar_schro}. Assume $Z$ satisfies \eqref{hyp_borne_Z_schro} and \eqref{A0}. Then $\mathcal{V}(./\varepsilon)$ weakly converges to $\displaystyle \langle Z \rangle \int_{\mathbb{R}^d} \varphi$ in $L^{\infty}(\mathbb{R}^d)- \star$ as $\varepsilon$ vanishes.
\end{lemme}

\begin{proof}
For $M>0$ and $x_0 \in \mathbb{R}^d$, we first introduce $g=1_{B_M(x_0)}$, the characteristic function of $B_M(x_0)$.  We have 
\begin{align}
    \langle \mathcal{V}(./\varepsilon), g\rangle_{L^{\infty}(\mathbb{R}^d),L^1(\mathbb{R}^d)} = \varepsilon^d \sum_{k \in \mathbb{Z}^d} Z_k \int_{B_{M}(x_0)/\varepsilon} \varphi(x-k) dx.  
\end{align}
We denote by $A>0$ a radius such that $\varphi$ is supported in $B_A$. We note that, if $k \notin B_{M/\varepsilon+A}(x_0/\varepsilon)$, we have $\varphi(x-k) = 0 $ for every $x \in B_{M}(x_0)/\varepsilon$ and, if $k \in B_{M/\varepsilon-A}(x_0/\varepsilon)$, we have $$\displaystyle \int_{B_M(x_0)/\varepsilon} \varphi(x-k) dx = \int_{\mathbb{R}^d} \varphi(x) dx.$$
Since $Z \in \left(l^{\infty}(\mathbb{Z}^d)\right)^d$ and the number of indices $k\in \mathbb{Z}^d$ such that $k\in B_{M/\varepsilon+A}(x_0/\varepsilon)\setminus{B_{M/\varepsilon-A}(x_0/\varepsilon)}$ is bounded by $C \varepsilon^{1-d}$, where $C$ only depends on $d$, $M$ and $x_0$, we deduce 
\begin{align}
\label{note_for_remark_schro}
    \langle \mathcal{V}(./\varepsilon), g\rangle_{L^{\infty}(\mathbb{R}^d),L^1(\mathbb{R}^d)}&= \varepsilon^d \sum_{k \in B_{M}(x_0)/\varepsilon} Z_k \int_{B_A} \varphi(x) dx + O(\varepsilon).
\end{align}
Assumption \eqref{A0} therefore shows that 
\begin{equation}
    \lim_{\varepsilon \to 0}\langle \mathcal{V}(./\varepsilon), g\rangle_{L^{\infty}(\mathbb{R}^d),L^1(\mathbb{R}^d)} =  \left(\langle Z \rangle \int_{B_A} \varphi\right) |B_M(x_0)| = \left(\langle Z \rangle \int_{B_A} \varphi\right)\int_{\mathbb{R}^d} g.
\end{equation}
Since $g$ is an arbitrary characteristic function, we conclude using the density of simple (a.k.a step) functions in $L^1(\mathbb{R}^d)$.
\end{proof}

\begin{remark}
\label{remark_average_intphi_vanish}
When $\displaystyle \int_{\mathbb{R}^d} \varphi = 0$, we note that the above proof (equality \eqref{note_for_remark_schro} in particular) also shows that $\mathcal{V}(./\varepsilon)$ weakly converges to 0 without assuming \eqref{A0}. 
\end{remark}

\subsection{Existence result}

We next turn to the main proposition of this section that shows the existence of an $\varepsilon$-dependent solution to \eqref{correcteur_lienaire_schro} on $B_{R/\varepsilon}$ for every fixed $R>0$. In the sequel, we use the notation $\displaystyle \fint_A = \dfrac{1}{|A|} \int_A$ for every Borel subset $A$ of $\mathbb{R}^d$.

\begin{prop}
\label{prop_correcteur_cas_lin_schro}
Assume $d\geq 2$. Let $\varphi \in \mathcal{D}(\mathbb{R}^d)$ and $\mathcal{V}$ be a function of the form \eqref{primitiv_pot_lienar_schro}. Assume $Z$ satisfies \eqref{hyp_borne_Z_schro}, \eqref{A0}, \eqref{A00}, \eqref{A2} and $\eqref{A3}$. We denote by $G$ be the Green function associated with $\Delta$ on $\mathbb{R}^d$ and by $u_i = G \ast \partial_i \varphi$ for $i \in \{1,...,d\}$. Let $\beta : \mathbb{R}^+ \rightarrow \mathbb{R}^+$ such that $\displaystyle \lim_{\varepsilon \to 0} \varepsilon^{-1} \beta(\varepsilon) = 0$. For every $R>0$ and $\varepsilon>0$, we define,
$$ \Tilde{w}_{\varepsilon} = \sum _{i \in \{1,...,d\} } \sum_{k \in B_{ \frac{1\mathstrut}{ \mathstrut \Large{\beta(\varepsilon)}\mathstrut}}}  (Z_k)_i u_i(. -k),$$
and
\begin{equation}
\label{corrector_schro_BMO}
    \Tilde{W}_{\varepsilon, R} = \Tilde{w}_{\varepsilon} - x. \fint_{B_{4R}} \nabla \Tilde{w}_{\varepsilon}(y/\varepsilon) dy.
\end{equation}
Then, when $\varepsilon$ is sufficiently small, $\Tilde{W}_{\varepsilon, R}$ is a solution to 
\begin{equation}
\label{correcteur_tronque}
    \Delta \Tilde{W}_{\varepsilon,R} =  \operatorname{div}(\mathcal{V}) \quad \text{on } B_{R/\varepsilon}, 
\end{equation}
such that  $ \nabla \Tilde{W}_{\varepsilon,R}(./\varepsilon) \in L^p(B_R)$ for every $p\in [1, + \infty[$ and 
\begin{empheq}[left=\empheqlbrace]{alignat=2}
\quad  & \left(  \nabla W_{\varepsilon, R}(./\varepsilon)\right)_{\varepsilon>0} \text{is bounded in } L^p(B_R), \ \forall p \in [1, + \infty[, \label{a}\\
\quad & \nabla \Tilde{W}_{\varepsilon,R}(./\varepsilon) \stackrel{\varepsilon \to 0}{\longrightarrow} 0 \quad \text{weakly in } L^p(B_R), \ \forall p \in [1, +\infty[, \label{b}\\
\quad &  \varepsilon \Tilde{W}_{\varepsilon,R}(./\varepsilon) \stackrel{\varepsilon \to 0}{\longrightarrow} 0 \quad \text{strongly in } L^{\infty}(B_R).  \label{c} 
\end{empheq}
In addition, if we denote by $\mathcal{C}_{l,i,j}$ the constant defined by Assumption \eqref{A00} for $l\in\mathbb{Z}^d$ and $i,j \in \{1,...,d\}$, and by $\Tilde{\mathcal{M}}$ the constant
$$\Tilde{\mathcal{M}} = \sum_{i,j \in \{1,...,d\}} \sum_{l\in \mathbb{Z}^d} \mathcal{C}_{l,i,j} \int_{\mathbb{R}^d} \varphi(x) \partial_i u_j (x-l) dx,$$
then, 
\begin{equation}
\label{d} 
    |\nabla \Tilde{W}_{\varepsilon,R}(./\varepsilon)|^2 \stackrel{\varepsilon \to 0}{\longrightarrow} \Tilde{\mathcal{M}} \quad \text{weakly in } L^p(B_R) \text{ for every } p\in [1,+\infty[.
\end{equation}
\end{prop} 

\begin{proof}
The proof is rather lengthy and proceeds in four steps. Here we assume that $\langle Z \rangle =0$ in \eqref{A0}. We indeed recall that we can always assume that $Z$ has a vanishing average without loss of generality. 

\textbf{Step 1 :  Proof of \eqref{a}}. For every $\varepsilon>0$, we denote $\displaystyle \mathcal{V}_{\varepsilon} = \sum_{k \in B_{ \frac{1\mathstrut}{ \mathstrut \Large{\beta(\varepsilon)}\mathstrut}}} Z_k \hspace{0.1cm} \varphi(.-k)$. For every $R>0$, since $\beta(\varepsilon) =o(\varepsilon)$ and $\varphi$ is compactly supported, we have $\mathcal{V}_{\varepsilon} = \mathcal{V}$ on $B_{R/\varepsilon}$ when $\varepsilon$ is sufficiently small. It follows that $\Tilde{w}_{\varepsilon} = G \ast \operatorname{div}(\mathcal{V}_{\varepsilon})$ is a solution to \eqref{correcteur_tronque} in $\mathcal{D}'(B_{R/\varepsilon})$. We next remark that 
$$\nabla \Tilde{w}_{\varepsilon} = \nabla G \ast \operatorname{div}(\mathcal{V}_{\varepsilon}) = \int_{\mathbb{R}^d} \nabla^2 G(.-y) \mathcal{V}_{\varepsilon}(y) dy =: T\mathcal{V}_{\varepsilon}.$$ 
It is well known that the operator $T : f \mapsto Tf$ is continuous from $\left(L^{\infty}(\mathbb{R}^d)\right)^d$ to $\left(BMO(\mathbb{R}^d)\right)^d$ (see \cite[Section 4.2]{stein1993harmonic}), that is, there exists a constant $C>0$ independent of $R$ and $\varepsilon$ such that : 
\begin{equation}
\label{Continuity_BMO_L}
    \|\nabla \Tilde{w}_{\varepsilon}\|_{BMO(\mathbb{R}^d)} \leq C \|\mathcal{V}_{\varepsilon} \|_{L^{\infty}(\mathbb{R}^d)} \leq C \|\mathcal{V}\|_{L^{\infty}(\mathbb{R}^d)}.
\end{equation}
In addition, for every $p\in [1,+\infty[$, the John-Nirenberg inequality (see for instance \cite[p.144]{stein1993harmonic}) yields the existence of a constant $c_p>0$, independent of $\varepsilon$, such that 
\begin{equation}
\label{Jonh_nirenberg}
    \sup_{M>0} \left(\fint_{B_M} \left|\nabla \Tilde{w}_{\varepsilon}(x) - \fint_{B_M} \nabla \Tilde{w}_{\varepsilon}(y)dy\right|^p dx\right)^{1/p} \leq c_p \|\nabla \Tilde{w}_{\varepsilon}\|_{BMO(\mathbb{R}^d)}.
\end{equation}
Using \eqref{Continuity_BMO_L} and \eqref{Jonh_nirenberg}, we obtain 
\begin{align*}
   \left(\int_{B_{4R}}  \left| \nabla \Tilde{W}_{\varepsilon,R}(x/\varepsilon) \right|^p dx\right)^{1/p} & =  \left(\varepsilon^d \int_{B_{4R/\varepsilon}}  \left| \nabla \Tilde{w}_{\varepsilon}(x) - \fint_{B_{4R/\varepsilon}} \nabla \Tilde{w}_{\varepsilon}(y) dy \right|^p dx\right)^{1/p}  \\
   & \leq |B_{4R}|^{1/p} c_p \|\nabla \Tilde{w}_{\varepsilon}\|^p_{BMO(\mathbb{R}^d)} \\
   & \leq |B_{4R}|^{1/p}c_pC \|\mathcal{V}\|_{L^{\infty}(\mathbb{R}^d)}.
\end{align*}
We deduce that $\nabla \Tilde{W}_{\varepsilon,R}(./\varepsilon)$ is uniformly bounded in $L^p(B_{4R})$ with respect to $\varepsilon$.\\

\textbf{Step 2 : Proof of \eqref{b}}. We begin by showing that $\nabla \Tilde{W}_{\varepsilon,R}(./\varepsilon)$ converges to 0 in $\mathcal{D}'(\mathbb{R}^d)$. Using \eqref{Continuity_BMO_L}, we know that $\nabla \Tilde{w}_{\varepsilon}(./\varepsilon)$ is uniformly bounded in $\left(BMO(\mathbb{R}^d)\right)^d$ with respect to $\varepsilon$. Up to an extraction, it therefore converges for the weak-$\star$ topology of $\left(BMO(\mathbb{R}^d)\right)^d$ when $\varepsilon \to 0$. Its limit is also a gradient and we denote it by $\nabla v$. Since $\Tilde{w}_{\varepsilon}$ is solution to $\Delta \Tilde{w}_{\varepsilon} = \operatorname{div}(\mathcal{V}_{\varepsilon})$ in $\mathcal{D}'(\mathbb{R}^d)$, we have for every $\phi \in \mathcal{D}(\mathbb{R}^d)$,
\begin{equation}
\label{correcteur_distrib_rescale}
    \int_{\mathbb{R}^d} \nabla \Tilde{w}_{\varepsilon}(./\varepsilon).\nabla \phi = \int_{\mathbb{R}^d} \mathcal{V}_{\varepsilon}(./\varepsilon) . \nabla \phi.
\end{equation}
We note that Assumption \eqref{A0} with $\langle Z\rangle =0$ and Lemma \ref{prop_convergence_primlitiv_pot} ensure that $\mathcal{V}_{\varepsilon}(./\varepsilon)$ converges to 0 for the weak-$\star$ topology of $\left(L^{\infty}(\mathbb{R}^d)\right)^d$. Since every function $\chi \in \mathcal{D}(\mathbb{R}^d)$ such that $\displaystyle \int_{\mathbb{R}^d} \chi = 0$ belongs to the Hardy space $\mathcal{H}^1$, the predual of $BMO$, we can pass to the limit in \eqref{correcteur_distrib_rescale} when $\varepsilon \to 0$ and it follows
$$\int_{\mathbb{R}^d} \nabla v.\nabla \phi = 0.$$
Therefore, $v$ is solution to $\Delta v = 0$ in $\mathcal{D}'(\mathbb{R}^d)$ such that $\nabla v$ belongs to $\left(BMO(\mathbb{R}^d)\right)^d$. Lemma~\ref{lemme_BMO_harmonique} therefore implies that $\nabla v$ is constant. The space $BMO(\mathbb{R}^d)$ being the quotient space of functions $g\in L^{1}_{loc}(\mathbb{R}^d)$ such that $\|g\|_{BMO(\mathbb{R}^d)} < + \infty$ modulo the space of constant functions, we have actually shown that $\nabla v$ is equal to 0 in $\left(BMO(\mathbb{R}^d)\right)^d$. We deduce that 0 is the only adherent value of $\nabla \Tilde{w}_{\varepsilon,R}(./\varepsilon)$ in $BMO(\mathbb{R}^d)-\star$, and a compactness argument shows that the whole sequence $\nabla \Tilde{w}_{\varepsilon}(./\varepsilon)$ converges to 0 for this topology.
Lemma \ref{Lemme_convergence_BMO} next shows that
$$\partial_i \Tilde{w}_{\varepsilon}(./\varepsilon) - \int_{\mathbb{R}^d} g_i \partial_i \Tilde{w}_{\varepsilon}(y/\varepsilon)dy \overset{\varepsilon \to 0}{\longrightarrow} 0 \quad  \text{in } \mathcal{D}'(\mathbb{R}^d),$$
for every compactly supported function $g\in \left(L^{\infty}(\mathbb{R}^d)\right)^d$ such that $\displaystyle \int_{\mathbb{R}^d} g_i = 1$ for all $i\in \{1,...,d\}$. Considering $g_i = \dfrac{1}{|B_{4R}|}1_{B_{4R}}$ for every $i\in \{1,...,d\}$, we obtain that $\nabla \Tilde{W}_{\varepsilon,R}(./\varepsilon)$ converges to 0 in $\mathcal{D}'(\mathbb{R}^d)$. Using the density of $\mathcal{D}'(B_R)$ in $L^p(B_R)$ for every $p\in [1,+\infty[$, we finally deduce the weak convergence of $\nabla \Tilde{W}_{\varepsilon,R}(./\varepsilon)$ to 0 in $L^p(B_R)$. \\

\textbf{Step 3 : Proof of \eqref{c}}. We have shown that $\nabla \Tilde{W}_{\varepsilon,R}(./\varepsilon)$ is uniformly bounded in $L^p(B_{4R})$ for every $p\in [1, +\infty[$. We note that we can always consider a solution $\Tilde{W}_{\varepsilon,R}$ such that $\Tilde{W}_{\varepsilon,R}(0)=0$ without altering the properties of $\nabla \Tilde{W}_{\varepsilon,R}$. The uniform convergence of $\varepsilon \Tilde{W}_{\varepsilon,R}(./\varepsilon)$ to 0 in $B_R$ is therefore a consequence of Lemma \ref{lemme_sous_linearité}. \\

\textbf{Step 4 : Proof of \eqref{d}}. For every $\phi \in \mathcal{D}'(B_R)$ and every $\varepsilon>0$, multiplying \eqref{correcteur_tronque} by $\chi_{\varepsilon}(x) = \Tilde{W}_{\varepsilon,R}(x/\varepsilon) \phi(x)$ and integrating, we obtain 
\begin{align*}
    \int_{B_R} |\nabla \Tilde{W}_{\varepsilon,R}|^2(./\varepsilon) \hspace{0.7pt} \phi & = \int_{B_R} \mathcal{V}(./\varepsilon).\nabla \Tilde{W}_{\varepsilon, R}(./\varepsilon) \hspace{0.7pt} \phi  \\
    & + \int_{B_R} \varepsilon \Tilde{W}_{\varepsilon,R}(./\varepsilon) \left(\mathcal{V}(./\varepsilon) - \nabla \Tilde{W}_{\varepsilon ,R}(./\varepsilon)\right). \nabla \phi.
\end{align*}
Since $\varepsilon \Tilde{W}_{\varepsilon ,R}(./\varepsilon)$ converges to 0 in $L^{\infty}(B_R)$ and $\mathcal{V}(./\varepsilon) - \nabla \Tilde{W}_{\varepsilon ,R}(./\varepsilon)$ is bounded in $L^2(B_R)$, uniformly with respect to $\varepsilon$, we have 
$$\lim_{\varepsilon \to 0} \int_{B_R} \varepsilon \Tilde{W}_{\varepsilon,R}(./\varepsilon) \left(\mathcal{V}(./\varepsilon) - \nabla \Tilde{W}_{\varepsilon ,R}(./\varepsilon)\right). \nabla \phi = 0.$$
It is therefore sufficient to show the weak convergence of $\mathcal{V}(./\varepsilon).\nabla \Tilde{W}_{\varepsilon, R}(./\varepsilon)$ to obtain the weak convergence of $|\nabla \Tilde{W}_{\varepsilon,R}|^2(./\varepsilon)$ and, in this case, we have :
\begin{equation}
\label{weak_convergence_square_schro}
    \operatorname{weak}\lim_{\varepsilon \to 0} |\nabla \Tilde{W}_{\varepsilon,R}|^2(./\varepsilon) = \operatorname{weak}\lim_{\varepsilon \to 0} \mathcal{V}(./\varepsilon).\nabla \Tilde{W}_{\varepsilon, R}(./\varepsilon).
\end{equation}

We thus study the sequence $\mathcal{V}(./\varepsilon).\nabla \Tilde{W}_{\varepsilon, R}(./\varepsilon)$ when $\varepsilon \to~0$. We consider $0<M\leq R$ and $x_0 \in \mathbb{R}^d$ such that $B_M(x_0) \subset B_R$ and we denote by $g = 1_{B_M(x_0)}$ the characteristic function of $B_M(x_0)$. We next introduce 
$$\alpha(\varepsilon) = \dfrac{\varepsilon^d}{|B_M|} \sum_{k \in B_M(x_0)/ \varepsilon} Z_k,$$
and we define $\eta(\varepsilon) = \max(\varepsilon \alpha(\varepsilon) , \beta(\varepsilon), \delta(\varepsilon), \varepsilon^2)$, where $\delta(\varepsilon)$ is defined in \eqref{A2}. In particular, Assumptions \eqref{A0} with $\langle Z \rangle =0$ and \eqref{A2} show that $\displaystyle \lim_{\varepsilon \to 0} \varepsilon^{-1}\eta(\varepsilon) = 0$. We also denote 
$$\Tilde{w}_{\varepsilon}^{\eta} = \sum _{i \in \{1,...,d\} } \sum_{k \in B_{\frac{1 \mathstrut}{\eta(\varepsilon) \mathstrut}}} (Z_k)_i u_i(. -k).$$
We now study the convergence of $\langle \mathcal{V}(./\varepsilon) .\nabla W_{\varepsilon,R}(./\varepsilon) , g \rangle = \displaystyle \int_{B_R} \mathcal{V}(x/\varepsilon) .\nabla W_{\varepsilon,R}(x/\varepsilon) g(x) dx$ splitting this quantity into three terms. We indeed have $ \displaystyle \langle \mathcal{V}(./\varepsilon) .\nabla W_{\varepsilon,R}(./\varepsilon) , g \rangle  = I^{\varepsilon} - J^{\varepsilon} + K^{\varepsilon}$, where 
\begin{align}
    & I^{\varepsilon} = \langle \mathcal{V}(./\varepsilon).\nabla \Tilde{w}_{\varepsilon}^{\eta}(./\varepsilon), g \rangle, \\
    & J^{\varepsilon} = \langle \mathcal{V}(./\varepsilon), g\rangle . \fint_{B_{4R/\varepsilon}} \nabla \Tilde{w}^{\eta}_{\varepsilon}(y) dy, \\
    & K^{\varepsilon} = \left\langle \mathcal{V}(./\varepsilon). \left( \nabla (\Tilde{w}_{\varepsilon} - \Tilde{w}_{\varepsilon}^{\eta})(./\varepsilon) - \fint_{B_{4R/\varepsilon}} \nabla (\Tilde{w}_{\varepsilon} - \Tilde{w}_{\varepsilon}^{\eta})(y)\right), g \right\rangle. 
\end{align}
In the sequel we denote by $C$ constants independent of $\varepsilon$ which may vary from one line to another.\\

\underline{\textit{Substep 4.1 : Convergence of $I^{\varepsilon}$}}. We remark that $I^{\varepsilon} = \displaystyle \sum_{i,j \in \{1,...,d\}} I_{i,j}^{\varepsilon}$ where 
\begin{equation}
    \label{defI_schro}
    I_{i,j}^{\varepsilon} =  \varepsilon^d  \sum_{k \in \mathbb{Z}^d} \sum_{l \in B_{\frac{1 \mathstrut}{\mathstrut\eta(\varepsilon)\mathstrut}}} (Z_k)_i(Z_l)_j \int_{B_{M}(x_0)/\varepsilon} \varphi(x-k) \partial_i u_j(x-l)dx.
\end{equation}
We first study the convergence of the sequence : 
\begin{equation}
\label{defI_tilde_schro}
    \Tilde{I}^{\varepsilon}_{i,j} = \varepsilon^d  \sum_{k \in \frac{B_{\mathstrut M}(x_0)\mathstrut}{ \mathstrut \varepsilon}} \sum_{l \in B_{\frac{1\mathstrut}{\mathstrut\eta(\varepsilon)}-\frac{D}{\varepsilon}}} (Z_k)_i(Z_{k+l})_j \int_{B_A} \varphi(x) \partial_i u_j(x-l)dx,
\end{equation}
where $D= |x_0| + M > 0$. In the sequel we shall show that $\displaystyle \lim_{\varepsilon \to 0}\Tilde{I}^{\varepsilon}_{i,j} = \lim_{\varepsilon \to 0}I^{\varepsilon}_{i,j}$. We have
\begin{equation}
    \Tilde{I}^{\varepsilon}_{i,j} = M_1^{\varepsilon} + M_2^{\varepsilon},
\end{equation}
where
\begin{align*}
    M_1^{\varepsilon} & = \sum_{l \in B_{\frac{1\mathstrut}{\mathstrut\eta(\varepsilon) }-\frac{D}{\varepsilon}}} \varepsilon^d  \sum_{k \in \frac{B_{\mathstrut M}(x_0)\mathstrut}{\mathstrut \varepsilon}} \left((Z_k)_i(Z_{k+l})_j - |B_M| \ \mathcal{C}_{l,i,j} \right) \int_{B_A} \varphi(x) \partial_i u_j(x-l)dx, \\
    M_2^{\varepsilon} & = \sum_{l \in B_{\frac{1\mathstrut}{\mathstrut\eta(\varepsilon)}-\frac{D}{\varepsilon}}} |B_M| \ \mathcal{C}_{l,i,j} \int_{B_A} \varphi(x) \partial_i u_j(x-l)dx.
\end{align*}
Since $u_j = G*\partial_j \varphi$, Lemma \ref{Prop_elementaire_schro} shows that $|\partial_i u_j(x)| \leq \dfrac{C}{1+|x|^{d}}$ for every $x \in \mathbb{R}^d$. We next use Assumption \eqref{A2} to obtain
\begin{align*}
    \left|M^{\varepsilon}_1\right| \leq C \gamma(\varepsilon) \sum_{l \in B_{\frac{1 \mathstrut}{\mathstrut \eta(\varepsilon) }-\frac{D}{\varepsilon}}}  \int_{B_A} |\varphi(x) \partial_i u_j(x-l)|dx\leq C \ \gamma(\varepsilon)  \sum_{l \in B_{\frac{1 \mathstrut}{\mathstrut \eta(\varepsilon) }}} \dfrac{1}{1+|l|^d} \leq C \gamma(\varepsilon) |\ln\left(\eta(\varepsilon)\right)|.
\end{align*}
Since $|\ln( \eta(\varepsilon))| \leq 2|\ln(\varepsilon)|$, it follows $\left|M^{\varepsilon}_1\right| \leq C \gamma(\varepsilon) |\ln\left(\varepsilon\right)|$ which shows the convergence of $M^{\varepsilon}_1$ to~0 due to assumption \eqref{A2}.

We recall that $\partial_i u_j = \partial_i \partial_j G \ast \varphi$ and, using assumption \eqref{A3}, we can also consider the limit in $M^{\varepsilon}_2$ to obtain 
\begin{align*}
    M^{\varepsilon}_2 \stackrel{\varepsilon \to 0}{\longrightarrow} |B_M| \sum_{l \in \mathbb{Z}^d} \mathcal{C}_{l,i,j} \int_{B_A} \varphi(x) \partial_i u_j(x-l) dx = \left\langle \sum_{l \in \mathbb{Z}^d} \mathcal{C}_{l,i,j} \int_{B_A} \varphi(x) \partial_i u_j(x-l) dx , g \right\rangle.
\end{align*}
We have proved that 
\begin{equation}
   \sum_{i,j \in \{1,...,d\}} \Tilde{I}^{\varepsilon}_{i,j} \stackrel{\varepsilon \to 0}{\longrightarrow} \left\langle  \Tilde{\mathcal{M}} , g \right\rangle.
\end{equation}

We next show that $\displaystyle \lim_{\varepsilon \to 0}\Tilde{I}^{\varepsilon}_{i,j} = \lim_{\varepsilon \to 0}I^{\varepsilon}_{i,j}$ (where $I^{\varepsilon}_{i,j}$ and $\Tilde{I}^{\varepsilon}_{i,j}$ are respectively given by \eqref{defI_schro} and \eqref{defI_tilde_schro}).
We denote by $A>0$ a radius such that $\operatorname{Supp}(\varphi) \subset B_A$. We remark that if $k \notin B_{M/\varepsilon+A}(x_0/\varepsilon)$, we have $\displaystyle \int_{B_M(x_0)/\varepsilon} \varphi(x-k)\partial_i u_j(x-l) dx = 0$ and if $k \in B_{M/\varepsilon-A}(x_0/\varepsilon)$, 
$$\displaystyle \int_{B_M(x_0)/\varepsilon} \varphi(x-k)\partial_i u_j(x-l) dx  = \int_{\mathbb{R}^d} \varphi(x) \partial_i u_j(x-l+k)dx.$$
We next denote $ \mathbf{C}_{\varepsilon}= B_{\frac{M}{\varepsilon}+A}(\frac{x_0}{\varepsilon})\setminus{B_{\frac{M}{\varepsilon}-A}(\frac{x_0}{\varepsilon})}$ and, using $|\partial_i u_j(x)| \leq C (1+|x|^d)^{-1}$, we have 
\begin{align*}
   \bigg|\varepsilon^d \sum_{k \in \mathbf{C}_{\varepsilon}} \sum_{l \in B_{\frac{1 \mathstrut}{\mathstrut \eta(\varepsilon) }}} (Z_k)_i(Z_l)_j & \int_{B_{\frac{M}{\varepsilon}}(\frac{x_0}{\varepsilon})-k}  \varphi(x)  \partial_i u_j(x-l+k)dx \bigg| \\
   & \leq \varepsilon^d \|Z\|^2_{l^{\infty}}\|\varphi\|_{L^{\infty}(\mathbb{R}^d)} \sum_{k \in \mathbf{C}_{\varepsilon}} \sum_{l \in B_{\frac{1 \mathstrut}{\mathstrut \eta(\varepsilon) }}} \int_{B_A} \left|\partial_i u_j(x-l+k)  \right| dx.\\
   & \leq C \varepsilon^d \sum_{k \in \mathbf{C}_{\varepsilon}} |\ln(\eta(\varepsilon))|.
   \end{align*}
Since $|\mathbf{C}_{\varepsilon}| = O\left(\varepsilon^{1-d}\right)$ and $|\ln(\eta(\varepsilon))|\leq 2 |\ln(\varepsilon)|$, we obtain
 \begin{align*}
 \bigg|\varepsilon^d \sum_{k \in \mathbf{C}_{\varepsilon}} \sum_{l \in B_{\frac{1 \mathstrut}{\mathstrut \eta(\varepsilon) }}} (Z_k)_i(Z_l)_j  \int_{B_{M}(x_0)/\varepsilon-k}  \varphi(x)  \partial_i u_j(x-l+k)dx \bigg|
   & \leq C \varepsilon \left| \ln(\varepsilon) \right| \stackrel{\varepsilon \to 0}{\rightarrow} 0.
\end{align*}

It follows that 
\begin{align*}
  I_{i,j}^{\varepsilon} & = \varepsilon^d  \sum_{k \in \frac{B_{\mathstrut M}(x_0)\mathstrut}{\mathstrut\varepsilon}} \sum_{l \in B_{\frac{1 \mathstrut}{\mathstrut \eta(\varepsilon) }}} (Z_k)_i(Z_{l})_j \int_{B_A} \varphi(x) \partial_i u_j(x-l+k)dx + O\left(\varepsilon |\ln(\varepsilon)|\right) \\
  & = \varepsilon^d  \sum_{k \in \frac{B_{\mathstrut M}(x_0)\mathstrut}{\mathstrut\varepsilon}} \sum_{l \in B_{\frac{1 \mathstrut}{\mathstrut \eta(\varepsilon) }}} (Z_k)_i(Z_{k+l})_j \int_{B_A} \varphi(x) \partial_i u_j(x-l)dx + O\left(\varepsilon |\ln(\varepsilon)|\right) \\
 & = \quad \qquad \qquad \qquad \qquad \mathcal{N}_{i,j}^{\varepsilon}  \hspace{4.7cm} +  O\left(\varepsilon |\ln(\varepsilon)|\right).
\end{align*}

We recall that $D = M + |x_0|$ and for every $k \in B_{M}(x_0)/\varepsilon$, we have $|k| \leq \frac{D}{\varepsilon}$. Since $\eta(\varepsilon)= o\left(\varepsilon\right)$, we have for $\varepsilon$ sufficiently small : 
$$B_{\frac{1}{\eta(\varepsilon)} - \frac{D}{\varepsilon}} \subset B_{\frac{1 \mathstrut}{\mathstrut \eta(\varepsilon) }}(-k) \subset{B_{\frac{1}{\eta(\varepsilon)} + \frac{D}{\varepsilon}}} \quad \text{ and } \quad  \left(B_{\frac{1 \mathstrut}{\mathstrut \eta(\varepsilon) }}(-k)\setminus{B_{\frac{1\mathstrut}{\mathstrut\eta(\varepsilon)} - \frac{D}{\varepsilon}}} \right)\subset \left(B_{\frac{1 \mathstrut}{\mathstrut \eta(\varepsilon)} + \frac{D}{\varepsilon}}\setminus{B_{\frac{1 \mathstrut}{\mathstrut\eta(\varepsilon)} - \frac{D}{\varepsilon}}}\right)  .$$ 
We next split $\mathcal{N}_{i,j}^{\varepsilon} $ in two parts : 
\begin{align*}
    \mathcal{N}_{i,j}^{\varepsilon}  & = \varepsilon^d  \sum_{k \in \frac{B_{\mathstrut M}(x_0)\mathstrut}{\mathstrut\varepsilon}}\sum_{l \in B_{\frac{1 \mathstrut}{\mathstrut \eta(\varepsilon)}- \frac{D}{\varepsilon}}} (Z_k)_i(Z_{k+l})_j \int_{B_A} \varphi(x) \partial_i u_j(x-l)dx \\
    & + \varepsilon^d  \sum_{k \in \frac{B_{\mathstrut M}(x_0)\mathstrut}{\mathstrut\varepsilon}} \sum_{l \in  \left( B_{\frac{1 \mathstrut}{\mathstrut \eta(\varepsilon) }}(-k)\setminus{B_{\frac{1}{\eta(\varepsilon)} - \frac{D}{\varepsilon}}}\right)} (Z_k)_i(Z_{k+l})_j \int_{B_A} \varphi(x) \partial_i u_j(x-l)dx \\
    & = \Tilde{I}^{\varepsilon}_{i,j} + R^{\varepsilon}.
\end{align*}

We have 
\begin{align*}
    \left| R^{\varepsilon} \right| & \leq C \|\varphi \|_{L^{\infty}(\mathbb{R}^d)}\|Z\|_{l^{\infty}}^2 \varepsilon^d \sum_{k \in \frac{B_{\mathstrut M}(x_0)\mathstrut}{\mathstrut\varepsilon}} \  \sum_{l \in \left(B_{\frac{1 \mathstrut}{\mathstrut \eta(\varepsilon) }}(-k)\setminus{B_{\frac{1\mathstrut}{\mathstrut\eta(\varepsilon)} - \frac{D}{\varepsilon}}}\right)} \dfrac{1}{1+|l|^d} \\
    & \leq C \varepsilon^d \sum_{k \in \frac{B_{\mathstrut M}(x_0)\mathstrut}{\mathstrut\varepsilon}}  \ \sum_{l \in \left(B_{\frac{1\mathstrut}{\mathstrut\eta(\varepsilon)} + \frac{D}{\varepsilon}}\setminus{B_{\frac{1\mathstrut}{\mathstrut\eta(\varepsilon)} - \frac{D}{\varepsilon}}}\right)} \dfrac{1}{1+|l|^d} \\
    & \leq  C \left(\ln\left( \dfrac{1}{\eta(\varepsilon)}+ \dfrac{D}{\varepsilon}\right) - \ln\left( \dfrac{1}{\eta(\varepsilon)}- \dfrac{D}{\varepsilon}\right) \right) \stackrel{\varepsilon \to 0}{\longrightarrow} 0.
\end{align*}
Since $I^{\varepsilon}_{i,j} = \Tilde{I}^{\varepsilon}_{i,j} + R^{\varepsilon} + O(\varepsilon |\ln(\varepsilon)|)$, we finally conclude that $\displaystyle \lim_{\varepsilon \to 0} I^{\varepsilon}_{i,j} = \lim_{\varepsilon \to 0} \Tilde{I}^{\varepsilon}_{i,j}$.\\

\underline{\textit{Substep 4.2} : Convergence of $J^{\varepsilon}$}. We claim that $J^{\varepsilon}$ converges to 0 when $\varepsilon \to 0$. We indeed remark that 
\begin{align}
    \langle \mathcal{V}(./\varepsilon), g\rangle = \varepsilon^d \sum_{k \in \mathbb{Z}^d} Z_k \int_{B_{M}(x_0)/\varepsilon} \varphi(x-k) dx.  
\end{align}
As in the previous substep, if $k \notin B_{M/\varepsilon+A}(x_0/\varepsilon)$, we have $\varphi(x-k) = 0 $ for every $x \in B_{M}(x_0)/\varepsilon$ and if $k \in B_{M/\varepsilon-A}(x_0/\varepsilon)$, we have $\displaystyle \int_{B_M(x_0)/\varepsilon} \varphi(x-k) dx = \int_{\mathbb{R}^d} \varphi(x) dx$. Since the number of $k\in \mathbb{Z}^d$ such that $k\in B_{M/\varepsilon+A}(x_0/\varepsilon)$ and $k \notin B_{M/\varepsilon-A}(x_0/\varepsilon)$ is proportional to $\varepsilon^{1-d}$, we have
\begin{align}
\label{inegalite1}
    \left|\langle \mathcal{V}(./\varepsilon), g\rangle \right|&= \left|\varepsilon^d \sum_{k \in \frac{B_{\mathstrut M}(x_0)\mathstrut}{\mathstrut\varepsilon}} Z_k \int_{B_A} \varphi(x) dx + O(\varepsilon) \right| \leq C (\alpha(\varepsilon) + \varepsilon).
\end{align}
In addition, since $\nabla \Tilde{w}^{\eta}_{\varepsilon}$ is uniformly bounded in $BMO(\mathbb{R}^d)$ as a consequence of \eqref{Continuity_BMO_L}, we have (see \cite[Chapter IV, Section 1]{stein1993harmonic}) : 
$$ \left| \fint_{B_{4R/\varepsilon}} \nabla \Tilde{w}^{\eta}_{\varepsilon}(y) dy -  \fint_{B_{\frac{1 \mathstrut}{\mathstrut \eta(\varepsilon) }}} \nabla \Tilde{w}^{\eta}_{\varepsilon}(y) dy\right| \leq C \left|\ln\left( \dfrac{\eta(\varepsilon)}{\varepsilon}\right)\right|.$$
We recall that $\eta(\varepsilon) \geq \max(\epsilon \alpha(\varepsilon), \varepsilon^2)$, that is $\left|\ln\left( \dfrac{\eta(\varepsilon)}{\varepsilon}\right)\right| \leq \min(|\ln(\alpha(\varepsilon))|, |\ln(\varepsilon)|)$ when $\varepsilon$ is sufficiently small. Using \eqref{inegalite1}, we obtain
\begin{equation}
\label{majoration_moyenne_carre1}
\begin{split}
   \left|\langle \mathcal{V}(./\varepsilon), g\rangle \right|  \Big| \fint_{B_{4R/\varepsilon}} \nabla \Tilde{w}^{\eta}_{\varepsilon}(y) dy -  \fint_{B_{\frac{1 \mathstrut}{\mathstrut \eta(\varepsilon) }}} & \nabla \Tilde{w}^{\eta}_{\varepsilon}(y) dy\Big| \\
   & \leq C (\alpha(\varepsilon) + \varepsilon) \min(|\ln(\alpha(\varepsilon))|, |\ln(\varepsilon)|) \stackrel{\varepsilon \to 0}{\longrightarrow} 0. 
\end{split}
\end{equation}

Next, since $\displaystyle \nabla \Tilde{w}^{\eta}_{\varepsilon} = \nabla^2 G * \left(\sum_{k \in B_{\frac{1 \mathstrut}{\mathstrut \eta(\varepsilon) }}} Z_k \varphi(.-k)\right) =: T\left(\sum_{k \in B_{\frac{1 \mathstrut}{\mathstrut \eta(\varepsilon) }}} Z_k\varphi(.-k)\right)$, the continuity from $\left(L^2(\mathbb{R}^d)\right)^d$ to $\left(L^2(\mathbb{R}^d)\right)^d$ of the operator $T$ (see \cite[Section 7.2.3]{MR3099262}) yields a constant $C>0$ such that : 
$$\|\nabla \Tilde{w}^{\eta}_{\varepsilon}\|_{L^2(\mathbb{R}^d)} \leq C \left\|\sum_{k \in B_{\frac{1 \mathstrut}{\mathstrut \eta(\varepsilon) }}} Z_k \varphi(.-k) \right\|_{L^2(\mathbb{R}^d)} \leq  C\dfrac{1}{\left( \eta(\varepsilon)\right)^{d/2}}\|\mathcal{V}\|_{L^{\infty}(\mathbb{R}^d)}.$$
The Cauchy-Schwarz inequality therefore gives
\begin{equation}
\label{inequality_schro_moyenne_pb_lin1}
    \left| \fint_{B_{\frac{1 \mathstrut}{\mathstrut \eta(\varepsilon) }}} \nabla \Tilde{w}^{\eta}_{\varepsilon}(y) dy\right| \leq  \left| \fint_{B_{\frac{1 \mathstrut}{\mathstrut \eta(\varepsilon) }}} |\nabla \Tilde{w}^{\eta}_{\varepsilon}(y)|^2 dy\right|^{1/2} \leq C \left(\eta(\varepsilon)\right)^{d/2}\|\nabla \Tilde{w}_{\varepsilon}^{\eta}\|_{L^2(\mathbb{R}^d)} \leq C.
\end{equation}
Using \eqref{inegalite1} and \eqref{inequality_schro_moyenne_pb_lin1}, we deduce that : 
\begin{equation}
\label{majoration_moyenne_carre2}
    \left|\langle \mathcal{V}(./\varepsilon), g\rangle \right| \left| \fint_{B_{\frac{1 \mathstrut}{\mathstrut \eta(\varepsilon) }}} \nabla \Tilde{w}^{\eta}_{\varepsilon}(y) dy\right| \leq C (\alpha(\varepsilon) + \varepsilon) \stackrel{\varepsilon \to 0}{\longrightarrow} 0.
\end{equation}
To conclude, we use a triangle inequality to bound $J^{\varepsilon}$ : 
\begin{align*}
    |J^{\varepsilon}| & \leq \left|\langle \mathcal{V}(./\varepsilon), g\rangle \right| \left( \left| \fint_{B_{4R/\varepsilon}} \nabla \Tilde{w}^{\eta}_{\varepsilon}(y) dy -  \fint_{B_{\frac{1 \mathstrut}{\mathstrut \eta(\varepsilon) }}} \nabla \Tilde{w}^{\eta}_{\varepsilon}(y) dy\right| \right)\\
    & + \left|\langle \mathcal{V}(./\varepsilon), g\rangle \right| \left| \fint_{B_{\frac{1 \mathstrut}{\mathstrut \eta(\varepsilon) }}} \nabla \Tilde{w}^{\eta}_{\varepsilon}(y) dy\right|.
\end{align*}
 \eqref{majoration_moyenne_carre1} and \eqref{majoration_moyenne_carre2} finally show that $J^{\varepsilon} \to 0$ when $\varepsilon \to 0$. \\

\underline{\textit{Substep 4.3 : Convergence of $K^{\varepsilon}$}}. Using that $\varphi$ is supported in $B_A$ and proceeding as in the previous steps, we can show that the convergence of $K^{\varepsilon}$ is equivalent to the convergence of 
\begin{equation}
\label{K_schro}
\begin{split}
    \Tilde{K}^{\varepsilon}_{i,j} = \varepsilon^d \hspace{-0.7pt}\sum_{k \in \frac{B_{\mathstrut M}(x_0)\mathstrut}{\mathstrut\varepsilon}} \sum_{\frac{1}{\eta(\varepsilon)}<|l-k|<\frac{1}{\beta(\varepsilon)}} \hspace{-0.5cm}(Z_k)_i(Z_{k+l})_j \int_{B_A} \hspace{-0.3cm}\varphi(x) & \Bigg(\partial_i u_j(x-l) - \fint_{B_{4R/\varepsilon}} \hspace{-0.5cm}\partial_i u_j(y-l-k) dy\Bigg) dx,
\end{split}
\end{equation}
for every $i,j \in \{1,...,d\}$. Since $u_j = G*\partial_j \varphi$, we remark that for every $x \in B_A$, $y\in B_{4R/\varepsilon}$, $k \in B_M(x_0)/\varepsilon$ and $l$ such that $|l-k|>\frac{1}{\eta(\varepsilon)}$, we have 
\begin{align}
    \partial_i u_j&(x-l) - \partial_i u_j(y-l-k)  = \int_{B_A} \left(\partial_i \partial_j G(z-x+l) - \partial_i \partial_j G(z-y+l+k)\right) \varphi(z) dz \\
    & = \int_{B_A} \left(\partial_i \partial_j G(z-x+l) - \partial_i \partial_j G(l) + \partial_i \partial_j G(l) - \partial_i \partial_j G(z-y+l+k)\right) \varphi(z) dz. 
\end{align}
The results of \cite[Lemma 7.18 p.151]{MR3099262} yield the existence of $C>0$ such that for every $u,v \in \mathbb{R}^d$ with $|u|>2|v|$, we have 
\begin{equation}
    \left|\partial_i \partial_j G(u-v) - \partial_i \partial_j G(u) \right| \leq C \dfrac{|v|}{|u|^{d+1}}.
\end{equation}
Since $\eta(\varepsilon) = o(\varepsilon)$, for every $z\in B_A$ and $x,y,k,l$ as above, we have $|l|>2|z-y+k|$ and $|l|>2|z-x|$ when $\varepsilon$ is sufficiently small. It follows
\begin{equation}
\label{majoriation_K1_schro}
  \fint_{B_{4R/\varepsilon}}  \left| \partial_i u_j(x-l) - \partial_i u_j(y-l-k) \right| dy \leq C \fint_{B_{4R/\varepsilon}} \left(\dfrac{|x|}{|l|^{d+1}}+ \dfrac{|y|}{|l|^{d+1}}\right)dy \leq C \dfrac{\varepsilon^{-1}}{|l|^{d+1}}.  
\end{equation}
If we now insert \eqref{majoriation_K1_schro} into \eqref{K_schro}, since $Z_k$ and $\varphi$ are uniformly bounded, we obtain 
\begin{equation}
    \left|\Tilde{K}^{\varepsilon}_{i,j}\right| \leq C \varepsilon^{-1} \sum_{\frac{1}{\eta(\varepsilon)}<|l|} \dfrac{1}{|l|^{d+1}} \leq C \varepsilon^{-1} \eta(\varepsilon).
\end{equation}
By definition of $\eta(\varepsilon)$, we have $\displaystyle \lim_{\varepsilon \to 0} \varepsilon^{-1} \eta(\varepsilon) =0$. We therefore conclude that $\displaystyle \lim_{\varepsilon \to 0} K^{\varepsilon}  =0$.

\underline{\textit{Substep 4.3 : Conclusion}}. In the three substeps above, we have finally shown that
\begin{equation}
    \langle \mathcal{V}(./\varepsilon) .\nabla \Tilde{W}_{\varepsilon,R}(./\varepsilon) , g \rangle \stackrel{\varepsilon \to 0}{\longrightarrow} \left\langle \Tilde{\mathcal{M}} , g \right\rangle,
\end{equation}
for every characteristic function $g$ defined on $B_R$. 
Using \eqref{weak_convergence_square_schro} and the density of simple functions in $L^p(B_R)$ for every $p>1$, we deduce that 
$$\operatorname{weak}\lim_{\varepsilon \to 0} |\nabla \Tilde{W}_{\varepsilon,R}|^2(./\varepsilon) = \operatorname{weak}\lim_{\varepsilon \to 0} \mathcal{V}.\nabla \Tilde{W}_{\varepsilon, R}(./\varepsilon) = \Tilde{\mathcal{M}},$$
which concludes the proof of Proposition \ref{prop_correcteur_cas_lin_schro}.
\end{proof}

\begin{remark}
Given the only assumptions \eqref{A0}-\eqref{A00}-\eqref{A2}-\eqref{A3}, we cannot expect to show a uniform bound with respect to $\varepsilon$ for $\nabla \Tilde{w}_{\varepsilon}(./\varepsilon)$ without subtracting a constant that depends on $\varepsilon$ as in \eqref{corrector_schro_BMO}. For $d=3$, consider indeed $Z_k = \dfrac{1}{\ln(2+|k|)}\left(1, 1, 1\right)$ that satisfies our assumptions. When $\displaystyle \int_{\mathbb{R}^3} \varphi \neq 0$, $\partial_j u_i(x)$ behaves as $\dfrac{1}{|x|^3}$ at infinity and $\displaystyle \sum_{|k|\leq \frac{1}{\beta(\varepsilon)}} (Z_k)_i \partial_j u_i(x/\varepsilon)$ growths at least as $|\ln(\ln(\beta(\varepsilon))|$ on every compact. This phenomenon is related to the non continuity of the operator $T : f \mapsto \nabla^2 G \ast f$ from $\left(L^{\infty}(\mathbb{R}^d)\right)^d$ to $\left(L^{\infty}(\mathbb{R}^d)\right)^d$. 
\end{remark}

\begin{remark}
In Proposition \ref{prop_correcteur_cas_lin_schro}, we note that the choice $\beta(\varepsilon) = 0$ is also admissible since the bound given by \eqref{Continuity_BMO_L} allows to show that the sum $ \displaystyle \sum_{k \in \mathbb{Z}^d} (Z_k)_i \nabla u_i(.-k)$ makes sense in $\left(BMO(\mathbb{R}^d)\right)^d$. However, we have shown that Proposition \ref{prop_correcteur_cas_lin_schro} holds under the less restrictive assumption $\displaystyle \lim_{\varepsilon \to 0} \varepsilon^{-1} \beta(\varepsilon) =0$, that is for approximation of the full sum. 
\end{remark}


\subsection{Some particular cases}
\label{section_linear_particular_case}

We introduce here some particular settings in terms of the function $\varphi$ and the sequence $Z$ for which the study of \eqref{correcteur_lienaire_schro} is simpler and several stronger results can be shown. In particular, we show that under additional assumptions on $\varphi$ and $Z$, the sum $ \displaystyle \sum_{k \in \mathbb{Z}^d} (Z_k)_i \nabla u_i(.-k)$ converges in $\left(L^{\infty}_{loc}(\mathbb{R}^d)\right)^d$ for all $1\leq i \leq d$, which shows the existence of a solution $w_1$ to \eqref{correcteur_lienaire_schro} that is independent of $\varepsilon$. 

\textbf{1)} \textit{Vanishing integral $\displaystyle \int_{\mathbb{R}^d} \varphi = 0$}.\\
Given this assumption, the Green formula shows $\displaystyle \int_{\mathbb{R}^d}x \partial_i \varphi =~0$ and, for every $i \in \{1,...,d\}$,  $\nabla u_i = \nabla \left(G \ast  \partial_i \varphi\right)$ satisfies $\displaystyle |\nabla u_i (x)| \leq \dfrac{M}{1+|x|^{d+1}}$ as a consequence of Lemma~\ref{Prop_elementaire_schro}. It follows that the sums $ \displaystyle \sum_{k \in \mathbb{Z}^d} (Z_k)_i \nabla u_i(.-k)$ normally converge in $\left(L^{\infty}(\mathbb{R}^d)\right)^d$, for all $1 \leq i \leq d$. The Schwarz lemma ensures its limit is a gradient, which shows the existence of $w_1$, solution to \eqref{correcteur_lienaire_schro} on $\mathbb{R}^d$ such that $\nabla w_1 \in \left(L^{\infty}(\mathbb{R}^d)\right)^d$ (and not only in $\left(BMO(\mathbb{R}^d)\right)^d$). 

Since $\mathcal{V}(./\varepsilon)$ also weakly converges to 0 as $\varepsilon$ vanishes without any assumption on $Z$ (see Remark~\ref{remark_average_intphi_vanish}), we have $\displaystyle \lim_{\varepsilon \to 0} \nabla w_1(./\varepsilon) =0$ in $L^{\infty}(\mathbb{R}^d)-\star$. In addition, Assumption \eqref{A00} is sufficient to show the weak convergence of $|\nabla w_1(./\varepsilon)|^2$ to a constant. For the proof of these assertions, we refer to the study of equation \eqref{schro_eq_corrector_nonlin} in Section \ref{Section_non_linear_equation}, where we perform a similar proof.

\textbf{2)} \textit{Existence of a partition of unity}. \\
If there exists a function $\chi \in \mathcal{D}(\mathbb{R}^d)$ such that $\displaystyle \int_{\mathbb{R}^d} \chi = 1$ and $\displaystyle \sum_{k \in \mathbb{Z}^d} \chi(.-k-Z_k)=1$, the potential $V$ defined by \eqref{potential_form_schro} satisfies $V = W_{per} + W$,
where $\displaystyle W_{per} = g_{per} + \int_{\mathbb{R}^d} \varphi$ is periodic with a vanishing average as of consequence of \eqref{Schro_condition_moyenne2} and $W = \displaystyle \sum_{k \in \mathbb{Z}^d} \psi(.-k-Z_k)$ where $\psi = \displaystyle \varphi - \left(\int_{\mathbb{R}^d} \varphi\right) \chi \in \mathcal{D}(\mathbb{R}^d)$ and $\displaystyle \int_{\mathbb{R}^d} \psi =0$. The first setting \textbf{1)} above thus shows the existence of a corrector independent of $\varepsilon$. However, the existence of a partition of unity $\chi$ only exceptionally exists for a sequence $Z$. For instance, a counter-example is given by $Z_0 \neq 0$ and $Z_k = 0$ if $k\neq 0$, which clearly satisfies $\eqref{A0}$ to \eqref{A4}. If there exists $\chi\in \mathcal{D}(\mathbb{R}^d)$ such that $1 \equiv \displaystyle \sum_{k \in \mathbb{Z}^d} \chi(.-k-Z_k)$, using that $\chi$ is compactly supported, we have 
\begin{align*}
    1 = \lim_{|x| \to \infty} \chi(x+Z_0) + \sum_{k \in \mathbb{Z}^d} \chi(x-k) - \chi(x)   = \lim_{|x| \to \infty} \sum_{k \in \mathbb{Z}^d} \chi(x-k).
\end{align*}
For every $x \in \mathbb{R}^d$, we obtain $\displaystyle \sum_{k \in \mathbb{Z}^d} \chi(x-k) = 1$ by periodicity. It follows $\chi = \chi(.+Z_0)$, which is a contradiction since $Z_0 \neq 0$ and $\chi \neq 0$ is compactly supported. 



\textbf{3)} \textit{Additional assumptions related to the distribution of $Z$}.\\
As in \eqref{assumption_convergence_average_schro}, if the convergence of $ \displaystyle \dfrac{\varepsilon^d}{|B_R|} \sum_{k \in B_R(x_0)/\varepsilon} Z_k$ to 0 is sufficiently fast, uniformly with respect to the center $x_0$, Proposition~\ref{Prop_schro_convergence_average} shows the convergence of the sum $\displaystyle \sum_{k\in \mathbb{Z}^d} (Z_k)_i \partial_i \partial_j G \star \varphi(x-k)$. In this case, the subtraction of a $\varepsilon$-dependent constant $C_{\varepsilon}$ as in the proof of Proposition \ref{prop_correcteur_cas_lin_schro} is therefore not required to obtain the existence of a solution $w_1$ to \eqref{correcteur_tronque}.

On the other hand, the convergence of the sum could also be established if $Z_k$ is the discrete gradient of a sequence $(T_k)_{k\in \mathbb{Z}^d}$, that is, if $(Z_k)_i = \delta_i T_k = T_{k+e_i} - T_k$ for every $i \in \{1,...d\}$. Formally, the idea would be to perform a summation by parts to obtain 
$$\sum_{k\in \mathbb{Z}^d} (Z_k)_i \partial_i \partial_j G \star \varphi(x-k) \sim \sum_{k\in \mathbb{Z}^d}  \dfrac{\delta_i T_k}{|x-k|^d} \sim \sum_{k\in \mathbb{Z}^d} T_k \left(\dfrac{1}{|x-k+e_i|^d} - \dfrac{1}{|x-k|^d}\right). $$
Since $\dfrac{1}{|x-k+e_i|^d} - \dfrac{1}{|x-k|^d}\leq \dfrac{C}{|x-k|^{d+1}}$, we could show the convergence as soon as $|T_k| =O\left( |k|^{\alpha}\right)$ for $\alpha \in [0,1[$. This property is again very specific and, in particular, $Z$ is a discrete gradient if and only if it satisfies a discrete Cauchy equation given by $\delta_i (Z_k)_j = \delta_j (Z_k)_j$ for every $i,j \in \{1,...d\}$, which is not always true for a generic vector-valued sequence.

\section{Corrector equation : the full equation (\ref{corrector_general_schro})}

\label{Schro_section4}

In this section we return to our original problem \eqref{corrector_general_schro} when $V$ is the general potential given by \eqref{potential_form_schro} and we prove Theorem \ref{theorem1_schro}. The existence of a corrector is performed in two steps. In Proposition~\ref{prop_corrector_nonlin}, we first prove the existence of a specific solution to equation \eqref{schro_eq_corrector_nonlin} associated with the nonlinear higher order terms of the Taylor expansion. We next conclude with the proof of Theorem \ref{theorem1_schro}.


\subsection{Preliminary properties of convergence}

We establish here a preliminary property related to the weak convergence of the functions of the form $\displaystyle \sum_{k \in \mathbb{Z}^d} G(Z_k,x)$ when each $ G(Z_k,x)$ behaves as $\dfrac{1}{|x|^{d+\alpha}}$ at infinity, for $\alpha >0$. To this end, we first consider compactly supported functions in Lemma \ref{Schro_convergence_carre_supp} and we conclude in Lemma \ref{corol_convergence_schro} using an argument of density given in Lemma \ref{Schro_lemme_densite}. 

\begin{lemme}
\label{Schro_convergence_carre_supp}
Assume $Z$ satisfies \eqref{hyp_borne_Z_schro} and \eqref{A4}. Let $G(x,y)$ and $H(x,y)$ be two continuous functions on $\mathbb{R}^d \times \mathbb{R}^d$, compactly supported with respect to $y$. We denote by
$$g(x) = \sum_{k,l \in \mathbb{Z}^d} G(Z_k, x-k)H(Z_{l},x-l) \in L^{\infty}(\mathbb{R}^d).$$
Then, when $\varepsilon$ converges to 0, $g(./\varepsilon)$ weakly converges to $\displaystyle \kappa = \sum_{l \in \mathbb{Z}^d} C_{F_l,l}$ in $L^{\infty}(\mathbb{R}^d)-\star$, where $C_{F_l,l}$ is the constant given by Assumption \eqref{A4} for $\displaystyle F_l(x,y) = \int_{\mathbb{R}^d} G(x,z)H(y,z-l) dz$. 
\end{lemme}

\begin{proof}
For $R>0$ and $x_0 \in \mathbb{R}^d$, we consider $h = 1_{B_{R}(x_0)}$ and we first show that
$$\lim_{\varepsilon \to 0} \langle g(./\varepsilon),h\rangle_{L^{\infty}, L^1} = \kappa \int_{\mathbb{R}^d} h.$$
For every $\varepsilon>0$, we have 
$$\langle g(./\varepsilon), h \rangle_{L^{\infty}, L^1} = \sum_{k,l \in \mathbb{Z}^d} \varepsilon^d \int_{B_{\frac{R}{\varepsilon}\left(\frac{x_0}{\varepsilon}\right)}} G(Z_k, x-k) H(Z_l,x-l) dx.$$
We consider a radius $A>0$ such that $G(x,y) = H(x,y) = 0$ for every $x \in \mathbb{R}^d$ and $y \notin B_A$. If $k \notin B_{\frac{x_0}{\varepsilon} +A}(\frac{x_0}{\varepsilon})$, we clearly have $\displaystyle \int_{B_{\frac{R}{\varepsilon}\left(\frac{x_0}{\varepsilon}\right)}} G(Z_k, x-k) H(Z_l,x-l) dx =0$. On the other hand, if $k \in B_{\frac{x_0}{\varepsilon} -A}(\frac{x_0}{\varepsilon})$, we have $\displaystyle \int_{B_{\frac{R}{\varepsilon}\left(\frac{x_0}{\varepsilon}\right)}} G(Z_k, x-k) H(Z_l,x-l) dx = \int_{\mathbb{R}^d} G(Z_k, x) H(Z_l,x-l+k) dx$. Using that $ \displaystyle \left| B_{\frac{x_0}{\varepsilon} -A}\left(\frac{x_0}{\varepsilon}\right)\setminus{B_{\frac{x_0}{\varepsilon}}\left(\frac{x_0}{\varepsilon}\right)}\right| = O(\varepsilon^{1-d})$ and the change of variable $l \mapsto l-k$, it follows that
$$ \langle g(./\varepsilon), h \rangle_{L^{\infty}, L^1} = \sum_{l \in \mathbb{Z}^d} \int_{\mathbb{R}^d} \varepsilon^d \sum_{k \in B_{\frac{R}{\varepsilon}\left(\frac{x_0}{\varepsilon}\right)}} G(Z_{k},x) H(Z_{k+l},x-l) dx  + O(\varepsilon).$$
Assumption \eqref{A4} therefore shows that 
$$ \lim_{\varepsilon \to 0} \langle g(./\varepsilon),h\rangle_{L^{\infty}, L^1} = \kappa |B_R|=  \kappa \int_{\mathbb{R}^d} h.$$
We conclude using the density of simple functions in $L^1(\mathbb{R}^d)$. 
\end{proof}

The next Lemma is an elementary result of density. We skip its proof for the sake of brevity. 

\begin{lemme}
\label{Schro_lemme_densite}
Let $(F_k)_{k \in \mathbb{Z}^d} \in \left(L^{\infty}(\mathbb{R}^d)\right)^{\mathbb{Z}^d}$ such there exists $\alpha >0$ and $C>0$ satisfying 
\begin{equation}
    |F_k(x)| \leq \dfrac{C}{1+|x|^{d+\alpha}} \quad \forall k \in \mathbb{Z}^d, \ \forall x \in \mathbb{R}^d.
\end{equation}
Then the sequence $\displaystyle f_N = \sum_{k \in \mathbb{Z}^d} F_k(.-k)1_{B_N}(.-k)$ converges to $\displaystyle f = \sum_{k \in \mathbb{Z}^d} F_k(.-k)$ in $L^{\infty}(\mathbb{R}^d)$ when $N$ tends to $+\infty$.
\end{lemme}


As a direct consequence of Lemmas \ref{Schro_convergence_carre_supp} and \ref{Schro_lemme_densite}, we therefore obtain : 

\begin{lemme}
\label{corol_convergence_schro}
Assume $Z$ satisfies \eqref{hyp_borne_Z_schro} and \eqref{A4}. Let $G(x,y)$ and $H(x,y)$ be two continuous functions on $\mathbb{R}^d \times \mathbb{R}^d$ such that 
$$\exists \alpha , \beta >0, \ \forall x,y \in \mathbb{R}^d, \quad |G(x,y)| \leq \dfrac{M_1}{1+|y|^{d+\alpha}}, \quad |H(x,y)|\leq \dfrac{M_2}{1+|y|^{d+\beta}},$$
where 
$M_1>0$ and $M_2>0$ are two constants independent of $x$ and $y$. 
We denote
$$g(x) = \sum_{k,l \in \mathbb{Z}^d} G(Z_k, x-k)H(Z_{l},x-l) \in L^{\infty}(\mathbb{R}^d).$$
Then, the conclusion of Lemma \ref{Schro_convergence_carre_supp} holds true.
\end{lemme}

\subsection{Existence result for equation \ref{c_schro}}

\label{Section_non_linear_equation}

We are now in point to study the equation
\begin{equation}
\label{schro_eq_corrector_nonlin}
    \Delta w_2 = \sum_{k \in \mathbb{Z}^d} Z_k^T \left( \int_{0}^1 (1-t) D^2\varphi(.-k-tZ_k)dt \right) Z_k.
\end{equation}
Using Lemma \ref{Prop_elementaire_schro}, the existence of $w_2$ is actually easier to establish than that of $w_1$ solution to \eqref{correcteur_lienaire_schro}. On the other hand, the nonlinearity with respect to $Z$ on the right-hand side of \eqref{schro_eq_corrector_nonlin} requires \eqref{A4} to show the convergence of $|\nabla w_2|^2(./\varepsilon)$.

\begin{prop}
\label{prop_corrector_nonlin}
Assume $d\geq 2$. Let $\varphi \in \mathcal{D}(\mathbb{R}^d)$ and $Z$ satisfying \eqref{hyp_borne_Z_schro} and \eqref{A4}. For $i,j \in \{1,...,d\}$, we denote $u_{i,j} = G \ast \partial_{i} \partial_j \varphi$. Then there exists a solution $w_2 \in L^1_{loc}(\mathbb{R}^d)$ to \eqref{schro_eq_corrector_nonlin} such that
$$ \nabla w_2 = \sum _{i,j \in \{1,...,d\} } \sum_{k \in \mathbb{Z}^d} (Z_k)_j(Z_k)_i \int_{0}^1 (1-t)\nabla u_{i,j}(.-k-tZ_k)dt.$$
In addition, $\nabla w_2 \in \left(L^{\infty}(\mathbb{R}^d)\right)^d$ and, when $\varepsilon$ converges to 0, $\nabla w_2(./\varepsilon)$ weakly converges to 0 in $L^{\infty}(\mathbb{R}^d)-\star$, $\varepsilon w_2(./\varepsilon)$ strongly converges to 0 in $L^{\infty}_{loc}(\mathbb{R}^d)$ and there exists $\mathcal{M}_2 \in \mathbb{R}$ such that 
\begin{equation}
    |\nabla w_{2}|^2(./\varepsilon) \stackrel{\varepsilon \to 0}{\longrightarrow} \mathcal{M}_2 \quad \text{ in } L^{\infty}(\mathbb{R}^d)-\star.
\end{equation}
\end{prop}

\begin{proof}
We first remark that a simple application of the Green formula shows $\displaystyle \int_{\mathbb{R}^d} \partial_{i} \partial_j \varphi(x) dx =~0$ and $\displaystyle \int_{\mathbb{R}^d} x \partial_{i} \partial_j \varphi(x) dx =0$ for every $i,j \in \{1,...,d\}$. Lemma \ref{Prop_elementaire_schro} therefore ensures the existence of $M>0$ such that 
\begin{equation}
\label{preuve_borne_ui_schro}
    |\nabla u_{i,j}(x)| \leq \dfrac{M}{1+|x|^{d+1}} \quad \forall x \in \mathbb{R}^d.
\end{equation}
For every $N \in \mathbb{N}$, we denote 
$$S_N = \sum _{i,j \in \{1,...,d\} } \sum_{|k|\leq N} (Z_k)_j(Z_k)_i \int_{0}^1 (1-t) u_{i,j}(.-k-tZ_k)dt,$$
solution to 
\begin{equation}
\label{equation_reste_tronque}
    \Delta S_N = \sum_{|k|\leq N} Z_K^T \left(\int_{0}^1 (1-t) D^2\varphi(.-k-tZ_k)dt \right)Z_k.
\end{equation}
Since $Z$ belongs to $\left(l^{\infty}(\mathbb{Z}^d)\right)^d$ and $\nabla u_{i,j}$ satisfies \eqref{preuve_borne_ui_schro}, the sum $\nabla S_N$ normally converges in $L^{\infty}_{loc}(\mathbb{R}^d)$ and its limit is a gradient $\nabla w_2$ as a consequence of the Schwarz lemma. Passing to the limit in \eqref{equation_reste_tronque} when $N \to \infty$, we show that $w_2$ is a solution to \eqref{schro_eq_corrector_nonlin}. Moreover, for every $q \in \mathbb{Z}^d$, we have 
$$\nabla w_2(.+q) = \sum _{i,j \in \{1,...,d\} } \sum_{k \in \mathbb{Z}^d} (Z_{k-q})_j(Z_{k-q})_i \int_{0}^1 (1-t) \nabla u_{i,j}(.-k-tZ_{k-q})dt.$$
Again, using \eqref{preuve_borne_ui_schro} and the fact that $Z_k$ is bounded uniformly with respect to $k$, we obtain
$$\forall q \in \mathbb{Z}^d, \quad \|\nabla w_2(.+q)\|_{L^{\infty}(B_2)} \leq C(d) \|Z\|_{l^{\infty}}^2 \sum_{k \in \mathbb{Z}^d} \dfrac{1}{1+|k|^{d+1}},$$
where $C(d)>0$ depends only on the dimension $d$. Since the right-hand side is independent of $q$, we deduce that $\nabla w_2$ belongs to $\left(L^{\infty}(\mathbb{R}^d)\right)^d$.

We have shown that the sequence $\nabla w_2(./\varepsilon)$ is uniformly bounded with respect to $\varepsilon$ in $\left(L^{\infty}(\mathbb{R}^d)\right)^d$. Up to an extraction, it therefore converges to a gradient $\nabla v$ for the weak-$\star$ topology of $\left(L^{\infty}(\mathbb{R}^d)\right)^d$ when $\varepsilon \to 0$. Moreover, since $w_2$ is solution to \eqref{schro_eq_corrector_nonlin}, we have for every $\varepsilon>0$ and $\psi \in \mathcal{D}(\mathbb{R}^d)$ :
\begin{equation}
\begin{split}
    \int_{\mathbb{R}^d} \nabla w_2(x/\varepsilon).\nabla \psi(x) dx &= - \int_{\mathbb{R}^d} \left( \sum_{i,j} \sum_{k\in \mathbb{Z}^d} (Z_k)_i(Z_k)_j \int_0^1 (1-t) \partial_j \partial_i \varphi(x/\varepsilon - k - tZ_k)dt\right)  \psi(x) dx. \\
    &= \int_{\mathbb{R}^d} \left( \sum_{i,j} \sum_{k\in \mathbb{Z}^d} (Z_k)_i(Z_k)_j \int_0^1 (1-t) \partial_i \varphi(x/\varepsilon - k - tZ_k)dt\right) \partial_j \psi(x) dx.
\end{split}
\end{equation}
We can next show that $\displaystyle \sum_{k\in \mathbb{Z}^d} (Z_k)_i(Z_k)_j \int_0^1 (1-t) \partial_i \varphi(x/\varepsilon - k - tZ_k)dt$ weakly converges to 0 using that $\displaystyle \int_{\mathbb{R}^d} \partial_{i} \varphi = 0$ for every $i \in \{1,...,d\}$ and proceeding exactly as in the proof of Lemma~\ref{prop_convergence_primlitiv_pot}. When $\varepsilon \to 0$, it follows that 
\begin{equation}
    \int_{\mathbb{R}^d} \nabla v. \nabla \psi = 0,
\end{equation}
that is $\Delta v = 0$ in $\mathcal{D}'(\mathbb{R}^d)$. Since $\nabla v \in \left( L^{\infty}(\mathbb{R}^d) \right)^d,$ we obtain $\nabla v=0$. We deduce that $\nabla w(./\varepsilon)$ converges to 0 in $L^{\infty}(\mathbb{R}^d)-\star$ and, as a consequence of Lemma \ref{lemme_sous_linearité}, that $\varepsilon w_{2}(./\varepsilon)$ converges to 0 in $L^{\infty}_{loc}(\mathbb{R}^d)$. 

Finally, since 
$$\nabla w_2 (x) = \sum _{i,j \in \{1,...,d\} } \sum_{k \in \mathbb{Z}^d} (Z_{k})_j(Z_{k})_i \int_{0}^1 (1-t) \nabla u_{i,j}(x-k-tZ_{k})dt,$$
and $\nabla u_{i,j}$ satisfies \eqref{preuve_borne_ui_schro},
the weak convergence of $|\nabla w_2(./\varepsilon)|^2$ to a constant is a direct consequence of \eqref{hyp_borne_Z_schro}, \eqref{A4} and Lemma \ref{corol_convergence_schro}. 
\end{proof}

\subsection{Proof of Theorem \ref{theorem1_schro}}

Problems \eqref{correcteur_tronque} and \eqref{schro_eq_corrector_nonlin} being solved, we are in position to prove Theorem \ref{theorem1_schro}. 

\begin{proof}[Proof of Theorem \ref{theorem1_schro}]
For every $\varepsilon>0$ and $R>0$, we define $W_{\varepsilon,R} = w_{per} - \Tilde{W}_{\varepsilon,R} + w_{2}$ where $w_{per}$ is the periodic solution (unique up to the addition of a constant) to $\Delta w_{per} = \displaystyle g_{per} + \sum_{k\in \mathbb{Z}^d} \varphi(.-k)=V_{per}$ and $\Tilde{W}_{\varepsilon,R}$ and $w_2$ are respectively defined by Proposition \ref{prop_correcteur_cas_lin_schro} and Proposition \ref{prop_corrector_nonlin}. By linearity, $W_{\varepsilon,R}$ is a solution to \eqref{equ_correcteur_theorem1}. In addition, since $\nabla w_{per}$ is the gradient of a periodic function, we have $\langle \nabla w_{per}\rangle=~0$ and, by periodicity, we know that $w_{per}$ is strictly sublinear at infinity. The properties of $W_{\varepsilon,R}$ and $w_2$ given respectively by Proposition \ref{prop_correcteur_cas_lin_schro} and Proposition \ref{prop_corrector_nonlin} therefore ensure that $\nabla W_{\varepsilon,R}(./\varepsilon)$ weakly converges to 0 in $L^p(B_R)$ for every $p\in[1,+\infty[$ and that $\|\varepsilon W_{\varepsilon,R}(./\varepsilon)\|_{L^{\infty}(B_R)}$ converges to 0 as $\varepsilon$ vanishes. We next turn to the weak convergence of $|\nabla W_{\varepsilon,R}(./\varepsilon)|^2$. We first remark that 
$$|\nabla W_{\varepsilon,R}|^2 = |\nabla w_{per}|^2 + |\nabla \Tilde{W}_{\varepsilon,R}|^2+ |\nabla w_{2}|^2-2 \nabla w_{per}. \nabla \Tilde{W}_{\varepsilon,R} - 2  \nabla w_{2}.\nabla \Tilde{W}_{\varepsilon,R} + 2 \nabla w_{per}.\nabla w_{2}.$$
As a consequence of the periodicity of $\nabla w_{per}$ and the results of Propositions \ref{prop_correcteur_cas_lin_schro} and \ref{prop_corrector_nonlin}, we already know that $|\nabla w_{per}(./\varepsilon)|^2$, $|\nabla \Tilde{W}_{\varepsilon,R}(./\varepsilon)|^2$ and $|\nabla w_{2}(./\varepsilon)|^2$ weakly converge as $\varepsilon$ vanishes. We have to show the convergence of the rightmost three terms. We only prove here the convergence of $ \nabla w_{2}(./\varepsilon).\nabla \Tilde{W}_{\varepsilon,R}(./\varepsilon)$, the proof for the other terms being extremely similar. We denote by $\mathcal{V}$ the function defined in~\eqref{primitiv_pot_lienar_schro}.
For $\phi \in \mathcal{D}(B_R)$, multiplying \eqref{correcteur_tronque} by $\chi(x) = w_{2}(x/\varepsilon) \phi( x)$ and integrating, we have 
\begin{align*}
    \int_{B_R} \nabla \Tilde{W}_{\varepsilon,R}(./\varepsilon).\nabla w_{2}(./\varepsilon) \hspace{0.7pt} \phi & = -\varepsilon \int_{B_R} w_{2}(./\varepsilon) \nabla \Tilde{W}_{\varepsilon,R}(./\varepsilon).\nabla  \phi  +  \int_{B_R}  \mathcal{V}(./\varepsilon).\nabla w_{2}(./\varepsilon) \hspace{0.7pt} \phi  \\
    &+ \varepsilon \int_{B_R} w_{2}(./\varepsilon) \mathcal{V}(./\varepsilon).\nabla \phi.
\end{align*}
We know that $\nabla \Tilde{W}_{\varepsilon,R}(./\varepsilon)$ is bounded in $L^2(B_R)$, uniformly with respect to $\varepsilon$, that $\varepsilon w_{2}(./\varepsilon)$ uniformly converges to 0 on $B_R$ and that $\mathcal{V}$ is bounded in $\left(L^{\infty}(\mathbb{R}^d)\right)^d$. It follows that  
$$\lim_{\varepsilon \to 0} \int_{B_R}  \nabla \Tilde{W}_{\varepsilon,R}(./\varepsilon).\nabla w_{2}(./\varepsilon) \hspace{0.7pt} \phi = \lim_{\varepsilon \to 0} \int_{B_R} \mathcal{V}(./\varepsilon).\nabla w_{2}(./\varepsilon) \hspace{0.7pt} \phi .$$ 
In addition, we have
$$\mathcal{V}.\nabla w_{2} = \sum_{i,j,n \in \{1,...,d\}} \sum_{k,l \in \mathbb{Z}^d} (Z_l)_n(Z_k)_j(Z_k)_i \varphi(.-l) \int_{0}^1 (1-t) \partial_n u_{i,j}(.-k-tZ_k)dt ,$$
where $u_{i,j} = G \ast \partial_{i}\partial_j \varphi$ and $|u_{i,j}(x)| \leq \dfrac{M}{1+|x|^{d+1}}$ (where $M>0$) as a consequence of Lemma~\ref{Prop_elementaire_schro}. Under assumptions \eqref{hyp_borne_Z_schro} and \eqref{A4}, Lemma \ref{corol_convergence_schro} therefore shows the existence of a constant $C$ such that $\Tilde{V}(./\varepsilon).\nabla w_{2}(./\varepsilon)$ converges to $C\in \mathbb{R}$ for the weak-$\star$ topology of $L^{\infty}(B_R)$ and we can conclude.  
\end{proof}

\begin{corol}
\label{Schro_corol_convergence_corrector}
Under the assumptions of Theorem \ref{theorem1_schro}, $V(./\varepsilon)W_{\varepsilon,R}(./\varepsilon)$ is bounded in $W^{-1,p}(B_R)$, uniformly with respect to $\varepsilon>0$ and for every $p \in ]1,+\infty[$. In addition, $V(./\varepsilon)W_{\varepsilon,R}(./\varepsilon)$ weakly converges in $W^{-1,p}(B_R)$ to $-\mathcal{M}$ as $\varepsilon$ vanishes. 
\end{corol}

\begin{proof}
For every $\phi \in \mathcal{D}(B_R)$ and every $\varepsilon>0$, multiplying \eqref{equ_correcteur_theorem1} by $\chi = W_{\varepsilon,R}(./\varepsilon)\phi$ and integrating, we obtain : 
\begin{align}
\label{formulation_varationnelle_WV}
  \int_{B_R}  W_{\varepsilon,R}(./\varepsilon) V(./\varepsilon) \phi  = - \int_{B_R} \varepsilon W_{\varepsilon,R}(./\varepsilon) \nabla W_{\varepsilon,R}(./\varepsilon). \nabla \phi -  \int_{B_R} |\nabla W_{\varepsilon,R}(./\varepsilon)|^2 \phi.
\end{align}
For every $q \in [1,+\infty[$, since $\varepsilon W_{\varepsilon,R}(./\varepsilon)$ and $\nabla W_{\varepsilon,R}(./\varepsilon)$ are both bounded in $L^q(B_R)$, uniformly with respect to $\varepsilon$, the H\"older inequality gives the existence of a constant $C>0$ independent of $\varepsilon$ such that for every $p\in]1,+\infty[$,
\begin{equation}
    \left|  \int_{B_R}  W_{\varepsilon,R}(./\varepsilon) V(./\varepsilon) \phi  \right| \leq C \|\phi \|_{W^{1,p'}(B_R)},
\end{equation}
where we have denoted by $p'$ the conjugate exponent of $p$. It follows that $ W_{\varepsilon,R}(./\varepsilon) V(./\varepsilon)$ is uniformly bounded in $W^{-1,p}(B_R)$. Since $\varepsilon W_{\varepsilon,R}(./\varepsilon)$ converges to 0 in $L^{\infty}(B_R)$ as $\varepsilon$ vanishes and $|\nabla W_{\varepsilon,R}(./\varepsilon)|^2$ weakly converges to $-\mathcal{M}$ in $L^p(B_R)$ for every $p \in ]1,+\infty[$, the weak convergence of $W_{\varepsilon,R}(./\varepsilon) V(./\varepsilon)$ to $-\mathcal{M}$ in $W^{-1,p}$ is therefore a consequence of \eqref{formulation_varationnelle_WV}. 
\end{proof}

\section{Homogenization results}

\label{Schro_section5}

The existence of a corrector solution to \eqref{equ_correcteur_theorem1} satisfying suitable properties being established, we now turn to the proof of our homogenization results stated in Theorems \ref{schro_theorem_2} and \ref{schro_theorem_3}. In Section~\ref{schro_scetion_wellposedness} we begin by studying the well-posedness of problem \eqref{homog_eps_schro}, showing that the first eigenvalue of $-\Delta + \dfrac{1}{\varepsilon}V(./\varepsilon)$ converges to the first eigenvalue of $-\Delta - \mathcal{M}$. This result will next provide the sufficient condition~\eqref{H} which allows to perform the homogenization of \eqref{homog_eps_schro} and to prove Theorem~\ref{schro_theorem_2}. We next use the result of Theorem~\ref{schro_theorem_2} to show the convergence of \emph{all} the eigenvalues of $-\Delta + \dfrac{1}{\varepsilon}V(./\varepsilon)$ in Proposition \ref{prop_convergence_autre_vp} and we conclude with the proof of Theorem \ref{schro_theorem_3}.

\subsection{Well-posedness of Problem (\ref{homog_eps_schro})}

\label{schro_scetion_wellposedness}

In order to show, for $\varepsilon$ sufficiently small, the well-posedness of problem (\ref{homog_eps_schro}), we first need to introduce the following technical lemma :

\begin{lemme}
\label{Lemme_convergence_Hminus1}
Assume $d\geq 2$. Let $(f^{\varepsilon})_{\varepsilon>0}$ be a sequence in $L^2(B_{4R})$ for some $R>0$. Assume there exists $p>d$ such that $f^{\varepsilon}$ is bounded in $L^{p}(B_{4R})$, uniformly with respect to $\varepsilon$, and $f^{\varepsilon}$ weakly converges to 0 in $L^p(B_{4R})$ as $\varepsilon$ vanishes. 
Then there exists a sequence $\left(\psi^{\varepsilon}\right)_{\varepsilon>0}$ of $W^{2,p}(B_{4R})$ such that
$$\left\{
\begin{array}{cc}
  \Delta \psi^{\varepsilon} =f^{\varepsilon}  & \text{on } B_{4R}, \\
   \psi^{\varepsilon} =0  & \text{on } \partial B_{4R},
\end{array} \right.$$
and $\displaystyle \lim_{\varepsilon \to 0} \|\nabla \psi^{\varepsilon}\|_{L^{\infty}(B_R)} =0$. 
\end{lemme}

\begin{proof}
We denote by $\psi^{\varepsilon} \in H^1_0(B_{4R})$ the unique solution (provided by the Lax-Milgram Lemma) to 
\begin{equation}
\label{laplace_partiel_schro}
    \Delta \psi^{\varepsilon} = f^{\varepsilon} \quad \text{on } B_{4R}, 
\end{equation}
and we denote by $C_1>0$ a constant independent of $\varepsilon$ such that 
\begin{equation}
    \|\psi^{\varepsilon}\|_{H^1(B_{4R})} \leq C_1 \|f^{\varepsilon}\|_{L^2(B_{4R})}.
\end{equation}
Since $f^{\varepsilon}$ is uniformly bounded in $L^2(B_{4R})$ with respect to $\varepsilon$, the sequence $\psi^{\varepsilon}$ is bounded in $H^1_0(B_{4R})$ and, up to an extraction, it therefore weakly converges to a function $v \in H^1_0(B_{4R})$. Passing to the limit in \eqref{laplace_partiel_schro} when $\varepsilon \to 0 $, we obtain
$$\Delta v = 0 \quad \text{on } B_{4R}.$$
Since $v \in H^1_0(B_{4R})$, we deduce that $v = 0$. 

We next use elliptic regularity, see for instance \cite[Theorem 7.4 p.141]{MR3099262}, to obtain the existence of a constant $C_2>0$, independent of $\varepsilon$, such that : 
\begin{equation}
    \|D^2 \psi^{\varepsilon} \|_{L^p(B_{4R})} \leq C_2 \|f^{\varepsilon} \|_{L^p(B_{4R})}.  
    \end{equation}
The Morrey inequality (see \cite[p.268]{evans10}) next yields a constant $C_3>0$, also independent of $\varepsilon$, such that for every $x,y \in B_{2R}$, $x\neq y$, and $i \in \{1,...,d\}$ : 
\begin{equation}
\label{Morrey_schro_correcteur2}
|\partial_i \psi^{\varepsilon}(x) - \partial_i \psi_{\varepsilon}(y)| \leq C_3 |x-y|^{1-d/p} \|D^2 \psi_{\varepsilon}\|_{L^p(B_{4R})}\leq C_2C_3 |x-y|^{1-d/p} \|f^{\varepsilon}\|_{L^p(B_{4R})}.
\end{equation}
Since $f^{\varepsilon}$ is bounded in $L^p(B_{4R})$, \eqref{Morrey_schro_correcteur2} shows that $\nabla \psi^{\varepsilon}$ is bounded in $L^{\infty}(B_{2R})$ and equicontinuous on $B_{2R}$, both uniformly with respect to $\varepsilon$. The Arzela-Ascoli theorem therefore shows that the sequence $\nabla \psi^{\varepsilon}$, up to an extraction, uniformly converges on every compact of $B_{2R}$. Since $\nabla \psi^{\varepsilon}$ weakly converges to 0 in $L^2(B_{4R})$, the uniqueness of the limit in $\mathcal{D}'(B_R)$ allows to conclude that $\nabla \psi^{\varepsilon}$ uniformly converges to 0. Since 0 is the unique adherent value of $\nabla \psi^{\varepsilon}$ in $L^{\infty}(B_{R})$, a compactness argument ensures that the whole sequence $\nabla \psi^{\varepsilon}$ converges to 0 in $L^{\infty}(B_{R})$. 
\end{proof}

We next establish the convergence of the first eigenvalue of $-\Delta + \dfrac{1}{\varepsilon}V(./\varepsilon)$ when $\varepsilon$ vanishes. 

\begin{prop}
\label{prop_convergence_vp}
Let $\lambda_1^{\varepsilon}$ and $\mu_1$ be respectively the first eigenvalue of $-\Delta + \dfrac{1}{\varepsilon}V(./\varepsilon)$ and the first eigenvalue of $-\Delta$ on $\Omega$ both with homogeneous Dirichlet boundary conditions. Then, under the assumptions of Theorem \ref{theorem1_schro}, 
\begin{equation}
    \lim_{\varepsilon \to 0} \lambda_1^{\varepsilon} = \mu_1 - \mathcal{M},
\end{equation}
where $\mathcal{M}$ is the constant given by \eqref{d1} in Theorem \ref{theorem1_schro}. 
\end{prop}

\begin{proof}
Our proof is adapted from that of the equivalent result \cite[Theorem 12.2 p.162]{bensoussan2011asymptotic} established in the \emph{periodic} case. For $\varepsilon>0$, we introduce the operator $\pi_{\varepsilon}$ defined for every $v\in H^1_0(\Omega)$ by : 
$$\pi_{\varepsilon}(v) = \int_{\Omega} |\nabla v|^2 + \dfrac{1}{\varepsilon}V(./\varepsilon)v^2.$$
We denote by $R$ the diameter of $\Omega$ and we introduce $W_{\varepsilon,\Omega}= W_{\varepsilon, R}$ the corrector given by Theorem~\ref{theorem1_schro}. For every $v\in H^1_0(\Omega)$ and $\varepsilon>0$ sufficiently small, we define $\phi^{\varepsilon}\in H^1_0(\Omega)$ by : 
$$v = (1+\varepsilon W_{\varepsilon,\Omega}(./\varepsilon))\phi^{\varepsilon}.$$
Since $\varepsilon W_{\varepsilon,\Omega}(./\varepsilon)$ uniformly converges to 0 on $\Omega$, $\phi^{\varepsilon}$ is uniquely defined when $\varepsilon$ is sufficiently small.

Since $\Delta W_{\varepsilon,\Omega} = V$ on $\Omega/\varepsilon$, we have
\begin{align*}
    -\Delta v + \dfrac{1}{\varepsilon}V(./\varepsilon) v &= - (1+\varepsilon W_{\varepsilon,\Omega}(./\varepsilon))\Delta \phi^{\varepsilon} - 2 \nabla W_{\varepsilon,\Omega}(./\varepsilon).\nabla \phi^{\varepsilon}
    + \left(VW_{\varepsilon,\Omega}\right)(./\varepsilon) \phi^{\varepsilon},
\end{align*}
and a direct calculation shows that
\begin{align*}
   \pi_{\varepsilon}(v) &= I_1^{\varepsilon} + I_2^{\varepsilon},
\end{align*}
where 
\begin{align}
\label{def_I1_schro_vp}
    I_1^{\varepsilon} & = \int_{\Omega} (1+\varepsilon W_{\varepsilon,\Omega}(x/\varepsilon))^2 |\nabla \phi^{\varepsilon}(x)|^2 dx, \\
    I_2 & = \int_{\Omega} \left(VW_{\varepsilon,\Omega}\right)(x/\varepsilon)(1+ \varepsilon W_{\varepsilon,\Omega}(x/\varepsilon))|\phi^{\varepsilon}(x)|^2 dx.
\end{align}

Our aim is now to study the behavior of $I^{\varepsilon}_1$ and $I^{\varepsilon}_2$. 
For $\varepsilon>0$, we define $\delta_1(\varepsilon)= \|\varepsilon W_{\varepsilon,\Omega}(./\varepsilon)\|_{L^{\infty}(\Omega)}$ and we have
\begin{equation}
    \label{majoration_I1_vp}
    I^{\varepsilon}_1 \geq (1-2 \delta_1(\varepsilon)) \int_{\Omega} |\nabla \phi^{\varepsilon}|^2.
\end{equation}
In order to bound $I_2^{\varepsilon}$ from below, we consider $\chi(x) =  W_{\varepsilon, \Omega}(x/\varepsilon) \left(1 + \varepsilon W_{\varepsilon, \Omega}(x/\varepsilon)\right)|\phi^{\varepsilon}(x)|^2$. Since $\Delta W_{\varepsilon, \Omega}(./\varepsilon) = V(./\varepsilon)$, we have :
\begin{align}
\label{inegalite_schor_vp_2_1}
    I_2^{\varepsilon} &= \int_{\Omega} V(x/\varepsilon) \chi(x) dx 
     = \int_{\Omega} \Delta W_{\varepsilon, \Omega}(x/\varepsilon) \chi(x) dx 
     = I_{2,1}^{\varepsilon} + I_{2,2}^{\varepsilon},
\end{align}
where 
\begin{align}
     I_{2,1}^{\varepsilon}  &  = -\int_{\Omega} |\nabla W_{\varepsilon, \Omega}(x/\varepsilon)|^2 \left(1 + \varepsilon W_{\varepsilon, \Omega}(x/\varepsilon)\right)|\phi^{\varepsilon}(x)|^2 dx, \\
     I_{2,2}^{\varepsilon} & = - \int_{\Omega} \varepsilon W_{\varepsilon, \Omega}(x/\varepsilon) \nabla W_{\varepsilon, \Omega}(x/\varepsilon). \nabla \left(\left(1 + \varepsilon W_{\varepsilon, \Omega}(x/\varepsilon)\right)|\phi^{\varepsilon}(x)|^2\right)dx .
\end{align}
We note that the rightmost equality in \eqref{inegalite_schor_vp_2_1} is valid since the trace the trace of $\phi^{\varepsilon}$ vanishes on $\partial \Omega$ (we recall that $\Omega$ is assumed to be $\mathcal{C}^1$). Since $|\nabla W_{\varepsilon,\Omega}(./\varepsilon)|^2$ weakly converges to $\mathcal{M}$ in $L^p(\Omega)$, for every $p \in ]1, +\infty[$, we can introduce the function $\psi^{\varepsilon} \in H^1(\Omega)$ given by Lemma \ref{Lemme_convergence_Hminus1}, solution to
$$\Delta \psi^{\varepsilon} = \mathcal{M} - |\nabla W_{\varepsilon,\Omega}(./\varepsilon)|^2   \quad \text{on } \Omega,$$
and such that $\nabla \psi^{\varepsilon}$ converges to 0 in $L^{\infty}(\Omega)$ as $\varepsilon$ vanishes. We also define $\delta_2(\varepsilon) =~\|\nabla \psi^{\varepsilon}\|_{L^{\infty}(\Omega)}$. We split $I^{\varepsilon}_{2,1}$ as
\begin{align}
\label{egalite_I21_vp_schro}
    I^{\varepsilon}_{2, 1}=&   \int_{\Omega} -\mathcal{M}(1+\varepsilon W_{\varepsilon,\Omega}(x/\varepsilon))|\phi^{\varepsilon}(x)|^2dx + \int_{\Omega} \Delta \psi^{\varepsilon}(x)  (1+\varepsilon W_{\varepsilon,\Omega}(x/\varepsilon))|\phi^{\varepsilon}(x)|^2 dx.
\end{align}
On the one hand, we have
\begin{equation}
\label{inegalite1_I21_vp_schro}
    \int_{\Omega} -\mathcal{M}(1+\varepsilon W_{\varepsilon,\Omega}(x/\varepsilon))|\phi^{\varepsilon}(x)|^2dx \geq -(1+\delta_1(\varepsilon)) \mathcal{M} \int_{\Omega} |\phi^{\varepsilon}(x)|^2 dx.
\end{equation}
On the other hand, since the trace of $\phi^{\varepsilon}$ vanishes on $\partial \Omega$, we obtain
\begin{align}
\label{inegalite2_I21_vp_schro}
    \int_{\Omega} \Delta \psi^{\varepsilon}  (1+\varepsilon W_{\varepsilon,\Omega}(./\varepsilon))&|\phi^{\varepsilon}|^2 dx = -  \int_{\Omega} \nabla \psi^{\varepsilon}.\nabla\left((1+\varepsilon W_{\varepsilon,\Omega}(./\varepsilon))|\phi^{\varepsilon}|^2\right)dx.
\end{align}
We remark that $ \displaystyle \nabla\left((1+\varepsilon W_{\varepsilon,\Omega}(./\varepsilon))|\phi^{\varepsilon}|^2\right) = \nabla W_{\varepsilon,\Omega}(./\varepsilon)|\phi^{\varepsilon}|^2 + (1+\varepsilon W_{\varepsilon,\Omega}(./\varepsilon))  \nabla |\phi^{\varepsilon}|^2$. Since $\nabla |\phi^{\varepsilon}|^2 = 2 \phi^{\varepsilon} \nabla \phi^{\varepsilon}$, the Cauchy-Schwarz inequality shows
$$\int_{\Omega} |\nabla (|\phi^{\varepsilon}|^2)| \leq 2\|\phi^{\varepsilon}\|_{L^2(\Omega)}\|\nabla \phi^{\varepsilon}\|_{L^2(\Omega)} \leq  2 \|\phi^{\varepsilon}\|_{H^1(\Omega)}^2.$$
Using the H\"older inequality and the fact that $\nabla W_{\varepsilon,R}(./\varepsilon)$ is bounded in $L^p(\Omega)$, uniformly with respect to $\varepsilon$ and for every $p\in]1,+\infty[$, we have the existence of $C>0$ independent of $\varepsilon$ and $\phi^{\varepsilon}$ such that
\begin{align}
\label{inegalite3_I21_vp_schro}
    \int_{\Omega} \left|\nabla W_{\varepsilon,\Omega}(./\varepsilon)\right| |\phi^{\varepsilon}|^2 \leq \left\{ \begin{array}{cc}
       \|\nabla W_{\varepsilon,\Omega}(./\varepsilon)\|_{L^2(\Omega)} \|\phi^{\varepsilon}\|^2_{L^4(\Omega)} \leq C \|\phi^{\varepsilon}\|^2_{H^1(\Omega)} & \text{if } d =2, \\
       & \\
       \quad \|\nabla W_{\varepsilon,\Omega}(./\varepsilon)\|_{L^{\frac{d}{2}}(\Omega)} \|\phi^{\varepsilon}\|^2_{L^{\frac{2d}{d-2}}(\Omega)} \leq C
      \|\phi^{\varepsilon}\|^2_{H^1(\Omega)}  & \text{if } d\geq 3.
    \end{array}  \right.
\end{align}
The latter inequality above is a consequence of a Sobolev embedding from $H^1_0(\Omega)$ to $L^{p}(\Omega)$ for every $p<+\infty$ if $d= 2$ or every $p \leq \dfrac{2d}{d-2}$ if $d\geq 3$. \eqref{inegalite2_I21_vp_schro} and \eqref{inegalite3_I21_vp_schro} therefore yield the existence of $c_1>0$ such that 
\begin{equation}
\label{inegalite4_I21_vp_schro}
    \int_{\Omega} \Delta \psi^{\varepsilon} (1+\varepsilon W_{\varepsilon, \Omega}(x/\varepsilon))|\phi^{\varepsilon}(x)|^2 dx \geq - c_1 \delta_2(\varepsilon)(1+ \delta_1(\varepsilon)) \|\phi^{\varepsilon}\|_{H^1(\Omega)}^2.
\end{equation}
Finally, using \eqref{egalite_I21_vp_schro}, \eqref{inegalite1_I21_vp_schro} and \eqref{inegalite4_I21_vp_schro}, we show that
\begin{equation}
\label{minoration_I21_vp}
    I_{2,1}^{\varepsilon} \geq -(1+\delta_1(\varepsilon)) \mathcal{M} \int_{\Omega} |\phi^{\varepsilon}(x)|^2 dx - c_1 \delta_2(\varepsilon)(1+ \delta_1(\varepsilon)) \|\phi^{\varepsilon}\|_{H^1(\Omega)}^2 .
\end{equation}
Using that $\varepsilon W_{\varepsilon, \Omega}(./\varepsilon)$ uniformly converges on $\Omega$ and $\nabla W_{\varepsilon, \Omega}(./\varepsilon)$ is bounded in $L^p(\Omega)$, we can similarly show the existence of $c_2>0$ such that 
\begin{equation}
\label{minoration_I22_vp}
    I_{2,2}^{\varepsilon} \geq - c_2 \delta_1(\varepsilon)(1+ \delta_1(\varepsilon)) \|\phi^{\varepsilon}\|_{H^1(\Omega)}^2,
\end{equation}
and we obtain 
\begin{equation}
\label{minoration_I2_vp}
    I^{\varepsilon}_2 \geq  -(1+\delta_1(\varepsilon)) \mathcal{M} \int_{\Omega} |\phi^{\varepsilon}(x)|^2 dx - \left(c_1 \delta_2(\varepsilon) + c_2 \delta_1(\varepsilon)\right)(1+ \delta_1(\varepsilon)) \|\phi^{\varepsilon}\|_{H^1(\Omega)}^2.
\end{equation}

Hence, we use \eqref{majoration_I1_vp} and \eqref{minoration_I2_vp}, and we obtain 
\begin{align*}
  \pi_{\varepsilon}(v) & \geq   (1-2 \delta_1(\varepsilon)- \left(c_1 \delta_2(\varepsilon) + c_2 \delta_1(\varepsilon)\right)(1+ \delta_1(\varepsilon))) \int_{\Omega} |\nabla \phi^{\varepsilon}|^2 \\
  & - \left((1+\delta_1(\varepsilon)) \mathcal{M} - \left(c_1 \delta_2(\varepsilon) + c_2 \delta_1(\varepsilon)\right)(1+ \delta_1(\varepsilon)) \right) \int_{\Omega} |\phi^{\varepsilon}(x)|^2.
\end{align*}
The definition of $\mu_1$ gives  
$$\int_{\Omega} |\nabla \phi^{\varepsilon}|^2 \geq \mu_1 \int_{\Omega} |\phi^{\varepsilon}|^2.$$
When $\varepsilon$ is sufficiently small, it therefore follows
\begin{align*}
    \pi_{\varepsilon}(v) \geq \left(\mu_1 - \mathcal{M} - \delta_3(\varepsilon)\right)\int_{\Omega} |\phi^{\varepsilon}|^2,
\end{align*}
where we have denoted 
$$\delta _3(\varepsilon) = \mu_1\left(2 \delta_1(\varepsilon)+ \left(c_1 \delta_2(\varepsilon) + c_2 \delta_1(\varepsilon)\right)(1+ \delta_1(\varepsilon))\right) + \delta_1(\varepsilon) \mathcal{M} + \left(c_1 \delta_2(\varepsilon) + c_2 \delta_1(\varepsilon)\right)(1+ \delta_1(\varepsilon)) .$$
When $\varepsilon$ is sufficiently small, we also have, by definition 
$$\phi^{\varepsilon} = \dfrac{v}{1+\varepsilon W_{\varepsilon,R}(./\varepsilon)},$$ and a Taylor expansion provides the existence of $c_3>0$ independent of $\varepsilon$ such that
$$ \|\phi^{\varepsilon}\|_{L^2(\Omega)} \geq (1 - c_3\delta_1(\varepsilon)) \|v\|_{L^2(\Omega)}.$$
So we obtain 
\begin{align}
    \pi_{\varepsilon}(v) \geq \left(\mu_1 - \mathcal{M} - \delta_3(\varepsilon)\right)(1-c_3\delta_1(\varepsilon))^2\int_{\Omega} |v|^2.
\end{align}
Since this inequality holds for every $v \in H^1_0$, it follows : 
\begin{equation}
    \label{minoration_pi_schro_vp}
    \lambda_1^{\varepsilon} = \inf_{v \in H^1_0} \dfrac{\pi_{\varepsilon}(v)}{\|v\|^2_{L^2(\Omega)}} \geq \left(\mu_1 - \mathcal{M} - \delta_3(\varepsilon)\right)(1-c_3\delta_1(\varepsilon))^2.
\end{equation}
We next establish a similar upper bound for $\pi_{\varepsilon}(v)$. We consider $\phi \in H^1_0(\Omega)$ such that $\|\phi\|_{L^2(\Omega)}=1$ and, for $\varepsilon>0$, we define $v^{\varepsilon} = (1+\varepsilon W_{\varepsilon, \Omega}(./\varepsilon))\phi$. We have again to bound the following two integrals :
$$I^{\varepsilon}_1 = \int_{\Omega} (1+\varepsilon W_{\varepsilon,\Omega}(x/\varepsilon))^2 |\nabla \phi(x)|^2 dx,$$
and
$$I^{\varepsilon}_2 = \int_{\Omega} VW_{\varepsilon,\Omega}(x/\varepsilon)(1+ \varepsilon W_{\varepsilon, \Omega}(x/\varepsilon))|\phi(x)|^2 dx.$$
As above, we can show that
\begin{align*}
  \pi_{\varepsilon}(v^{\varepsilon}) & \leq   (1+\delta_4(\varepsilon)) \int_{\Omega} |\nabla \phi|^2  + \left(- \mathcal{M} + \delta_5(\varepsilon) \right) \int_{\Omega} |\phi|^2 \\
  & = (1+\delta_4(\varepsilon)) \int_{\Omega} |\nabla \phi|^2  - \mathcal{M} + \delta_5(\varepsilon),
\end{align*}
where $\delta_4(\varepsilon)$ and $\delta_5(\varepsilon)$ both depend on $\|\varepsilon W_{\varepsilon, \Omega}(./\varepsilon)\|_{L^{\infty}(\Omega)}$ and $\|\nabla \psi^{\varepsilon}\|_{L^{\infty}(\Omega)}$ and converge to 0 when $\varepsilon \to 0$. Since $v^{\varepsilon} = (1+\varepsilon W_{\varepsilon,\Omega}(./\varepsilon))\phi$, we have $    \|v^{\varepsilon}\|_{L^2(\Omega)} \geq (1-\delta_1(\varepsilon))\|\phi\|_{L^2(\Omega)} =  1- \delta_1(\varepsilon)$, and it follows
\begin{align*}
    \lambda_1^{\varepsilon} & = \inf_{v \in H^1_0(\Omega)} \dfrac{\pi_{\varepsilon}(v)}{\|v\|^2_{L^2(\Omega)}} \leq  \dfrac{\pi_{\varepsilon}(v^{\varepsilon})}{\|v^{\varepsilon}\|^2_{L^2(\Omega)}} \leq \dfrac{1}{(1-\delta_1(\varepsilon))^2}\left((1+\delta_4(\varepsilon)) \int_{\Omega} |\nabla \phi|^2  - \mathcal{M} + \delta_5(\varepsilon)\right).
\end{align*}
This inequality holds for every $\phi \in H^1_0(\Omega)$ such that $\|\phi\|_{L^2(\Omega)} =1$. By definition of $\mu_1$, we have $\mu_1 = \displaystyle \inf_{\phi\in H^1_0(\Omega), \|\phi\|_{L^2(\Omega)} =1}\int_{\Omega}|\nabla \phi|^2$ and since $\delta_1(\varepsilon)$ converges to 0, we obtain the existence of $c_4>0$ independent of $\varepsilon$ such that 
\begin{equation}
\label{majoration_pi_schro_vp}
\begin{split}
     \lambda_1^{\varepsilon}& \leq \dfrac{1}{(1-\delta_1(\varepsilon))^2}\left((1+\delta_4(\varepsilon))\inf_{\phi\in H^1_0(\Omega), \|\phi\|_{L^2(\Omega)} =1}\int_{\Omega}|\nabla \phi|^2 - \mathcal{M} + \delta_5(\varepsilon)\right) \\
   & \leq c_4\left((1+\delta_4(\varepsilon))\mu_1 - \mathcal{M} + \delta_5(\varepsilon)\right).
  \end{split}
\end{equation}
Using \eqref{minoration_pi_schro_vp} and \eqref{majoration_pi_schro_vp}, we have finally shown that
$$ \left(\mu_1 - \mathcal{M} - \delta_3(\varepsilon)\right)(1-c_3\delta_1(\varepsilon))^2 \leq \lambda_1^{\varepsilon} \leq c_4\left((1+\delta_4(\varepsilon))\mu_1 - \mathcal{M} + \delta_5(\varepsilon)\right).$$
Passing to the limit when $\varepsilon \to 0$, we conclude that $\displaystyle \lim_{\varepsilon \to 0} \lambda_1^{\varepsilon} = \mu_1 - \mathcal{M}$.
\end{proof}

In particular, when \eqref{H} is satisfied and if $\varepsilon$ is sufficiently small, Proposition \ref{prop_convergence_vp} ensures that the quadratic form 
$$a^{\varepsilon}(u,u)= \int_{\Omega} |\nabla u|^2 + \dfrac{1}{\varepsilon}\int_{\Omega} V(./\varepsilon) |u|^2 + \nu \int_{\Omega} |u|^2,$$
is coercive in $H^1_0(\Omega)$, uniformly with respect to $\varepsilon$. 
Indeed, given \eqref{H}, a simple adaptation of the proof of Proposition \ref{prop_convergence_vp} shows that there exists $0<\eta<1$ such that, for every $\varepsilon>0$ sufficiently small, the first eigenvalue of $-(1-\eta)\Delta + \dfrac{1}{\varepsilon}V(./\varepsilon) + \nu$ is positive and, consequently, we have 
$$0\leq (1-\eta)\int_{\Omega} |\nabla u|^2 + \dfrac{1}{\varepsilon}\int_{\Omega} V(./\varepsilon) |u|^2 + \nu \int_{\Omega} |u|^2.$$
It follows that 
\begin{equation}
\begin{split}
\label{ineg_1_schro_coercivite_split}
    \dfrac{\eta}{2} \|\nabla u \|^2_{L^2(\Omega)} & \leq \dfrac{\eta}{2} \int_{\Omega} |\nabla u|^2 + \dfrac{1}{2}\left((1-\eta)\int_{\Omega} |\nabla u|^2 + \dfrac{1}{\varepsilon}\int_{\Omega} V(./\varepsilon) |u|^2 + \nu \int_{\Omega} |u|^2\right)\\
    & =  \dfrac{1}{2}a^{\varepsilon}(u,u).
\end{split}
\end{equation}
Moreover, for $\varepsilon$ sufficiently small, Proposition \ref{prop_convergence_vp} also ensures that
\begin{equation}
    \label{ineg_2_schro_coercivite_split}
    \dfrac{\mu_1 - \mathcal{M} + \nu}{2} \|u\|^2_{L^2(\Omega)} \leq \dfrac{1}{2} a^{\varepsilon}(u,u).
\end{equation}
Denoting $C = \min\left(\eta, \left( \mu_1 - \mathcal{M} + \nu\right)\right)$, assumption \eqref{H} gives $C>0$ and, using \eqref{ineg_1_schro_coercivite_split} and \eqref{ineg_2_schro_coercivite_split}, we obtain 
\begin{equation}
\label{coercivite_unif_schro}
   \dfrac{C}{2} \|u\|_{H^1(\Omega)}^2 \leq a^{\varepsilon}(u,u). 
\end{equation}
For such values of $\varepsilon$, problem \eqref{homog_eps_schro} is therefore well-posed in $H^1_0(\Omega)$ as a consequence of the Lax-Milgram lemma.

\begin{remark}
If the periodic potential $\displaystyle V_{per} = g_{per} + \sum_{k \in \mathbb{Z}^d} \varphi(.-k)$ satisfies assumption \eqref{Schro_condition_suff_per}, we remark that assumption \eqref{H} is not necessarily satisfied by $\displaystyle V = g_{per} + \sum_{k \in \mathbb{Z}^d} \varphi(.-k-Z_k)$. Consider indeed the one-dimensional example for which $g_{per} = 0$ and $V(x) = \displaystyle \sum_{k \in \mathbb{Z}} \psi'(x-k-Z_k)$ where $ \displaystyle \psi \in \mathcal{D}(\mathbb{R})$ is nonnegative and has support in $[0,1]$ and $Z_k\in]0,1[$ satisfies 
\begin{align}
    &\lim_{\varepsilon \to 0} \dfrac{\varepsilon}{R}\sum_{k \in B_{R}(x_0)/\varepsilon} F(Z_k) = \int_0^1 F(t)dt,\\
    \forall l \neq 0, \ &\lim_{\varepsilon \to 0} \dfrac{\varepsilon}{R}\sum_{k \in B_{R}(x_0)/\varepsilon} G(Z_k,Z_{k+l}) = \int_0^1 \int_0^1 G(t,u)dtdu,
\end{align}
for every $F \in \mathcal{C}^0(\mathbb{R})$, $G \in \mathcal{C}^0(\mathbb{R}, \mathbb{R})$. Such a sequence $Z$ indeed exists, as for example shown by the deterministic approximation of random uniform distribution given in Section \ref{Section_schro_examples}. In this case, the periodic corrector $w_{per}$ and our adapted corrector $w$, respectively solution to $w_{per}'' = V_{per}$ and solution to $w'' = V$, can both be made explicit and their derivatives are respectively given by $w_{per}' = \displaystyle \sum_{k \in \mathbb{Z}} \psi(.-k)$ and $w' = \displaystyle \sum_{k \in \mathbb{Z}} \psi(.-k-Z_k)$. An explicit calculation using the properties of $Z_k$ shows that 
\begin{align*}
  \langle |w'_{per}|^2 \rangle &=  \operatorname{weak} \underset{\varepsilon \to 0}{\lim}\ |w'_{per}(./\varepsilon)|^2 = \int_{\mathbb{R}} |\psi(t)|^2 dt \\
  \mathcal{M} &=  \operatorname{weak} \underset{\varepsilon \to 0}{\lim} \ |w'(./\varepsilon)|^2 = \int_{\mathbb{R}} |\psi(t)|^2 dt + \int_{\mathbb{R}} \psi(t) \int_0^1 \int_0^1 \psi(t+1-u+v)du \hspace{1pt} dv\hspace{1pt} dt, 
\end{align*}
and $\mathcal{M} > \langle |w'_{per}|^2 \rangle$ as soon as $\displaystyle \int_{\mathbb{R}} \psi(t) \int_0^1 \int_0^1 \psi(t+1-u+v)du \hspace{1pt} dv\hspace{1pt} dt >0$. In this case, \eqref{H} is therefore more restrictive than \eqref{Schro_condition_suff_per}. 
\end{remark}

\subsection{Proof of Theorem \ref{schro_theorem_2}}

\label{Sction_homog_schro}

In this section, we homogenize \eqref{homog_eps_schro} given the sufficient assumption~\eqref{H}. We first introduce the unique solution $u^*$ in $H^1_0(\Omega)$ to \eqref{equation_homog_schrodinger}. The existence and uniqueness of $u^*$ is, of course, ensured by~\eqref{H}. Our aim is now to show that the sequence $(u^{\varepsilon})_{\varepsilon>0}$ of solutions to \eqref{homog_eps_schro}, for $\varepsilon$ small, indeed converges to $u^*$ and, considering the sequence of remainders $R^{\varepsilon}$ defined by \eqref{schro_def_R}, to make precise the behavior of $u^{\varepsilon}$ in $H^1(\Omega)$. 

To start with, we establish the following proposition :  

\begin{prop}
\label{prop_schro_homog}
Under the assumptions of Theorem \ref{theorem1_schro}, and if $V$ additionally satisfies \eqref{H}, then the sequence $u^{\varepsilon}$ of solutions to \eqref{homog_eps_schro} converges strongly in $L^2(\Omega)$ and weakly in $H^1(\Omega)$ to the unique solution $u^* \in H^1_0(\Omega)$ to \eqref{equation_homog_schrodinger}. 
\end{prop}

\begin{proof}
When $\varepsilon$ is sufficiently small, assumption \eqref{H} and Proposition \ref{prop_convergence_vp} give the existence of $C>0$ independent of $\varepsilon$ such that
$$C \|u^{\varepsilon}\|^2_{H^1(\Omega)} \leq \int_{\Omega} \left(|\nabla u^{\varepsilon}|^2 + \dfrac{1}{\varepsilon}V(./\varepsilon)|u^{\varepsilon}|^2 + \nu |u^{\varepsilon}|^2\right) \leq \|f\|_{L^2(\Omega)}\|u^{\varepsilon}\|_{H^1(\Omega)}.$$
The function $u^{\varepsilon}$ is therefore bounded in $H^1(\Omega)$, uniformly with respect to $\varepsilon$, and, up to an extraction, it converges strongly in $L^2(\Omega)$ and weakly in $H^1(\Omega)$ to a function $u^* \in H^1_0(\Omega)$. For every $\varepsilon>0$, we introduce $W_{\varepsilon, \Omega}$, the corrector given by Theorem \ref{theorem1_schro} for $R = Diam(\Omega)$ and we define
$$\theta^{\varepsilon} = 1 + \varepsilon W_{\varepsilon,\Omega}(./\varepsilon).$$
Since $W_{\varepsilon,\Omega}$ satisfies $\Delta W_{\varepsilon,\Omega}(./\varepsilon) = V(./\varepsilon)$ in $\mathcal{D}'(\Omega)$, we have
\begin{equation}
\label{schro_equation_theta}
   -\Delta \theta^{\varepsilon} + \dfrac{1}{\varepsilon} V(./\varepsilon) \hspace{0.9pt} \theta^{\varepsilon} = V(./\varepsilon)W_{\varepsilon,\Omega}(./\varepsilon). 
\end{equation}
For $\psi \in \mathcal{D}(\Omega)$, we multiply \eqref{homog_eps_schro}  and \eqref{schro_equation_theta} respectively by $\theta^{\varepsilon}\psi$ and $u^{\varepsilon}\psi$ and we obtain 
$$\int_{\Omega} \nabla u^{\varepsilon}.\nabla \theta^{\varepsilon} \hspace{0.9pt} \psi + \int_{\Omega} \nabla u^{\varepsilon}.\nabla \psi \hspace{0.9pt} \theta^{\varepsilon} + \int_{\Omega} \dfrac{1}{\varepsilon} V(./\varepsilon)u^{\varepsilon} \theta^{\varepsilon}\psi + \int_{\Omega} \nu \hspace{0.9pt} u^{\varepsilon} \theta^{\varepsilon}\psi = \int_{\Omega} f \theta^{\varepsilon} \psi,$$
and 
$$\int_{\Omega} \nabla u^{\varepsilon}.\nabla \theta^{\varepsilon} \psi + \int_{\Omega} \nabla \theta^{\varepsilon}.\nabla \psi \hspace{0.9pt} u^{\varepsilon} + \int_{\Omega} \dfrac{1}{\varepsilon} V(./\varepsilon)u^{\varepsilon} \hspace{0.9pt} \theta^{\varepsilon}\psi = \int_{\Omega} V(./\varepsilon)W_{\varepsilon,\Omega}(./\varepsilon)u^{\varepsilon} \psi.$$
Subtracting the above two equalities, we have 
\begin{equation}
\label{egalite_integrale_schro}
   \int_{\Omega} \nabla u^{\varepsilon}.\nabla \psi \hspace{0.9pt} \theta^{\varepsilon} - \int_{\Omega} \nabla \theta^{\varepsilon}.\nabla \psi \hspace{0.9pt} u^{\varepsilon} + \int_{\Omega} \nu \hspace{0.9pt} u^{\varepsilon} \hspace{0.9pt} \theta^{\varepsilon}\psi  = \int_{\Omega} f \theta^{\varepsilon} \psi - \int_{\Omega} V(./\varepsilon)W_{\varepsilon,\Omega}(./\varepsilon)u^{\varepsilon} \psi. 
\end{equation}
Since $u^{\varepsilon}$ weakly converges to $u^*$ in $H^1(\Omega)$ and  $\varepsilon W_{\varepsilon,\Omega}(./\varepsilon)$ uniformly converges to 0 on $\Omega$, we have 
$$\lim_{\varepsilon \to 0} \int_{\Omega} \nabla u^{\varepsilon}.\nabla \psi \hspace{0.9pt} \theta^{\varepsilon} = \int_{\Omega} \nabla u^*. \nabla \psi,$$ 
$$ \lim_{\varepsilon \to 0} \int_{\Omega} \nu \hspace{0.9pt} u^{\varepsilon} \hspace{0.9pt} \theta^{\varepsilon}\psi  = \int_{\Omega} \nu \hspace{0.9pt} u^{*} \psi,$$ and  
$$ \lim_{\varepsilon \to 0} \int_{\Omega} f \theta^{\varepsilon} \psi = \int_{\Omega} f  \psi.$$
Similarly, the weak convergence of $\nabla \theta^{\varepsilon} = \nabla W_{\varepsilon,\Omega}(./\varepsilon)$  to 0 and the strong convergence of $u^{\varepsilon}$ to $u^*$ in $L^2(\Omega)$ imply
$$\lim_{\varepsilon \to 0} \int_{\Omega} \nabla \theta^{\varepsilon}.\nabla \psi \hspace{0.9pt} u^{\varepsilon} = 0.$$
Multiplying next the corrector equation \eqref{equ_correcteur_theorem1} by $\chi = W_{\varepsilon , \Omega}(./\varepsilon)u^{\varepsilon} \psi$ and since $u^{\varepsilon} \psi$ converges to $u^* \psi$ strongly in $L^2(\Omega)$ and weakly in $H^1(\Omega)$, the convergence properties of the corrector imply 
\begin{align*}
    \lim_{\varepsilon \to 0} \int_{\Omega}  W_{\varepsilon,\Omega}(./\varepsilon)V(./\varepsilon)u^{\varepsilon} \psi & = \lim_{\varepsilon \to 0} - \int_{\Omega} |\nabla W_{\varepsilon,\Omega}(./\varepsilon)|^2 u^{\varepsilon} \psi - \int_{\Omega} \varepsilon W_{\varepsilon, \Omega}(./\varepsilon)\nabla (u^{\varepsilon} \psi)\\
    &= - \mathcal{M} \int_{\Omega} u^*\psi.
\end{align*}
Passing to the limit in \eqref{egalite_integrale_schro} when $\varepsilon \to 0$, we have shown that for every $\psi \in \mathcal{D}(\Omega)$ we have
$$ \int_{\Omega} \nabla u^*. \nabla \psi + (- \mathcal{M} + \nu) \int_{\Omega} u^* \psi = \int_{\Omega} f \psi.$$
We have therefore proved that $u^*$ is a solution to \eqref{equation_homog_schrodinger}. The limit being independent of the extraction, we can conclude that the sequence $u^{\varepsilon}$ converges to $u^*$. 
\end{proof}

We are now able to study the behavior of $u^{\varepsilon}$ in $H^1(\Omega)$ and to prove Theorem \ref{schro_theorem_2}.

\begin{proof}[Proof of Theorem \ref{schro_theorem_2}]
We first show that $R^{\varepsilon} = u^{\varepsilon} - u^* - \varepsilon u^* W_{\varepsilon, \Omega}(./\varepsilon)$ is uniformly bounded in $H^1(\Omega)$ with respect to $\varepsilon$. Indeed, when $\varepsilon$ is sufficiently small, assumption \eqref{H} and Proposition \ref{prop_convergence_vp} ensure that the bilinear form $a$ defined by 
$$a^{\varepsilon}(u,v) = \int_{\Omega} \nabla u(x).\nabla v(x) dx + \dfrac{1}{\varepsilon} \int_{\Omega} V(x/\varepsilon) u(x)v(x)dx + \nu \int_{\Omega}  u(x)v(x)dx,$$
is coercive, uniformly with respect to $\varepsilon$, and that the sequence $u^{\varepsilon}$ is therefore uniformly bounded in $H^1(\Omega)$. Moreover, we know that $\varepsilon W_{\varepsilon,\Omega}(./\varepsilon)$ is uniformly bounded in $L^{\infty}(\Omega)$ and $\nabla W_{\varepsilon,\Omega}(./\varepsilon)$ is uniformly bounded in $L^{p}(\Omega)$ for every $p\in [1,+\infty[$. It follows that $R^{\varepsilon}$ is bounded in $H^1(\Omega)$ uniformly with respect to $\varepsilon$. Proposition \ref{prop_schro_homog} also ensures that $R^{\varepsilon}$ strongly converges to 0 in $L^2(\Omega)$. 

For every $\varepsilon>0$, a simple calculation shows that $R^{\varepsilon}$ is solution in $\mathcal{D}'(\Omega)$ to
\begin{align*}
    -\Delta R^{\varepsilon} + \dfrac{1}{\varepsilon}V(./\varepsilon) R^{\varepsilon} + \nu R^{\varepsilon} & = \left( -\mathcal{M}- W_{\varepsilon,\Omega}(./\varepsilon)V(./\varepsilon) \right) u^*\\
    &+ 2\nabla W_{\varepsilon,\Omega}(./\varepsilon). \nabla u^* + \varepsilon W_{\varepsilon,\Omega}(./\varepsilon) \left(-\mathcal{M}u^* - f\right).
\end{align*}
Since $R^{\varepsilon}$ belongs to $H^1_0(\Omega)$, we have $a(R^{\varepsilon}, R^{\varepsilon}) = I^{\varepsilon}= I_1^{\varepsilon} + I_2^{\varepsilon} + I_3^{\varepsilon}$ where
\begin{align}
    I_1^{\varepsilon} & = \int_{\Omega} \left( -\mathcal{M}- W_{\varepsilon,\Omega}(./\varepsilon)V(./\varepsilon) \right) u^* R^{\varepsilon},\\
    I_2^{\varepsilon} & = \int_{\Omega} 2\nabla W_{\varepsilon,\Omega}(./\varepsilon). \nabla u^* R^{\varepsilon}, \\
    I_3^{\varepsilon} & =  \int_{\Omega} \varepsilon W_{\varepsilon,\Omega}(./\varepsilon) \left(-\mathcal{M}u^* - f\right) R^{\varepsilon}.
\end{align}
We next show that the three integrals $I^{\varepsilon}_1$, $I^{\varepsilon}_2$ and $I^{\varepsilon}_3$ converge to 0 when $\varepsilon \to 0$. 

Since $u^*$ and $R^{\varepsilon}$ both belong to $H^1_0(\Omega)$, the function $u^* R^{\varepsilon}$ belong to $W^{1,1}_0(\Omega)$ and since $\Delta W_{\varepsilon, \Omega}(./\varepsilon) = V(./\varepsilon)$, an integration by parts shows that
$$I^{\varepsilon}_1 = \int_{\Omega} \left( -\mathcal{M} + |\nabla W_{\varepsilon,\Omega}(./\varepsilon)|^2 \right) u^*R^{\varepsilon} + \int_{\Omega} \varepsilon W_{\varepsilon, \Omega}(./\varepsilon) \nabla W_{\varepsilon, \Omega}(./\varepsilon).\nabla (u^* R^{\varepsilon}).$$
Lemma \ref{Lemme_convergence_Hminus1} next gives the existence of $\psi^{\varepsilon} \in L^1_{loc}(\Omega)$ such that $\Delta \psi^{\varepsilon} = -\mathcal{M} + |\nabla W_{\varepsilon,\Omega}(./\varepsilon)|^2$ in $\mathcal{D}'(\Omega)$ and $\displaystyle \lim_{\varepsilon \to 0} \|\nabla \psi^{\varepsilon}\|_{L^{\infty}(\Omega)} = 0$. Using the boundedness of $R^{\varepsilon}$ in $H^1(\Omega)$, we therefore have
\begin{align*}
    \left|\int_{\Omega} \left( -\mathcal{M} + |\nabla W_{\varepsilon,\Omega}(./\varepsilon)|^2 \right) u^*R^{\varepsilon}\right| &=  \left|\int_{\Omega} \Delta \psi^{\varepsilon} u^*R^{\varepsilon}\right|  \\
    &= \left|\int_{\Omega} \nabla \psi^{\varepsilon}.\nabla u^* R^{\varepsilon} + \int_{\Omega} \nabla \psi^{\varepsilon}.\nabla R^{\varepsilon}  u^*\right| \\
    & \leq 2 \|\nabla \psi^{\varepsilon}\|_{L^{\infty}(\Omega)} \|u^*\|_{H^1(\Omega)}\|R^{\varepsilon}\|_{H^1(\Omega)} \stackrel{\varepsilon \to 0}{\longrightarrow} 0.
\end{align*}
On the other hand, we have 
\begin{align}
    \left| \int_{\Omega} \varepsilon W_{\varepsilon, \Omega}(./\varepsilon) \nabla W_{\varepsilon, \Omega}(./\varepsilon).\nabla (u^* R^{\varepsilon}) \right| \leq \|\varepsilon W_{\varepsilon, \Omega}(./\varepsilon)\|_{L^{\infty}(\Omega)}  \int_{\Omega} |\nabla W_{\varepsilon, \Omega}(./\varepsilon)|\left(|u^*||\nabla R^{\varepsilon}| + |\nabla u^*| |R^{\varepsilon}|\right).
\end{align}
The H\"older inequality and a Sobolev embedding give the existence of $C_1>0$ independent of $\varepsilon$ such that 
\begin{align*}
    \int_{\Omega} |u^*||\nabla W_{\varepsilon, \Omega}(./\varepsilon)||\nabla R^{\varepsilon}| & \leq \left\{ \begin{array}{cc}
    \|u^*\|_{L^{4}(\Omega)} \|\nabla W_{\varepsilon, \Omega}(./\varepsilon)\|_{L^4(\Omega)} \|\nabla R^{\varepsilon}\|_{L^2(\Omega)} & \text{if } d=2, \\
     \|u^*\|_{L^{\frac{2d}{d-2}}(\Omega)} \|\nabla W_{\varepsilon, \Omega}(./\varepsilon)\|_{L^d(\Omega)}  \|\nabla R^{\varepsilon}\|_{L^2(\Omega)}  & \text{if } d \geq 3,
\end{array}
\right. \\
& \leq C_1 \left\{ \begin{array}{cc}
    \|u^*\|_{H^1(\Omega)} \|\nabla W_{\varepsilon, \Omega}(./\varepsilon)\|_{L^4(\Omega)} \|\nabla R^{\varepsilon}\|_{L^2(\Omega)} & \text{if } d=2, \\
    \|u^*\|_{H^1(\Omega)}\|\nabla W_{\varepsilon, \Omega}(./\varepsilon)\|_{L^d(\Omega)}  \|\nabla R^{\varepsilon}\|_{L^2(\Omega)}  & \text{if } d \geq 3.
\end{array}
\right.
\end{align*}
We have similarly : 
\begin{align*}
    \int_{\Omega} |\nabla u^*||\nabla W_{\varepsilon, \Omega}(./\varepsilon)||R^{\varepsilon}| 
& \leq C_1 \left\{ \begin{array}{cc}
   \|u^*\|_{H^1(\Omega)}\|\nabla W_{\varepsilon, \Omega}(./\varepsilon)\|_{L^4(\Omega)}\| R^{\varepsilon}\|_{H^1(\Omega)}   & \text{if } d=2, \\
     \|u^*\|_{H^1(\Omega)} \|\nabla W_{\varepsilon, \Omega}(./\varepsilon)\|_{L^d(\Omega)}  \|R^{\varepsilon}\|_{H^1(\Omega)} & \text{if } d \geq 3.
\end{array}
\right.
\end{align*}
Since $R^{\varepsilon}$ is bounded in $H^1(\Omega)$ and $\nabla W_{\varepsilon, \Omega}(./\varepsilon)$ is bounded in $L^p(\Omega)$ for every $p\geq 1$, we deduce the existence of $C>0$ such that
\begin{align*}
    \left|\int_{\Omega} \varepsilon W_{\varepsilon,\Omega}(./\varepsilon) \nabla W_{\varepsilon,\Omega}(./\varepsilon).\nabla (u^* R^{\varepsilon}) \right| \leq C \|\varepsilon W_{\varepsilon, \Omega}(./\varepsilon)\|_{L^{\infty}(\Omega)} \stackrel{\varepsilon \to 0}{\longrightarrow} 0,
\end{align*}

and we have proved that  $\displaystyle \lim_{\varepsilon \to 0} I^{\varepsilon}_1 =0$.

We next study the behavior of $I^{\varepsilon}_2$. We know that $f \in L^2(\Omega)$ and $u^*$ is solution to \eqref{equation_homog_schrodinger}. The elliptic regularity therefore ensures that $u^* \in H^2(\Omega)$. Since $R^{\varepsilon} \in H^1_0(\Omega)$, an integration by parts shows
\begin{align*}
    \left|I^{\varepsilon}_2\right| &= 2 \left|\int_{\Omega} \varepsilon W_{\varepsilon,\Omega}(./\varepsilon) \left(\Delta u^* R^{\varepsilon} + \nabla u^*.\nabla R^{\varepsilon}\right)\right|\\
    &\leq 4 \|\varepsilon W_{\varepsilon, \Omega}(./\varepsilon)\|_{L^{\infty}(\Omega)} \|\Delta u^*\|_{L^2(\Omega)}\|R^{\varepsilon}\|_{H^1(\Omega)}.
\end{align*}
Since $\varepsilon W_{\varepsilon, \Omega}(./\varepsilon)$ uniformly converges to 0 in $\Omega$ and $R^{\varepsilon}$ is bounded in $H^1(\Omega)$, we deduce that $I^{\varepsilon}_2$ converges to 0 when $\varepsilon \to 0$. 

Similar arguments for $I_3^{\varepsilon}$ give : 
$$|I^{\varepsilon}_3| \leq \|\varepsilon W_{\varepsilon, \Omega}(./\varepsilon)\|_{L^{\infty}(\Omega)} \|-\mathcal{M}u^* - f \|_{L^2(\Omega)}\|R^{\varepsilon}\|_{L^2(\Omega)} \stackrel{\varepsilon \to 0}{\longrightarrow} 0.$$

We finally conclude that $I^{\varepsilon}$ converges to 0 when $\varepsilon$ converges to 0. The uniform coercivity of $a^{\varepsilon}$ in $H^1_0(\Omega)$ next yields a constant $C>0$ independent of $\varepsilon$ such that
$$C\|R^{\varepsilon}\|_{H^1(\Omega)}^2 \leq a^{\varepsilon}(R^{\varepsilon},R^{\varepsilon})= I^{\varepsilon}.$$
We deduce that $R^{\varepsilon}$ strongly converges to 0 in $H^1(\Omega)$ and we can conclude.
\end{proof}

\subsection{Proof of Theorem \ref{schro_theorem_3}}

We are now in position to study the convergence of \emph{all} the eigenvalues of the operator $-\Delta + \dfrac{1}{\varepsilon} V(./\varepsilon)$ with homogeneous Dirichlet boundary conditions on $\Omega$. This result is established in the following proposition~:

\begin{prop}
\label{prop_convergence_autre_vp}
Let $\lambda_l^{\varepsilon}$ and $\mu_l$ be respectively the lower $l^{th}$ eigenvalue (counting multiplicities) of $-\Delta + \dfrac{1}{\varepsilon}V(./\varepsilon)$ and $-\Delta $  on $\Omega$ with homogeneous Dirichlet boundary conditions. Then, under the assumptions of Theorem \ref{theorem1_schro},  
\begin{equation}
    \lim_{\varepsilon \to 0} \lambda_l^{\varepsilon} = \mu_l - \mathcal{M}. 
\end{equation}
\end{prop}

\begin{proof}
Our approach is an adaptation of the method introduced in \cite[Section 3]{kenig2013estimates} and used in \cite[Section 4]{zhang2021estimates} for the periodic setting.  We fix $\kappa \in \mathbb{R}^*$ such that
\begin{equation}
\label{condition_nu}
   \mu_1 - \mathcal{M} + \kappa >0.
\end{equation}
For $f\in L^2(\Omega)$ and $\varepsilon>0$, we consider $u^{\varepsilon,\kappa}$, solution to 
\begin{equation}
\label{eq_schro_nu_eps_inter}
\left\{
\begin{array}{cc}
    -\Delta u^{\varepsilon, \kappa} + \dfrac{1}{\varepsilon} V(./\varepsilon) u^{\varepsilon, \kappa} + \kappa u^{\varepsilon,\kappa} = f & \text{on } \Omega, \\
     u^{\varepsilon, \kappa } = 0& \text{on } \partial \Omega.
\end{array}
    \right.
\end{equation}
Given \eqref{condition_nu}, Proposition \ref{prop_convergence_vp} ensures that $u^{\varepsilon,\kappa}$ is well defined when $\varepsilon$ is sufficiently small and Proposition~\ref{prop_schro_homog} shows that $u^{\varepsilon, \kappa}$ strongly converges in $L^2(\Omega)$ to $u^{*,\kappa}$, solution in $H^1(\Omega)$ to 
\begin{equation}
\label{eq_schro_nu_homog_inter}
\left\{
\begin{array}{cc}
    -\Delta u^{*, \kappa} - \mathcal{M}  u^{*, \kappa} + \nu u^{*,\kappa} = f & \text{on } \Omega, \\
     u^{*, \kappa } = 0& \text{on } \partial \Omega.
\end{array}
    \right.
\end{equation}
In addition, for $l\in \mathbb{N}^*$, we remark that the eigenvalues $\lambda^{\varepsilon,\kappa}_l$ and $\lambda^{*,\kappa}$, respectively defined as the lower $l^{th}$ eigenvalue of $-\Delta + \dfrac{1}{\varepsilon} V(./\varepsilon) + \kappa $ and $-\Delta - \mathcal{M}   + \kappa $ on $\Omega$ with homogeneous Dirichlet boundary conditions, are given by $\lambda_l^{\varepsilon, \kappa} = \lambda_l^{\varepsilon} + \kappa$ and $\lambda_l^{*, \kappa} = \mu_l - \mathcal{M}   + \kappa$. For $f \in L^2(\Omega)$, we next denote by $T^{\varepsilon,\kappa}(f)$ and $T^{*,\kappa}(f)$, respectively the unique solution in $H^1_0(\Omega)$ to \eqref{eq_schro_nu_eps_inter} and the unique solution in $H^1_0(\Omega)$ to \eqref{eq_schro_nu_homog_inter}. We also denote by $\left(v^{\varepsilon, \kappa}_l\right)_{l \in \mathbb{N}^*}$ and $\left(v^{*,\kappa}_l\right)_{l \in \mathbb{N}^*}$ two Hilbert bases of $L^2(\Omega)$ composed of eigenfunctions respectively associated with $\left(\lambda^{\varepsilon,\kappa}_l\right)_{\mathbb{N}^*}$ and $\left(\lambda^{*,\kappa}_l\right)_{\mathbb{N}^*}$. An adaptation of the results of \cite[Lemma 4.1]{zhang2021estimates}, which uses the method of \cite[Lemma 3.1]{kenig2013estimates}, next shows that
\begin{equation}
\label{majoration_schro_vp_opertor}
    \left|\dfrac{1}{\lambda_l^{\varepsilon, \kappa}} -  \dfrac{1}{\lambda_l^{*, \kappa} }\right| \leq \max \left\{  \max_{\substack{f \perp V^{*}_{l-1} \\ \|f\|_{L^2(\Omega)}= 1 }}  \left| \langle (T^{\varepsilon, \kappa} - T^{*,\kappa})f , f \rangle \right| ,  \max_{\substack{f \perp V^{\varepsilon}_{l-1} \\ \|f\|_{L^2(\Omega)}= 1 }} \left| \langle (T^{\varepsilon, \kappa} - T^{*,\kappa})f , f \rangle \right| \right\},
\end{equation}
where $V^{\varepsilon}_{l-1} = Span(v^{\varepsilon,\kappa}_1,..., v^{\varepsilon,\kappa}_{l-1})$ and $V^*_{l-1} = Span(v^{*,\kappa}_1,..., v^{*,\kappa}_{l-1})$. We note this inequality is actually established in \cite[Lemma 4.1]{zhang2021estimates} and \cite[Lemma 3.1]{kenig2013estimates} for a periodic setting. However, the periodicity of the coefficients is only used to ensure the existence of an homogenized operator $T^*$. Knowing the existence of this homogenized operator, the proof of \eqref{majoration_schro_vp_opertor} can be easily generalized in our setting since it is only based on a consequence of the min-max principle which uses the fact that both $T^{\varepsilon, \kappa}$ and $T^{*,\kappa}$ are self-adjoint and compact (ensured by the assumption $V \in L^{\infty}(\mathbb{R}^d)$).

From \eqref{majoration_schro_vp_opertor}, it therefore follows
\begin{equation}
    \left|\dfrac{1}{\lambda_l^{\varepsilon, \kappa}} -  \dfrac{1}{\lambda_l^{*, \kappa} }\right| \leq \left\|T^{\varepsilon,\kappa} - T^{*, \kappa} \right\|_{\mathcal{L}(L^2(\Omega), L^2(\Omega))}.
\end{equation}
Since $\displaystyle \lim_{\varepsilon \to 0} \left\|T^{\varepsilon,\kappa} - T^{*, \kappa} \right\|_{\mathcal{L}(L^2(\Omega), L^2(\Omega))} =0$ as a consequence of Proposition \ref{prop_schro_homog}, we deduce that $\dfrac{1}{\lambda_l^{\varepsilon, \kappa}}$ converges to $\dfrac{1}{\lambda_l^{*, \kappa} }$ when $\varepsilon \to 0$. We can therefore conclude that $\lambda_l^{\varepsilon}$ converges to $\mu_l - \mathcal{M}$ for every $l \in \mathbb{N}^*$.
\end{proof}

We next turn to the proof of Theorem \ref{schro_theorem_3}.

\begin{proof}[Proof of Theorem \ref{schro_theorem_3}]
Given the assumption $\mu_l-\mathcal{M} + \nu \neq 0$ for every $l \in \mathbb{N}^*$, Proposition \ref{prop_convergence_autre_vp} first implies that all the eigenvalues of the operator $-\Delta + \dfrac{1}{\varepsilon} V(./\varepsilon) + \nu$ are isolated from zero when $\varepsilon$ is sufficiently small. The well-posedness of \eqref{homog_eps_schro} is therefore a consequence of the Fredholm alternative. 
We next claim there exists a constant $C>0$ independent of $\varepsilon$ such that
\begin{equation}
\label{bound_H1_schro_Th3}
    \|u^{\varepsilon}\|_{H^1(\Omega)} \leq C \|f\|_{L^2(\Omega)},
\end{equation}
when $\varepsilon$ is sufficiently small. We indeed denote by $(v^{\varepsilon}_l)_{l \in \mathbb{N}^*} \in H^1_0(\Omega)^{\mathbb{N}}$ an Hilbert basis of $L^2(\Omega)$ composed of the eigenfunctions associated with the sequence $(\lambda_l^{\varepsilon}+ \nu)_{\varepsilon>0}$ of the eigenvalues of $-\Delta + \dfrac{1}{\varepsilon} V(./\varepsilon) + \nu$. We also define $E^{\varepsilon}_+ = \operatorname{Span}\left( v^{\varepsilon}_l, \ \lambda_l^{\varepsilon}+ \nu>0 \right)$, $E^{\varepsilon}_{-} = \operatorname{Span}\left( v^{\varepsilon}_l, \ \lambda_l^{\varepsilon} + \nu <0 \right)$ and $\Pi_+^{\varepsilon}$, $\Pi_-^{\varepsilon}$ the orthogonal projections respectively associated with $E^{\varepsilon}_+ $ and $E^{\varepsilon}_{-}$. For every $u,v \in H^1_0(\Omega)$, we define $a^{\varepsilon}(u,v) = \displaystyle \int_{\Omega} \nabla u. \nabla v + \dfrac{1}{\varepsilon} V(./\varepsilon) uv + \nu uv$. Since $u^{\varepsilon}$ is solution to \eqref{homog_eps_schro} in $H^1_0(\Omega)$, we have 
\begin{equation}
  a^{\varepsilon}(\Pi_+^{\varepsilon}(u^{\varepsilon}), \Pi_+^{\varepsilon}(u^{\varepsilon})) = a(u^{\varepsilon}, \Pi_+^{\varepsilon}(u^{\varepsilon})) = \int_{\Omega} f \Pi_+^{\varepsilon}(u^{\varepsilon}) = \int_{\Omega} \Pi_+^{\varepsilon}(f) \Pi_+^{\varepsilon}(u^{\varepsilon}).
\end{equation}
We next remark that 
\begin{equation}
    a^{\varepsilon}(\Pi_+^{\varepsilon}(u^{\varepsilon}), \Pi_+^{\varepsilon}(u^{\varepsilon})) \geq \min_{\lambda^{\varepsilon}_l +\nu >0} \left\{\lambda^{\varepsilon}_l+ \nu \right\} \|\Pi_+^{\varepsilon}(u^{\varepsilon})\|_{L^2(\Omega)}^2.
\end{equation}
We also know that $\displaystyle \min_{\lambda^{\varepsilon}_l +\nu >0}\left\{\lambda^{\varepsilon}_l+ \nu \right\} > \dfrac{\displaystyle \min_{\mu_l - \mathcal{M} +\nu >0}\left\{\mu_l - \mathcal{M} +\nu \right\} }{2}>0$ as a consequence of Proposition \ref{prop_convergence_autre_vp}. Exactly as in the proof of \eqref{coercivite_unif_schro} established in Section \ref{schro_scetion_wellposedness}, we can therefore show the coercivity of $a^{\varepsilon}$ in $H^1_0(\Omega) \cap E^{\varepsilon}_+$ and establish the existence of $C_+>0$, independent of $\varepsilon$, such that 
\begin{equation}
     a^{\varepsilon}(\Pi_+^{\varepsilon}(u^{\varepsilon}), \Pi_+^{\varepsilon}(u^{\varepsilon})) \geq C_+ \|\Pi_+^{\varepsilon}(u^{\varepsilon})\|_{H^1(\Omega)}^2.
\end{equation}
Moreover, the Cauchy-Schwarz inequality gives
\begin{equation}
    \left|\int_{\Omega} \Pi_+^{\varepsilon}(f) \Pi_+^{\varepsilon}(u^{\varepsilon})\right| \leq \|\Pi_+^{\varepsilon}(f)\|_{L^2(\Omega)}\|\Pi_+^{\varepsilon}(u^{\varepsilon})\|_{H^1(\Omega)} \leq  \|f\|_{L^2(\Omega)} \|\Pi_+^{\varepsilon}(u^{\varepsilon})\|_{H^1(\Omega)},
\end{equation}
and we obtain 
\begin{equation}
\label{Schro_born_h1_pip}
    \|\Pi_+^{\varepsilon}(u^{\varepsilon})\|_{H^1(\Omega)} \leq \dfrac{1}{C_+} \|f\|_{L^2(\Omega)}.
\end{equation}
On the other hand, we have
\begin{equation}
  - a^{\varepsilon}(\Pi_-^{\varepsilon}(u^{\varepsilon}), \Pi_-^{\varepsilon}(u^{\varepsilon})) = -a^{\varepsilon}(u^{\varepsilon}, \Pi_-^{\varepsilon}(u^{\varepsilon})) = - \int_{\Omega} f \Pi_-^{\varepsilon}(u^{\varepsilon}) = -\int_{\Omega} \Pi_-^{\varepsilon}(f) \Pi_-^{\varepsilon}(u^{\varepsilon}),
\end{equation}
and 
\begin{equation}
    - a^{\varepsilon}(\Pi_-^{\varepsilon}(u^{\varepsilon}), \Pi_-^{\varepsilon}(u^{\varepsilon})) \geq \min_{\lambda_l^{\varepsilon} + \nu <0} \left\{|\lambda^{\varepsilon}_l+ \nu |\right\}\|\Pi_-^{\varepsilon}(u^{\varepsilon})\|_{L^2(\Omega)}^2.
\end{equation}
As above, we can also deduce the existence of $C_->0$ such that 
\begin{equation}
\label{Schro_born_h1_pim}
    \|\Pi_-^{\varepsilon}(u^{\varepsilon})\|_{H^1(\Omega)} \leq \dfrac{1}{C_-} \|f\|_{L^2(\Omega)}.
\end{equation}
We next denote $M = \max\left(\dfrac{1}{C_+} , \dfrac{1}{C_-}\right)$. Since $\lambda_l^{\varepsilon} + \nu  \neq 0$ for every $l\in \mathbb{N}$ and for $\varepsilon$ sufficiently small, we have $\|u^{\varepsilon}\|_{H^1(\Omega)} \leq  \|\Pi_+^{\varepsilon}(u^{\varepsilon})\|_{H^1(\Omega)} + \|\Pi_-^{\varepsilon}(u^{\varepsilon})\|_{H^1(\Omega)}$.  Using \eqref{Schro_born_h1_pip} and \eqref{Schro_born_h1_pim}, we obtain 
\begin{equation}
    \|u^{\varepsilon}\|_{H^1(\Omega)} \leq 2M \|f\|_{L^2(\Omega)}, 
\end{equation}
which yields \eqref{bound_H1_schro_Th3}.
Therefore, up to an extraction, $u^{\varepsilon}$ converges strongly in $L^2(\Omega)$ and weakly in $H^1(\Omega)$. Repeating step by step the proof of Proposition \ref{prop_schro_homog}, we can then conclude that the sequence $u^{\varepsilon}$ converges, strongly in $L^2(\Omega)$ and weakly in $H^1(\Omega)$, to $u^*$ solution to \eqref{equation_homog_schrodinger}. If we next define $R^{\varepsilon}$ by \eqref{schro_def_R}, $R^{\varepsilon}$ belongs to $H^1_0(\Omega)$, is clearly bounded in $H^1(\Omega)$ uniformly with respect to $\varepsilon$ and converges to 0 in $L^2(\Omega)$ as $\varepsilon$ vanishes. Moreover, a calculation shows that 
\begin{align*}
    -\Delta R^{\varepsilon} + \dfrac{1}{\varepsilon}V(./\varepsilon) R^{\varepsilon} + \nu R^{\varepsilon} & = F^{\varepsilon},
\end{align*}
where 
\begin{equation}
    F^{\varepsilon} = \left( -\mathcal{M}- W_{\varepsilon,\Omega}(./\varepsilon)V(./\varepsilon) \right) u^*
    + 2\nabla W_{\varepsilon,\Omega}(./\varepsilon). \nabla u^* + \varepsilon W_{\varepsilon,\Omega}(./\varepsilon) \left(-\mathcal{M}u^* - f\right).
\end{equation}
Using that $a^{\varepsilon}(R^{\varepsilon},v)= \displaystyle \int_{\Omega} F^{\varepsilon} v$ for every $v \in H^1_0(\Omega)$, we can also show that (see the above proofs of estimates \eqref{bound_H1_schro_Th3}, \eqref{Schro_born_h1_pip} and \eqref{Schro_born_h1_pim}) :
\begin{equation}
    \dfrac{1}{C^2}\|R^{\varepsilon}\|^2_{H^1(\Omega)} \leq \left|\int_{\Omega} F^{\varepsilon} \Pi^{\varepsilon}_{-} (R^{\varepsilon}) \right| + \left|\int_{\Omega} F^{\varepsilon} \Pi^{\varepsilon}_{+} (R^{\varepsilon})\right|. 
\end{equation}
We finally conclude that both $\displaystyle \left|\int_{\Omega} F^{\varepsilon} \Pi^{\varepsilon}_{-} (R^{\varepsilon}) \right|$ and $\displaystyle \left|\int_{\Omega} F^{\varepsilon} \Pi^{\varepsilon}_{+} (R^{\varepsilon})\right|$ converge to 0 as $\varepsilon$ vanishes proceeding exactly as in the proof of Theorem \ref{schro_theorem_2}. 

\end{proof}

We conclude this article with a discussion regarding the possibility to extend our homogenization results to a larger class of non periodic potentials $V$. To this end, we suggest two possible complementary approaches : 

\textbf{1)} \underline{Extension by density arguments}. We could adapt all of our proofs and establish results similar to those of Theorems \ref{theorem1_schro}, \ref{schro_theorem_2} and \ref{schro_theorem_3} considering a potential of the form \eqref{potential_form_schro} when $\varphi$ is no longer compactly supported but decays sufficiently fast at infinity. This is for instance expressed by $|\varphi(x)| \leq \dfrac{C}{1+|x|^{\alpha}}$ for $\alpha>d$. We indeed remark that a simple adaptation of Lemma \ref{Schro_lemme_densite} shows that such a potential is a limit in $L^{\infty}(\mathbb{R}^d)$ of functions of the class \eqref{potential_form_schro} that we study in this article. In addition, the proof of Theorem \ref{theorem1_schro} is based on the continuity from $L^{\infty}(\mathbb{R}^d)$ to $BMO(\mathbb{R}^d)$ of $T: f \rightarrow \nabla ^2 G \ast f$ (see the proof of Proposition~\ref{prop_correcteur_cas_lin_schro}) and we could easily show that all of our results regarding the corrector equation (the weak-convergence properties of the gradient of our sequence of correctors $W_{\varepsilon,R}$ in particular) could therefore be extended by density arguments.
Consequently, since the homogenization results of the present contribution are established only using the specific properties of the adapted corrector and the fact that $V \in L^{\infty}(\mathbb{R}^d)$, they could be naturally extended to this setting.

\textbf{2)} \underline{Extension by algebraic operations}. The homogenization results could be also established considering a potential of the form, say, $ \displaystyle V = \sum_{k,l \in \mathbb{Z}^d} \varphi(.-k-Z_k)\psi(.-l-Z_l)$ obtained by multiplying two potentials of the particular class \eqref{potential_form_schro} that we have studied. For this new setting, although our approach based on the Taylor expansion of $V$ would still be possible to solve the corrector equation, Assumptions \eqref{A0} to \eqref{A4} would no longer be sufficient to establish the existence of an adapted corrector $w$ since the convergence of $|\nabla w(./\varepsilon)|^2$ would involve the correlations of the sequence $Z$ up to the fourth order. Adapting \eqref{A00},\eqref{A2},\eqref{A3} and \eqref{A4} to the fourth order correlations of $Z$, the method introduced in the present article would again allow to conclude. Iterating this argument and under suitable assumptions for the correlations of $Z$ of any order, our homogenization results could therefore be extended to the whole algebra generated by the potential of the form \eqref{potential_form_schro}.

\section*{Acknowledgements}
The research of the second author is partially supported by ONR under Grant N00014-20-1-2691 and by EOARD under Grant FA8655-20-1-7043.

\end{document}